\documentclass[11 pt, reqno]{amsart}
\usepackage{amsfonts}
\usepackage{amssymb, amsmath, amsthm, enumitem, mathtools, extpfeil, stmaryrd}
\usepackage[colorlinks=true, linkcolor = blue, citecolor= Green]{hyperref}
\usepackage[alphabetic,lite]{amsrefs}
\usepackage{amscd}   % for commutative diagrams
\usepackage{fullpage}% !TEX root =  
\usepackage[all]{xy} % for complicated commutative diagrams
\usepackage{faktor}
\usepackage[utf8x]{inputenc}
\usepackage{verbatim}
\usepackage{mathrsfs}
\usepackage{epstopdf}
\usepackage[normalem]{ulem}
\usepackage{setspace}
\usepackage{stmaryrd}
\DeclareMathAlphabet{\mathpzc}{OT1}{pzc}{m}{it}
\DeclareFontEncoding{OT2}{}{} % to enable usage of cyrillic fonts
\newcommand{\textcyr}[1]{%
 {\fontencoding{OT2}\fontfamily{wncyr}\fontseries{m}\fontshape{n}\selectfont #1}}
\newcommand{\Sha}{{\mbox{\textcyr{Sh}}}}
\usepackage[usenames,dvipsnames]{color}
\newcommand{\yuan}[1]{{\color{Orange} \sf $\clubsuit\clubsuit\clubsuit$ Yuan: [#1]}}

\makeatletter
\def\@tocline#1#2#3#4#5#6#7{\relax
  \ifnum #1>\c@tocdepth % then omit
  \else
    \par \addpenalty\@secpenalty\addvspace{#2}%
    \begingroup \hyphenpenalty\@M
    \@ifempty{#4}{%
      \@tempdima\csname r@tocindent\number#1\endcsname\relax
    }{%
      \@tempdima#4\relax
    }%
    \parindent\z@ \leftskip#3\relax \advance\leftskip\@tempdima\relax
    \rightskip\@pnumwidth plus4em \parfillskip-\@pnumwidth
    #5\leavevmode\hskip-\@tempdima
      \ifcase #1
       \or\or \hskip 1em \or \hskip 2em \else \hskip 3em \fi%
      #6\nobreak\relax
    \hfill\hbox to\@pnumwidth{\@tocpagenum{#7}}\par% <---- \dotfill -> \hfill
    \nobreak
    \endgroup
  \fi}
\makeatother

\newtheorem{lemma}{Lemma}[section]
\newtheorem{theorem}[lemma]{Theorem}
\newtheorem{proposition}[lemma]{Proposition}
\newtheorem{corollary}[lemma]{Corollary}
\newtheorem{conjecture}[lemma]{Conjecture}

\newtheorem{claim*}{Claim}
\newtheorem{definition}[lemma]{Definition}

\newtheorem{notation}[lemma]{Notation}

\theoremstyle{definition}
\newtheorem{remark}[lemma]{Remark}

\newcommand{\PP}{{\mathbb P}}

\newcommand{\F}{{\mathbb F}}
\newcommand{\Q}{{\mathbb Q}}
\newcommand{\R}{{\mathbb R}}
\newcommand{\Z}{{\mathbb Z}}
\newcommand{\NN}{{\mathbb N}}

\newcommand{\calA}{{\mathcal A}}
\newcommand{\calB}{{\mathcal B}}
\newcommand{\calC}{{\mathcal C}}

\newcommand{\calE}{{\mathcal E}}
\newcommand{\calF}{{\mathcal F}}
\newcommand{\calG}{{\mathcal G}}

\newcommand{\calM}{{\mathcal M}}
\newcommand{\calN}{{\mathcal N}}
\newcommand{\calO}{{\mathcal O}}
\newcommand{\calP}{{\mathcal P}}

\newcommand{\calT}{{\mathcal T}}

\newcommand{\calZ}{{\mathcal Z}}

\newcommand{\fraka}{{\mathfrak a}}

\newcommand{\frakc}{{\mathfrak c}}

\newcommand{\frakm}{{\mathfrak m}}

\newcommand{\frakp}{{\mathfrak p}}
\newcommand{\frakq}{{\mathfrak q}}

\newcommand{\frakx}{{\mathfrak x}}
\newcommand{\fraky}{{\mathfrak y}}

\newcommand{\frakA}{{\mathfrak A}}
\newcommand{\frakB}{{\mathfrak B}}

\newcommand{\frakP}{{\mathfrak P}}

\newcommand{\frakS}{{\mathfrak S}}

\usepackage[mathscr]{euscript}

\newcommand{\scrR}{{\mathscr R}}
\newcommand{\scrS}{{\mathscr S}}

\DeclareMathOperator{\tr}{tr}

\DeclareMathOperator{\Frob}{Frob}
\DeclareMathOperator{\coker}{coker}
\DeclareMathOperator{\rk}{rk}
\DeclareMathOperator{\Char}{char}

\DeclareMathOperator{\im}{im}

\DeclareMathOperator{\End}{End}

\DeclareMathOperator{\Hom}{Hom}
\DeclareMathOperator{\Ext}{Ext}

\DeclareMathOperator{\Aut}{Aut}
\DeclareMathOperator{\Gal}{Gal}
\DeclareMathOperator{\Ind}{Ind}

\DeclareMathOperator{\Cl}{Cl}

\DeclareMathOperator{\Spec}{Spec}

\DeclareMathOperator{\res}{res}
\DeclareMathOperator{\cor}{cor}

\DeclareMathOperator{\Disc}{Disc}

\DeclareMathOperator{\id}{id}
\DeclareMathOperator{\val}{val}

\DeclareMathOperator{\B}{{\mbox{\textcyr{B}}}}

\DeclareMathOperator{\Sur}{Sur}

\DeclareMathOperator{\Hur}{Hur}

\DeclareMathOperator{\nr}{nr}
\DeclareMathOperator{\rDisc}{rDisc}

\DeclareMathOperator{\Nm}{Nm}
\DeclareMathOperator{\Idem}{Idem}
\DeclareMathOperator{\ab}{ab}
\DeclareMathOperator{\Ram}{Ram}
\newcommand{\hZ}{\hat{\mathbb{Z}}}

\numberwithin{equation}{section}
\numberwithin{table}{section}

\usepackage{tikz-cd}

\makeatletter
\newcommand{\mainsectionstyle}{%
  \renewcommand{\@secnumfont}{\bfseries}
  \renewcommand\section{\@startsection{section}{2}%
    \z@{.5\linespacing\@plus.7\linespacing}{-.5em}%
    {\normalfont\bfseries}}%
}
\makeatother

\title{On the Distribution of Class Groups of Abelian Extensions}
\author{Yuan Liu}
\address{Department of Mathematics\\
University of Illinois at Urbana-Champaign \\ 1409 W Green St \\
Urbana, IL 61801 USA}  
\email{yyyliu@illinois.edu}

\begin{document}

	%%%%%%%%%%%%%%%%%%%%%%%%%%%%%%%%%%%%%%%%%%%%%%%%%%%%%%%%%%%%%%%%%%%%%%%%%%%%
	 \begin{abstract}
	 	Given a finite abelian group $\Gamma$, we study the distribution of the $p$-part of the class group $\Cl(K)$ as $K$ varies over Galois extensions of $\Q$ or $\F_q(t)$ with Galois group isomorphic to $\Gamma$. We first construct a discrete valuation ring $e\Z_p[\Gamma]$ for each primitive idempotent $e$ of $\Q_p[\Gamma]$, such that 1) $e\Z_p[\Gamma]$ is a lattice of the irreducible $\Q_p[\Gamma]$-module $e\Q_p[\Gamma]$, and 2) $e\Z_p[\Gamma]$ is naturally a quotient of $\Z_p[\Gamma]$. For every $e$, we study the distribution of $e\Cl(K):=e\Z_p[\Gamma] \otimes_{\Z_p[\Gamma]} \Cl(K)[p^{\infty}]$, and prove that there is an ideal $I_e$ of $e\Z_p[\Gamma]$ such that $e\Cl(K) \otimes (e\Z_p[\Gamma]/I_e)$ is too large to have finite moments, while $I_e \cdot e\Cl(K)$ should be equidistributed with respect to a Cohen--Lenstra type of probability measure. We give conjectures for the probability and moment of the distribution of $I_e\cdot e\Cl(k)$, and prove a weighted version of the moment conjecture in the function field case. Our weighted-moment technique is designed to deal with the situation when the function field moment, obtained by counting points of Hurwitz spaces, is infinite; and we expect that this technique can also be applied to study other bad prime cases. Our conjecture agrees with the Cohen--Lenstra--Martinet conjecture when $p\nmid |\Gamma|$, and agrees with the Gerth conjecture when $\Gamma=\Z/p\Z$. We also study the kernel of $\Cl(K) \to \bigoplus_e e\Cl(K)$, and show that the average size of this kernel is infinite when $p^2\mid |\Gamma|$.
	 \end{abstract}
	%%%%%%%%%%%%%%%%%%%%%%%%%%%%%%%%%%%%%%%%%%%%%%%%%%%%%%%%%%%%%%%%%%%%%%%%%%%%

	\maketitle
	
%	\hypersetup{linkcolor=black}
%	\tableofcontents
%	\hypersetup{linkcolor=blue}

\section{Introduction}\label{sect:intro}
	
	In \cite{Cohen-Lenstra}, Cohen and Lenstra gave a conjecture that predicts the distribution of abelian $p$-groups, for an odd prime $p$, that occur as the $p$-primary part of the class group $\Cl(K)$ of a quadratic number field $K$, as the field $K$ varies. Their conjecture does not hold for the 2-primary part of $\Cl(K)$ for quadratic $K/\Q$, because by Gauss's genus theory, the 2-torsion subgroup of $\Cl(K)$ (which is isomorphic to $\Cl(K)/2\Cl(K)$) is determined by the number of primes ramified in $K/Q$, which implies that the average of $\dim_{\F_2} \Cl(K)[2]$ is infinite (while Cohen--Lenstra heuristics suggest that the average of $\dim_{\F_p}\Cl(K)[p]$ is finite when $p$ is odd). Instead of studying the whole class group, Gerth \cite{Gerth84} considered the part that is not determined by the genus theory, and conjectured that the distribution of the 2-primary part of $2\Cl(K)$ can be predicted by probability measures similar to the ones used in the Cohen--Lenstra heuristics. 
	
	In this paper, we show that the above Cohen--Lenstra--Gerth type of conjectures together with the genus theory can be extended to the family of $\Gamma$-extensions of $\Q$ for any finite abelian group $\Gamma$. Roughly speaking, for a Galois extension $K/\Q$ with $\Gal(K/\Q)\simeq \Gamma$ being abelian, we prove that there is some special quotient of $\Cl(K)$ whose rank is bounded below by the number of primes ramified in a particular way in $K/\Q$; and moreover, we conjecture that the part of $\Cl(K)$ that is not determined by the number of ramified primes should be randomly distributed in the way similar to Cohen--Lenstra, as $K$ varies over all $\Gamma$-extensions of $\Q$.
	
\subsection{Main results}
\hfill

	Throughout the paper, we let $\Gamma$ be a finite abelian group and $p$ a prime number. Let $\Cl(K)(p)$ denote the $p$-primary part of the class group $\Cl(K)$ for a number field $K$. 
	A $\Gamma$-extension of $\Q$ is a Galois extension $K/\Q$ together with a chosen isomorphism $\Gal(K/\Q)\to \Gamma$.
	For a $\Gamma$-extension $K/\Q$, the Galois group $\Gal(K/\Q)\simeq \Gamma$ naturally acts on $\Cl(K)$, so $\Cl(K)(p)$ has a $\Z_p[\Gamma]$-module structure. In order to the study the distribution of $\Cl(K)(p)$, we first need to classify all the $\Z_p[\Gamma]$-modules that could appear as $\Cl(K)(p)$.
	When $p\nmid |\Gamma|$, $\F_p[\Gamma]$ is semisimple, and $\Z_p[\Gamma]$ can be decomposed as the direct product discrete valuation rings whose residue fields are exactly the simple $\F_p[\Gamma]$-modules.
	When $\Gamma\simeq \Z/p\Z$, since the norm map annihilates $\Cl(K)$, $\Cl(K)(p)$, as a $\Z_p[\Gamma]$-module, is annihilated by $\sum_{\gamma \in \Gamma} \gamma$; so $\Cl(K)(p)$ is a module over the discrete valuation ring $\Z_p[\Gamma]/(\sum_{\gamma \in \Gamma} \gamma)$.
	In general, $\Z_p[\Gamma]/(\sum_{\gamma \in \Gamma} \gamma)$ is not a product of discrete valuation rings. 
	
	We will study $\Z_p[\Gamma]$-modules by taking tensor product along projection maps from the ring $\Z_p[\Gamma]$ to a family discrete valuation rings, where this family bijectively corresponds to the set of simple $\Q_p[\Gamma]$-modules. Explicitly, let $\calE$ denote the set of all the primitive idempotents of the ring $\Q_p[\Gamma]$. Then $e\Q_p[\Gamma]$ with $e \in \calE$ is a simple $\Q_p[\Gamma]$-modules, and conversely every simple $\Q_p[\Gamma]$-module can be written in this form. For each $e \in \calE$, we will define a $\Z_p$-lattice of $e\Q_p[\Gamma]$, denoted by $e\Z_p[\Gamma]$, which is a quotient ring of $\Z_p[\Gamma]$ and is a discrete valuation ring. Then
	\begin{equation}\label{eq:eCl-intro}
		e\Cl(K):=e\Z_p[\Gamma] \otimes_{\Z_p[\Gamma]} \Cl(K)(p)
	\end{equation}
	is a module over $e\Z_p[\Gamma]$ and is a quotient of $\Cl(K)(p)$. We will first prove an analogue of the genus theory for $e\Cl(K)$ of any $\Gamma$-extension $K/\Q$.  
	
	Let $\frakm_e$ denote the maximal ideal of $e\Z_p[\Gamma]$. Then by the classfication of modules over discrete valuation rings, $e\Cl(K)$ can be decomposed as
	\begin{equation}\label{eq:intro-eCl-decomp}
		e\Cl(K)\simeq \bigoplus_{i=1}^{\infty} (e\Z_p[\Gamma]/\frakm_e^i)^{\oplus n_i}, \quad n_i \in \Z_{\geq 0} \quad \text{and} \quad \sum_{i=1}^{\infty} n_i < \infty.
	\end{equation}
	For a nonzero proper ideal $I$ of $e\Z_p[\Gamma]$, there is a positive integer $d$ such that $I=\frakm_e^d$, and then, using notation in \eqref{eq:intro-eCl-decomp}, we define
	\[
		\rk_I e\Cl(K):=\sum_{i=d}^{\infty} n_i.
	\]
	A \emph{ramification type for $\Gamma$-extensions} is a pair $(\calG, \calT)$ such that $\calT \leq \calG \leq \Gamma$; and for a $\Gamma$-extension $K/Q$ of global fields, we say a prime $\frakp$ of $Q$ \emph{satisfies the ramification type $(\calG, \calT)$} if the decomposition subgroup and inertia subgroup of $K/Q$ at $\frakp$ are $\calG$ and $\calT$ respectively.
	
	\begin{theorem}[Special case of Theorem~\ref{thm:lb-rank}]\label{thm:main-A}
		Let $Q$ be either $\Q$ or $\F_q(t)$ with $\gcd(q, p|\Gamma|)=1$ and $K$ a $\Gamma$-extension of $Q$ and $e \in \calE$. Assume $I$ is a proper ideal of $e\Z_p[\Gamma]$, and there exists a nontrivial $\gamma \in \Gamma$ such that the $\Z_p[\Gamma]$-module $e\Z_p[\Gamma]/I$ is annihilated by both $1-\gamma$ and $\sum_{j=1}^{|\gamma|} \gamma^j$ (note that every $e\Z_p[\Gamma]$-module is naturally a $\Z_p[\Gamma]$-module via the base change $\Z_p[\Gamma] \to e\Z_p[\Gamma]$). 
		
		Then there exist
		\begin{enumerate}
			\item\label{item:main-A-1}	a nonempty family of ramification types for $\Gamma$-extensions, and
			\item a constant $c$ depending on $\Gamma$, $e$ and $Q$,
		\end{enumerate}
		such that for any $\Gamma$-extension $K/Q$, 
		\[
			\rk_I e\Cl(K) \geq \#\{\frakp \subset Q \mid \text{$\frakp$ satisfies a ramification type in \eqref{item:main-A-1} for $K/Q$} \} -c.
		\]
	\end{theorem}  
	
	For each $e \in \calE$, if $p\mid |\Gamma|$, then there is a unique smallest ideal $I$ that satisfies the assumption in Theorem~\ref{thm:main-A}, and we let $I_e$ denote that ideal $I$. 
	If $p\nmid |\Gamma|$, then there does not exist a proper ideal $I$ as described in Theorem~\ref{thm:main-A}, and we define $I_e:=e\Z_p[\Gamma]$.
	In the decomposition of $e\Cl(K)/(I\cdot e\Cl(K))$ there are exactly $\rk_{I} e\Cl(K)$ copies of $e\Z_p[\Gamma]/I$, so Theorem~\ref{thm:main-A} provides information about $e\Cl(K)/(I_e \cdot e\Cl(K))$.
	For an extension $K/Q$ of global fields, let $\rDisc K$ denote the norm of the radical of the discriminant ideal $\Disc(K/Q)$.
	For $Q=\Q$ or $\F_q(t)$, let $\calA^+_{\Gamma}(D, Q)$ be the set of isomorphism classes of totally real
	\footnote{When $Q=\F_q(t)$, an extension $K/Q$ is \emph{totally real} if it is completely split at the place $\infty$ of $\F_q(t)$.}
	 $\Gamma$-extensions of $Q$ with $\rDisc K =D$. We prove the following theorem regarding the distribution of $e\Cl(K)$ for $K \in\calA_{\Gamma}^+(D, Q)$.

	\begin{theorem}\label{thm:main-B}
		Let $e\in \calE$. 
		\begin{enumerate}
			\item \textit{(Special case of Theorem~\ref{thm:conductor})} \label{item:main-B-1} Assume $p\mid |\Gamma|$. 				
			\[
				\lim_{X \to \infty} \frac{\sum\limits_{D\leq X}\sum\limits_{K \in \calA^+_{\Gamma}(D, \Q)} \rk_{I_e} e\Cl(K)}{\sum\limits_{D\leq X}\# \calA^+_{\Gamma}(D, \Q)}= \infty.
			\]
				
			\item \label{item:main-B-2} Assume $e$ does not correspond to the trivial representation (that is, $e\neq \frac{\sum_{\gamma \in \Gamma} \gamma}{|\Gamma|}$).
		 Let $M$ be a finite $e\Z_p[\Gamma]$-module, and let $r:=\rk_{\frakm_e}M$. Define a weight function on $\Gamma$-extensions $K/\F_q(t)$ as
			\[
				w_{e,M}(K):=
				\begin{cases}
				\#\Hom_{\Gamma}\left(\Cl(K),(e\Z_p[\Gamma]/I_e)^{\oplus r}\right) & \text{ if }\Sur_{\Gamma}(\Cl(K), (e\Z_p[\Gamma]/I_e)^{\oplus r})\neq \O \\
					0 & \text{ otherwise.}
				\end{cases}
			\]
			Then 
			\footnote{See \S\ref{ss:notation} for our definition of the notation of iterated limit. }
			\begin{equation}\label{eq:ff-weight-moment}
				\lim_{N \to \infty}\lim_{\substack{q \to \infty \\ p\nmid q(q-1) \\ \gcd(q,|\Gamma|)=1}} 
				\frac{\sum\limits_{0\leq n \leq N} \sum\limits_{K \in \calA^+_{\Gamma}(q^n, \F_q(t))} w_{e, M}(K) \#\Sur_{\Gamma}(I_e \cdot e\Cl(K), M)}{\sum\limits_{0\leq n \leq N} \sum\limits_{K \in \calA^+_{\Gamma}(q^n, \F_q(t))}w_{e, M}(K)}=\frac{1}{|M|}.
			\end{equation}
		\end{enumerate}
	\end{theorem}
	
	The statement \eqref{item:main-B-1} above follows by Theorem~\ref{thm:main-A} and Theorem~\ref{tCond}. The statement \eqref{item:main-B-2} is about a weighted moment of the distribution of $I_e\cdot e\Cl(K)$ in the function field case: it says that a weighted average of $\#\Sur_{\Gamma}(I_e \cdot e\Cl(K), M)$ is $1/|M|$. Comparing that with the moment version of Cohen--Lenstra heuristics, \eqref{eq:ff-weight-moment} suggests, despite the fact that here is a weight function, the distribution of $I_e \cdot e\Cl(K)$ should be analogous to the one in the Cohen--Lenstra heuristics. For a fixed $M$, the weight function $w_{e,M}(K)$ is determined by the bad part of the class group, i.e., $e\Cl(K)/I_e\cdot e\Cl(K)$. Since Theorem~\ref{thm:main-B} shows the bad part is statistically infinite while the good part $I_e \cdot e\Cl(K)$ is statistically finite, it is reasonable to believe that the bad part and the good part are not statistically correlated, and hence applying the weight function should not change the moments.
	So we conjecture that as $K$ varies over totally real $\Gamma$-extensions of $Q=\Q$ or $\F_q(t)$, $I_e\cdot e\Cl(K)$ is distributed according to a probability measure whose $M$-moment is $1/|M|$. In this context the moments are known to determine a unique distribution, so we give both the moment version and probability version of the conjecture for the distribution of $I_e\cdot e\Cl(K)$ in Conjecture~\ref{conj:GG} for every nontrivial idempotent $e$. 
	When $e$ is the trivial primitive idempotent $e_0:=\frac{\sum_{\gamma \in \Gamma} \gamma}{|\Gamma|}$, the above moment result \eqref{eq:ff-weight-moment} does not hold: in Proposition~\ref{prop:trivial-idem}, we prove $|I_{e_0}(e_0\Cl(K))| \leq |\wedge^2 \Gamma_p|$ for any $\Gamma$-extension $K$ of $\F_q(t)$ or $\Q$,  where $\Gamma_p$ is the Sylow $p$-subgroup of $\Gamma$.
	
	In the good prime case (that is $p \nmid |\Gamma|$), it is known that the distributions of $p$-part of class group of $\Gamma$-extensions are different between the cases of $p \mid q-1$ and of $p \nmid q-1$; however, in Gerth's conjecture, the base field $\Q$ contains $\mu_2$. 
	In Theorem~\ref{thm:main-B}\eqref{item:main-B-2}, we only consider the finite fields $\F_q$ that do not contain the $p$th roots of unity, because when $p \nmid q-1$ counting points on Hurwitz spaces is easier (see \S\ref{ss:Hurwitz}). The trade-off is, when $|\Gamma|$ is even and $p=2$, \eqref{eq:ff-weight-moment} is an empty statement. 
	When $p\mid q-1$, the function field moment can still be computed by the method described in \S\ref{ss:Hurwitz}, but one needs to carefully analyze the Schur multipliers associated to the Hurwitz spaces.  
	For example, we study the case that $\Gamma=\Z/2\Z$ and $p=2$, and we show that when $q \equiv 3 \bmod 4$, the weighted moments of the distribution of $2\Cl(K)[2^{\infty}]$ agrees with the actual moment in Gerth's Conjecture (proven by Smith). When $q\equiv 1 \bmod 4$, we show that the weighted moment is different from the $q\equiv 3 \bmod 4$ case. Note that in Smith's result (see Theorem~1.12 in \cite{Smith2022}), he assumed that the base field does not contain $\mu_4$ in order to get the distribution conjectured by Gerth; so our result gives another evidence showing that assumption is necessary.
	
	When $\Gamma:=\Z/2\Z$, there is a unique nontrivial primitive idempotent $e$ of $\Q_2[\Z/2\Z]$. By definition of $I_e$, one can see that $I_e=\frakm_e=(2)$.
	For a quadratic extension $K/\F_q(t)$ with $2\nmid q$ that splits completely at $\infty$, the 2-part of class group is an $e\Z_2[\Z/2\Z]$-module. 
	
	\begin{theorem}\label{thm:Z/2Z}
		Let $M$ be a finite $e\Z_2[\Z/2\Z]$-module for the unique nontrivial primitive idempotent $e$ of $\Q_2[\Z/2\Z]$, and let $w_M(K)$ denote the weight function $w_{e,M}(K)$ defined in Theorem~\ref{thm:main-B}\eqref{item:main-B-2}. For an integer $m$, let $\val_2(m)$ denote the (additive) $2$-adic valuation of $m$. Then for any positive integer $v$,
		\[
			\lim_{N \to \infty} \lim_{\substack{q \to \infty \\ \val_2(q-1)=v}} \frac{\sum\limits_{0 \leq n \leq N} \sum\limits_{K \in \calA^+_{\Z/2\Z}(q^n, \F_q(t))} w_M(K) \#\Sur(2\Cl(K)[2^{\infty}], M)}{\sum\limits_{0 \leq n \leq N} \sum\limits_{K \in \calA^+_{\Z/2\Z}(q^n, \F_q(t))} w_M(K)} = \frac{|(\wedge^2 M) [2^{v-1}]|}{|M|}.
		\]
		In particular, when $v=1$ (i.e., $q \equiv 3 \bmod 4$), the weighted moment on the left-hand side above equals $1/|M|$.
	\end{theorem}

	Define
	\[
		\rho_K: \Cl(K)(p) \longrightarrow \bigoplus_{e \in \calE} e\Cl(K),
	\] 
	and note that $\rho_K$ is obtained by taking tensor product of $\Cl(K)(p)$ with the injective homomorphism 
	$\Z_p[\Gamma] \to \oplus_{e\in \calE} e\Z_p[\Gamma]$.
	The image of $\rho_K$ can be described by Theorem~\ref{thm:main-B} and Proposition~\ref{prop:trivial-idem}, then one may naturally ask about the kernel of $\rho_K$. 
	
	\begin{theorem}\label{thm:main-C}
		Let $Q$ be either $\Q$ or $\F_q(t)$ with $\gcd(q, p|\Gamma|)=1$.
		\begin{enumerate}
			\item\label{item:C-1} If $p\nmid |\Gamma|$ or $\Gamma=\Z/p\Z$, then $\ker \rho_K=1$ for every $\Gamma$-extension $K/Q$. 
			\item\label{item:C-2}(Special case of Theorem~\ref{thm:kernel-limit}) If $p^2 \mid |\Gamma|$, then for every simple $\F_p[\Gamma]$-module $A$, 
			\[
				\lim_{X \to \infty} \frac{\sum\limits_{D \leq X} \sum\limits_{K \in \calA_{\Gamma}^+(D,Q)} \rk_A \ker \rho_K}{ \# \calA^+_{\Gamma}(D, Q)}= \infty,
			\]
			where $\rk_A \ker \rho_K:=\max\{r\in \Z \mid A^{\oplus r} \text{ is a quotient of } \ker\rho_K\}$.
		\end{enumerate}
	\end{theorem}
	When $\Gamma=\Gamma' \times \Z/p\Z$ for some nontrivial abelian group $\Gamma'$ with $p \nmid |\Gamma'|$ (i.e., the only case when neither of the assumptions in \eqref{item:C-1} and \eqref{item:C-2} holds), our method cannot help to determine the distribution of $\ker \rho_K$.
	
\subsection{Comparison to previous work}
\hfill

	The theorems above and Conjecture~\ref{conj:GG} agree with Cohen--Lenstra--Martinet heuristics (when $p\nmid |\Gamma|$) and the Gerth conjecture (when $\Gamma=\Z/p\Z$), and we will explain that in \S\ref{ss:CLM} and \S\ref{ss:Gerth}.

\subsubsection{Comparing to the Cohen--Lenstra--Martinet heuristics}\label{ss:CLM}

	Cohen and Martinet \cite{Cohen-Martinet} generalized the Cohen--Lenstra heuristics to the situation of $\Gamma$-extensions of $Q$ for an arbitrary number field $Q$ as a base field and an arbitrary finite group $\Gamma$. In particular, when $p\nmid |\Gamma|$ and $Q=\Q$, as $K$ varies over all totally real $\Gamma$-extensions of $\Q$, they conjectured that the probability that $\Cl(K)(p) \simeq H$ is inversely proportional to $|\Aut_{\Gamma}(H)||H|$ for any $\Z_p[\Gamma]$-module $H$ with $H^{\Gamma}=1$ (see \cite[Theorem~1.1]{Wang-Wood}). 
	
	Assume $p \nmid |\Gamma|$ and $\Gamma$ is abelian. For every idempotent $e$ of $\Q_p[\Gamma]$, $p$ does not divide the denominator of $e$, so $\Z_p[\Gamma]=e\Z_p[\Gamma] \oplus (1-e) \Z_p[\Gamma]$. It follows that $\Z_p[\Gamma] = \bigoplus_{e \in \calE} e\Z_p[\Gamma]$ and $M=\bigoplus_{e \in \calE} eM$ for any $\Z_p[\Gamma]$-module $M$. For every $\gamma \in \Gamma$, because $p\nmid |\gamma|$, there is no nonzero $\Z_p[\Gamma]$-module that is annihilated by both $1-\gamma$ and $\sum_{j=1}^{|\gamma|} \gamma^j$. So the proper ideal $I$ described in Theorem~\ref{thm:main-A} does not exist, and hence we previously defined $I_e:=e\Z_p[\Gamma]$ when $p \nmid |\Gamma|$. In Theorem~\ref{thm:main-B}\eqref{item:main-B-2}, the weight function has value constantly 1, so \eqref{eq:ff-weight-moment} proves the moment of the Cohen--Lenstra--Martinet conjecture in this case under a large $q$ limit.
	Moreover, since $\Cl(K)(p)=\bigoplus_{e\in \calE} e\Cl(K)(p)$, Conjecture~\ref{conj:GG} agrees with the Cohen--Lenstra--Martinet conjecture.

\subsubsection{Comparing to the Gerth conjecture}\label{ss:Gerth}

	Assume $Q=\Q$ and $\Gamma=\Z/p\Z$ with a generator $\gamma$. Let $R$ denote the ring $\Z_p[\Gamma]/ (\sum_{j=1}^p \gamma^j)$, which is a local ring where the maximal ideal is generated $1-\gamma$. The norm map annihilates $\Cl(K)$, so $\Cl(K)(p)$ is an $R$-module. By the genus theory, the $\Gamma$-coinvariant of $\Cl(K)(p)$, which is $\Cl(K)(p)/(1-\gamma)\Cl(K)(p)$, is an $\F_p$-vector space whose rank is determined by the number of primes ramified in $K/\Q$. Gerth \cite{Gerth84, Gerth86} proposed conjecture about the distribution of $(1-\gamma)\Cl(K)(p)$, and Gerth's conjecture is proven by Smith, Koymans and Pagano \cite{Smith2022, Koymans-Pagano}.
	
	Consider the ring $\Q_p[\Gamma]$. There are two isomorphism classes of irreducible $\Q_p[\Gamma]$-modules: the trivial one $V_0:=\Q_p$ and the nontrivial one $V_1:=\Q_p[\Gamma]/\Q_p$, corresponding to the idempotents $e_0:=\frac{\sum_{j=1}^p \gamma^j}{p}$ and $e_1:=1-e_0$ respectively. By definition of $e\Z_p[\Gamma]$, one see that
	\[
		e_0\Z_p[\Gamma]\simeq \Z_p[\Gamma]/(1-\gamma) \quad \text{and}\quad e_1\Z_p[\Gamma] \simeq \Z_p[\Gamma]/(\sum_{j=1}^p \gamma^j).
	\]
	 Note that if a finite $\Z_p[\Gamma]$-module is annihilated by both $1-\gamma$ and $\sum_{j=1}^p \gamma^j$, then it must be isomorphic to $\F_p^{\oplus r}$ for some $r \in \NN$. So $I_{e_0}=\frakm_{e_0}=pe_0\Z_p[\Gamma]$ and $I_{e_1}=\frakm_{e_1}=(1-\gamma)e_1\Z_p[\Gamma]$. Theorem~\ref{thm:main-A} (together with the explicit description of the family of ramification type given in Theorem~\ref{thm:lb-rank}) says that the rank $\rk_{\F_p} \Cl(K)(p)/(1-\gamma)\Cl(K)(p)$ has a lower bound determined by the number of primes ramified in $K/\Q$. Comparing to the genus theory result, Theorem~\ref{thm:main-A} only gives a lower bound of the rank, but is strong enough to imply that the average of the rank is infinite. Since the norm map is zero on $\Cl(K)$, $\Cl(K)(p)$ is an $e_1\Z_p[\Gamma]$-module, so we have $e_1\Cl(K)=\Cl(K)$ and $I_{e_1}e_1\Cl(K)=\frakm_{e_1} e_1\Cl(K)=(1-\gamma)\Cl(K)(p)$. So Conjecture~\ref{conj:GG} agrees with the Gerth conjecture in the totally real case, and Theorem~\ref{thm:main-B} proves a weighted version of the moment conjecture in the function field and totally real case (under $q\to \infty$).

\subsection{Methods and outline of the paper}
\hfill
	
	Theorem~\ref{thm:main-A} is proved by studying the presentation of Galois group with restricted ramification, which generalizes the method in the author's previous work \cite{Liu2022b}. The basic idea is: if $e\Cl(K)$ can be presented by generators and relations using only the local information, then one can estimate $\rk_I e\Cl(K)$ since the relations are in a particular form (in the form of tame local relations). For example, when $p=3$ and $\Gamma=\Z/3\Z$, if $K/\Q$ is a tamely ramified $\Z/3\Z$-extension, then by \cite[Theorem~4.3]{Liu2022b}, there is a surjective homomorphism
	\begin{equation}\label{eq:pres-example}
		\varphi: e\Z_3[\Gamma]^{\oplus r} \rtimes \Gamma \longrightarrow \Cl(K)(3) \rtimes \Gal(K/\Q)
	\end{equation}
	where $e$ is the nontrivial idempotent of $\Q_3[\Gamma]$, and $r$ is one less than the number of primes ramified in $K/\Q$; and $\ker \varphi$ is generated by relations 
	\[
		x_{\ell}^{-1} y_{\ell}^{-1} x_{\ell} y_{\ell}, \quad \ell \in \{\text{prime numbers ramified in $K/\Q$}\},
	\]
	where $x_{\ell}$ has order 3 and $x_{\ell} \not\in e\Z_p[\Gamma]^{\oplus r}$. Then one see that all the relators are contained in $\frakm_e \cdot (e\Z_3 [\Gamma])^{\oplus r}$, and it follows immediately that $\rk_{\frakm_e} \Cl(K)(3) = r$. 
	The method in \cite{Liu2022b} uses the local-global principle for central embedding problems, so it can be applied to study pro-$p$ extensions. In general situation, working only with central embedding problems is not enough; and also, when we change the base field to an arbitrary global field, the local-global principle of embedding problem could fail. Therefore a nice presentation as \eqref{eq:pres-example} usually does not exist.
	
	For the general case, we show that the local-global principle of embedding problem with restricted ramification holds if the associated cohomology invariant $\B$ vanishes (see Lemma~\ref{lem:EP}). When the invariant $\B$ does not vanish, we can relax the ramification restriction at finitely many primes to make $\B$ vanish (see Lemma~\ref{lem:frakS}). Then, after applying the local-global principle, we obtain a presentation of the maximal Galois group with the relaxed ramification restriction. By carefully estimating the number of those ``relaxed'' primes and comparing that to the generator rank (e.g., the number $r$ in \eqref{eq:pres-example}), we obtain a presentation similar to \eqref{eq:pres-example}. Although we cannot give all the relations in that presentation explicitly, we show that all but a bounded number of the relations are in the form of tame local relations, which is sufficient to conclude Theorem~\ref{thm:main-A}. 
	Theorem~\ref{thm:main-B}\eqref{item:main-B-1} follows by Theorem~\ref{thm:main-A}, and the proof of Theorem~\ref{thm:main-C} uses the presentations described above and the properties of projection maps $M \to eM$.
	
	The proof of Theorem~\ref{thm:main-B}\eqref{item:main-B-2} utilizes the method of counting $\F_q$-points on the Hurwitz spaces, which has been previously used in proving the function field case of Cohen--Lenstra heurstics and its generalizations (\cite{EVW}, \cite{Boston-Wood}, \cite{LWZB}, etc.). For an $e\Z_p[\Gamma]$-module $H$, by counting points on appropriate Hurwitz spaces, one can compute the average of $\#\Sur_{\Gamma}(e\Cl(K), H)$. We prove in Proposition~\ref{prop:weight-formula} that, if $H$ and $M:=I_eH$ have the same rank, then
	\[
		\#\Sur_{\Gamma}(e\Cl(K), H) = \#\Sur_{\Gamma} (I_e \cdot e\Cl(K), M) \cdot w_{e, M}(K);
	\]
	and then we prove \eqref{eq:ff-weight-moment} by comparing the number of points on the Hurwitz spaces that correspond to $\#\Sur_{\Gamma}(e\Cl(K), H)$ and $\#\Sur_{\Gamma}(e\Cl(K), H/M)$.	
	
	We define the ring $e\Z_p[\Gamma]$ and prove basic properties of $e\Z_p[\Gamma]$-modules in Section~\ref{sec:module}. In Section~\ref{sec:MainResults}, we establish the statements of the main results of the paper in the most general form; and we show that Theorem~\ref{thm:main-A}, Theorem~\ref{thm:main-B}\eqref{item:main-B-1} and Thereom~\ref{thm:main-C} follow from those main results. In Section~\ref{sect:B}, we study the cohomology invariant $\B$. In Section~\ref{sect:A-rk-bounds}, we estimate the generator rank of the presentation of Galois groups with restricted ramification, which will be used in the proofs of Theorem~\ref{thm:lb-rank} (general form of Theorem~\ref{thm:main-A}) and Theorem~\ref{thm:kernel-limit} (general form of Theorem~\ref{thm:main-C}). In Section~\ref{sect:presentation}, we prove the local-global principle for embedding problems and apply it to construct the desired presentations. Then we prove the main results Theorem~\ref{thm:lb-rank} and Theorem~\ref{thm:kernel-limit} in Sections~\ref{sect:proof-main} and \ref{sect:Proof-kernel-limit} respectively. In Sections~\ref{sect:ff-moment-prep} and \ref{sect:ff-moment-proof}, we prove the function field weighted moment result Thereom~\ref{thm:main-B}\eqref{item:main-B-2}; and in Section~\ref{sect:Z/2Z}, we prove Theorem~\ref{thm:Z/2Z}. Finally, in Section~\ref{sect:conj}, we compute the probability measure that is determined by the moment in \eqref{eq:ff-weight-moment} without the weight function, and state our conjecture about the probability and moment for the distribution of $I_e\cdot e\Cl(K)$.

\subsection{Notation}\label{ss:notation}
\hfill

	In this paper, groups are always finite or profinite groups, and subgroups are topologically closed subgroups. For a group $G$, we let $G^{\ab}$ denote the abelianization of $G$. For two elements $a,b \in G$, we write $a^b:=b^{-1}ab$ and $[a,b]:=a^{-1}b^{-1}ab$. 
	For a group $G$, we write $G(p)$ for the pro-$p$ completion of $G$. For an abelian group $G$, we let $G[p^{\infty}]$ denote the Sylow $p$-subgroup of $G$.
	If $H$ is a group with a continuous $G$-action, then the semidirect product $H \rtimes G$ is the group with underlying set $\{(h,g)\mid h \in H, g \in G\}$ and the multiplication $(h_1,g_1)(h_2,g_2)=(h_1g_1(h_2),g_1g_2)$. We write $\Hom_G$, $\Sur_G$, and $\Aut_G$ to represent the sets of $G$-equivariant homomorphisms, surjections, and automorphisms. If $M$ is a $G$-module, $M^G$ and $M_G$ are the $G$-invariant and $G$-coinvariant of $H$ respectively. 
	
	For a ring $R$, an ideal $I$ of $R$ and an $R$-module $M$, we denote the modules $M[I]:=\{x \in M \mid Ix=0\}$ and $M_{/I}:=M/IM$. Let
	\[
		\calM_R:=\{ \text{isomorphism classes of finite simple $R$-modules}\}.
	\]
	For a field $k$, we write $\overline{k}$ for a fixed choice of separable closure of $k$, and denote $G_k:=\Gal(\overline{k}/k)$. For a global field $k$ and a prime $\frakp$ of $k$, denote by $k_{\frakp}$ the completion of $k$ at $\frakp$. We fix an embedding $\overline{k} \hookrightarrow \overline{k_{\frakp}}$, then we have an injection $\eta: G_{k_{\frakp}} \hookrightarrow G_k$. Let $\calG_{\frakp}(k):=\im(\eta)$ and $\calT_{\frakp}(k)$ be the image of the inertia subgroup of $G_{k_{\frakp}}$ under the map $\eta$. When the choice of $k$ is clear, we denote $\calG_{\frakp}(k)$ and $\calT_{\frakp}(k)$ by $\calG_{\frakp}$ and $\calT_{\frakp}$. 
	For a Galois extension $K/k$, let $\calG_{\frakp}(K/k)$ and $\calT_{\frakp}(K/k)$ be the images of $\calG_{\frakp}(k)$ and $\calT_{\frakp}(k)$ under the quotient map $G_{k} \twoheadrightarrow \Gal(K/k)$.

	Throughout the paper, we let $Q$ be a global field, $\Gamma$ a finite abelian group, and $p$ a prime number such that $\Char Q$ does not divide $p|\Gamma|$. Let $\Gamma_p$ denote the Sylow $p$-subgroup of $\Gamma$, and $\Gamma'$ the maximal prime-to-$p$ subgroup of $\Gamma$; so $\Gamma=\Gamma_p\times \Gamma'$. 		
	For a function $f(x,y)$ of two variables $x$ and $y$, if 
	\[
		\lim_{x \to \infty} \limsup_{y \to \infty} f(x,y) = \lim_{x \to \infty} \liminf_{y \to \infty} f(x,y)=C,
	\]
	then we write 
	\[
		\lim_{x \to \infty} \lim_{y \to \infty} f(x,y) =C.
	\]

\subsection*{Acknowledgement}
	
	The author would like to thank Melanie Matchett Wood for motivating her to study the question of generalizing Gerth's conjectures and for comments on an early draft, and Peter Koymans for writing the appendix. The author would like to thank Stephanie Chan, Carlo Pagano, Nick Rome, and Alexander Smith for helpful conversations. The author was partially supported by NSF grant DMS-2200541, and part of the work was supported by NSF DMS-1928930 while the author was in residence at the Simons Laufer Mathematical Sciences Institute in Berkeley, California.

\section{Structure of $\Z_p[\Gamma]$-modules}\label{sec:module}

	 The ring $\Q_p[\Gamma]$ is semisimple. By the Krull-Schmidt theorem, $\Q_p[\Gamma]$ has the unique decomposition property; and moreover, each simple $\Q_p[\Gamma]$-module is isomorphic to $e \Q_p[\Gamma]$ for some primitive idempotent $e$. In particular,
	 \[
	 	\Q_p[\Gamma]=\bigoplus_{e \in \calE} e \Q_p[\Gamma],
	 \]
	 where 
	 \[
	 	\calE:=\{\text{primitive idempotents of $\Q_p[\Gamma]$}\}.
	\]
	The ring $\Z_p[\Gamma]$ can be uniquely decomposed as a direct sum of indecomposible modules, as we discuss below.
	For $A\in \calM_{\F_p[\Gamma']}$, there is a unique (up to isomorphism) projective $\Z_p[\Gamma']$-module $P$ such that $P/pP\simeq A$, and we define $P_A:=\Z_p[\Gamma_p] \otimes_{\Z_p} P$. 
	By \cite[Proposition 42(a) and \S15.7(c)]{Serre}, every projective $\Z_p[\Gamma]$-module is isomorphic to $P_A$ for some $A\in \calM_{\F_p[\Gamma]}$, $\Z_p[\Gamma]$ can be decomposed as 
	\begin{equation}\label{eq:PA}
		\Z_p[\Gamma]= \bigoplus_{A\in \calM_{\F_p[\Gamma]}} P_A,
	\end{equation}
	and each $P_A$ is a projective indecomposible $\Z_p[\Gamma]$-module. In particular, $\calM_{\F_p[\Gamma]}=\calM_{\F_p[\Gamma']}$.
	
	\begin{definition}\label{def:eM}
		Let $e$ be an idempotent of the ring $\Q_p[\Gamma]$. Define
		\[
			e\Z_p[\Gamma]:= \{ex \mid x \in \Z_p[\Gamma] \} \subset \Q_p[\Gamma],
		\]
		which is naturally a $\Z_p[\Gamma]$-module and a commutative ring with multiplicative identity $e$. For a $\Z_p[\Gamma]$-module $M$, define an $e\Z_p[\Gamma]$-module
		\[
			eM:= e\Z_p[\Gamma] \otimes_{\Z_p[\Gamma]} M.
		\]
	\end{definition}
	
	There is a natural surjective ring homomorphism
	\begin{eqnarray}
		\Z_p[\Gamma] & \longrightarrow & e\Z_p[\Gamma] \label{eq:to-e} \\
		x &\longmapsto & ex. \nonumber
	\end{eqnarray}
	As a $\Z_p[\Gamma]$-module, $e\Z_p[\Gamma]$ can also be defined as a quotient of $\Z_p[\Gamma]$ using the following lemma.

	\begin{lemma}\label{lem:prop_eM}
		Let $e$ be a primitive idempotent of $\Q_p[\Gamma]$. The following are equivalent.
		\begin{enumerate}
			\item\label{item:prop_eM-1} $M\simeq e\Z_p[\Gamma]$.
			\item\label{item:prop_eM-2} $M$ is a quotient module of $\Z_p[\Gamma]$ such that $\ker(\Z_p[\Gamma] \to M) =(1-e) \Q_p[\Gamma] \cap \Z_p[\Gamma]$. In other words, $M$ is the image of $\Z_p[\Gamma]$ under the quotient map $\Q_p[\Gamma] \to e\Q_p[\Gamma]$.
\		\item\label{item:prop_eM-3} $M$ is a quotient module of $\Z_p[\Gamma]$ satisfying both of the following conditions
			\begin{enumerate}
				\item \label{item:peM-1} $M$ is free as a $\Z_p$-module.
				\item \label{item:peM-2} $M \otimes_{\Z_p}\Q_p \simeq e\Q_p[\Gamma]$.
			\end{enumerate}
		\end{enumerate}
	\end{lemma}
	
	\begin{proof}
		By definition of $e\Z_p[\Gamma]$, \eqref{item:prop_eM-1} implies \eqref{item:prop_eM-3}. 
		The kernel of the surjection \eqref{eq:to-e} is $(1-e)\Q_p[\Gamma] \cap \Z_p[\Gamma]$, so \eqref{item:prop_eM-1} and \eqref{item:prop_eM-2} are equivalent.
				
		Suppose $\pi: \Z_p[\Gamma] \to M$ is a surjection such that $M$ satisfies both \eqref{item:peM-1} and \eqref{item:peM-2}. Because $\Q_p$ is a flat $\Z_p$-module, by taking tensor product, $\pi$ gives
		\[
			1 \longrightarrow \ker \pi \otimes_{\Z_p} \Q_p \longrightarrow \Q_p[\Gamma] \longrightarrow M \otimes_{\Z_p}\Q_p \longrightarrow 1.
		\]
		By \eqref{item:peM-2}, it follows that $\ker \pi \otimes_{\Z_p} \Q_p$ is $(1-e) \Q_p[\Gamma]$. Since $\ker \pi$ is a submodule of $\Z_p[\Gamma]$, it is $\Z_p$-free, so $\ker \pi$ embeds into $\ker \pi \otimes_{\Z_p} \Q_p$, and hence $\ker \pi \subseteq \Z_p[\Gamma] \cap (1-e)\Q_p[\Gamma]$. By comparing the $\Z_p$-ranks, $\ker \pi$ is a submodule of $\Z_p[\Gamma] \cap (1-e)\Q_p[\Gamma]$ of finite index, so $M \twoheadrightarrow \Z_p[\Gamma]/(\Z_p[\Gamma] \cap (1-e) \Q_p[\Gamma])$ has finite kernel. Finally, since both $M$ and $\Z_p[\Gamma]/(\Z_p[\Gamma] \cap (1-e) \Q_p[\Gamma])$ are $\Z_p$-free, $\ker \pi = \Z_p[\Gamma] \cap (1-e)\Q_p[\Gamma]$, so $M$ is isomorphic to $e\Z_p[\Gamma]$.
	\end{proof}
	
	The following lemma shows that each $e\Z_p[\Gamma]$ is a quotient of $P_A$ for a unique $A$.
	
	\begin{lemma}\label{lem:e-A}
		For each $e\in \calE$, there is a unique simple $\F_p[\Gamma]$-module $A$ such that the quotient map $\Z_p[\Gamma]\to e\Z_p[\Gamma]$ in Lemma~\ref{lem:prop_eM} factors through $\Z_p[\Gamma] \to P_A$. In particular, $e\Z_p[\Gamma]$ is a local ring and its quotient by the maximal ideal is isomorphic to $A$.
	\end{lemma}
	\begin{proof}
		Because $\Gamma$ is abelian, the direct sum decomposition of $\Q_p[\Gamma]$ as irreducible modules is unique \cite[Lemma~1.8.2]{Benson}, and in particular, irreducible modules in this decomposition are pairwisely non-isomorphic. So there is a unique $A$ such that $\Q_p[\Gamma] \to e\Q_p[\Gamma]$ factors through the quotient map $\Q_p[\Gamma] \to P_A \otimes_{\Z_p} \Q_p$. For all $B\in \calM_{\F_p[\Gamma]}$ such that $B\neq A$, the image of the submodule $P_B\otimes_{\Z_p}\Q_p \subset \Q_p[\Gamma]$ in  $e\Q_p[\Gamma]$ is zero, then because $e\Z_p[\Gamma]$ is $\Z_p[\Gamma]$-free, we have $P_B\subset \ker(\Z_p[\Gamma] \to e\Z_p[\Gamma])$. So $\Z_p[\Gamma]\to e\Z_p[\Gamma]$ factors through $P_A$ as desired.
		
		Recall $P_A=\Z_p[\Gamma_p]\otimes_{\Z_p} P$ for the projective $\Z_p[\Gamma']$-module $P$ satisfying $P/pP\simeq A$. Since $\Z_p[\Gamma_p]$ is a local ring with residue field $\F_p$, $P_A$ has a unique maximal proper submodule and the quotient of $P_A$ by the maximal ideal is isomorphic to $A$. So $e\Z_p[\Gamma]$ also has a unique maximal proper $\Z_p[\Gamma]$-submodule, as it is a quotient of $P_A$. Passing along the ring morphism $\Z_p[\Gamma] \to e\Z_p[\Gamma]$ sending $1\mapsto e$, the $\Z_p[\Gamma]$-submodules of $e\Z_p[\Gamma]$ are exactly the ideals of the ring $e\Z_p[\Gamma]$. So $e\Z_p[\Gamma]$ as a ring has a unique maximal ideal, and then it is a local ring.
	\end{proof}

	\begin{notation}
		\begin{enumerate}
			\item For a primitive idempotent $e$ of the ring $\Q_p[\Gamma]$, let $\frakm_e$ be the maximal ideal of the local ring $e\Z_p[\Gamma]$.
			\item For a simple $\F_p[\Gamma]$-module $A$, define the following set
			\[
				\Idem(A):=\{\text{primitive idempotents $e$ of $\Q_p[\Gamma]$ such that $e\Z_p[\Gamma]/\frakm_e \simeq A$}\}.
			\]
		\end{enumerate}
	\end{notation}

\subsection{Properties of $e\Z_p[\Gamma]$}
\hfill

	In this subsection, we collect basic properties of the ring $e\Z_p[\Gamma]$ for every primitive idempotent $e$ of $\Q_p[\Gamma]$. Throughout, we assume $e$ is a primitive idempotent, i.e., $e \in \calE$.
	
	\begin{lemma}\label{lem:max-ideal}
		For each $e \in \calE$, there exists a cyclic quotient $C$ of $\Gamma$ such that the $\Gamma$-action on $e\Z_p[\Gamma]$ factors through $C$ and $C$ acts faithfully on $e\Z_p[\Gamma]$. The existence of $C$ defines a bijective correspondence between $\calE$ and the set of all cyclic quotients of $\Gamma$. Moreover, the maximal ideal $\frakm_e$ is described as follows.
		\begin{enumerate}
			\item\label{item:max-ideal-1}
				If $p\nmid |C|$, then $\frakm_e=p(e\Z_p[\Gamma])$.
			
			\item\label{item:max-ideal-2}
				If $p\mid |C|$, then $\frakm_e=(1-\gamma)e\Z_p[\Gamma]$, where $\gamma\in \Gamma$ is a preimage of a generator of the Sylow $p$-subgroup of $C$.
		\end{enumerate} 
	\end{lemma}
	
	\begin{proof}
		By Schur's lemma, the endomorphism ring $\End_{\Gamma}(e\Q_p[\Gamma])$ of the irreducible $\Q_p[\Gamma]$-module $e\Q_p[\Gamma]$ is a finite dimensional division algebra over $\Q_p$. Because $\Gamma$ is finite abelian, the image of $\Gamma \to \End_{\Gamma}(e\Q_p[\Gamma])$ is a torsion subgroup of the center of $\End_{\Gamma}(e\Q_p[\Gamma])$, and hence is a torsion subgroup of the multiplicative group of a field extension of $\Q_p$. So the image of $\Gamma\to \End_{\Gamma}(e\Q_p[\Gamma])$ is a finite cyclic group, and then the $\Gamma$-action on $e\Z_p[\Gamma] \subset e\Q_p[\Gamma]$ factors through a finite cyclic quotient of $\Gamma$. Let $C$ be the smallest such cyclic quotient, so that $C$ acts faithfully on $e\Z_p[\Gamma]$. By the representation theory of cyclic groups, there is a unique irreducible $\Q_p[C]$-module with faithful $C$-action, so $e\Q_p[\Gamma]$ is isomorphic to this unique irreducible module. Because $e\Q_p[\Gamma] \not\simeq e'\Q_p[\Gamma]$ for any $e'\in \calE$ with $e'\neq e$ (by \cite[Proposition~1.7.2]{Benson}), the map from $\frakc: \calE \to \{\text{cyclic quotients of }\Gamma\}$ that sends $e$ to its associated $C$ is an injection. This map is also surjective, since for any cyclic quotient $C$ of $\Gamma$, an irreducible $\Q_p[C]$-module is naturally an irreducible $\Q_p[\Gamma]$-module. Thus, the map $\frakc$ gives a bijective correspondence.
				
		If $p\nmid |C|$, by \cite[Proposition~43(ii)]{Serre}, as $e\Z_p[\Gamma]$ is a $\Z_p$-lattice of $e\Q_p[\Gamma]$ and $e\Q_p[\Gamma]$ is an irreducible $\Q_p[C]$-module, $e\Z_p[\Gamma]/  p(e\Z_p[\Gamma])$ is an irreducible $\F_p[C]$-module, so the maximal ideal of $e\Z_p[\Gamma]$ is generated by $p$. 
		If $p\mid |C|$, then by \cite[\S 15.7.(a)]{Serre}, the Sylow $p$-subgroup of $C$ acts trivially on the quotient of $e\Z_p[\Gamma]$ by its maximal ideal, so $(1-\gamma)e\Z_p[\Gamma]$ is contained in the maximal ideal. Let $\sigma$ be an element of $C$ whose order is $p$. Consider the map
		\begin{eqnarray*}
			\alpha: e\Z_p[\Gamma] &\longrightarrow& e\Z_p[\Gamma] \\
			x &\longmapsto& \sum_{i=1}^p \sigma^i(x)
		\end{eqnarray*}
		which is a homomorphism of $\Z_p[\Gamma]$-modules because $\Gamma$ is abelian. Then since $e\Q_p[\Gamma]$ is irreducible, the homomorphism $\hat{\alpha}: e\Q_p[\Gamma]\to e\Q_p[\Gamma]$ obtained by taking tensor product of $\Q_p$ along $\alpha$ is either zero or an isomorphism. Because $\sigma$ acts trivially on $\im\alpha$, it also acts trivially on $\im\hat{\alpha}$. Thus, the assumption that $C$ acts faithfully on $e\Z_p[\Gamma]$ implies $\im \hat{\alpha}=0$, so $\im \alpha=0$. Thus, $\sum_{i=1}^p \sigma^i$ annihilates $e\Z_p[\Gamma]$. Then, about the module $H:=e\Z_p[\Gamma]/(1-\gamma)$, we know that $\sigma$ acts trivially on $H$ and $\sum_{i=1}^p \sigma^i$ annihilates $H$. So $H$ has exponent $p$, and hence it is an $\F_p[C/\langle \gamma\rangle]$-module. Finally, because $\F_p[C/\langle{\gamma}\rangle]$ is semisimple (as $p\nmid |C/\langle \gamma \rangle|$) and $H$ is a quotient of local ring, $H$ is simple, which shows that the maximal ideal of $e\Z_p[\Gamma]$ is $(1-\gamma)e\Z_p[\Gamma]$.
	\end{proof}
	
		The following lemma provides more information about the bijective correspondence in Lemma~\ref{lem:max-ideal}. 
	
	\begin{lemma}\label{lem:gamma-ann}
		For every $\gamma \in \Gamma$, exactly one of $1-\gamma$ and $\sum_{j=1}^{|\gamma|} \gamma^j$ annihilates $e\Z_p[\Gamma]$. Moreover, for each simple $\F_p[\Gamma]$-module $A$, the map
		\begin{equation}\label{eq:idem-cyclicquot}
			\Idem(A) \longrightarrow \{ \text{cyclic quotients of $\Gamma_p$}\}
		\end{equation}
		sending $e$ to the quotient of $\Gamma_p$ by the maximal subgroup of $\Gamma_p$ that acts trivially on $e\Z_p[\Gamma]$ is a bijection.
	\end{lemma}

	\begin{proof}
		First, note that if $\gamma$ acts trivially on a $\Q_p[\Gamma]$-module, then $\sum_{j=1}^{|\gamma|} \gamma^j$ acts as multiplication by $|\gamma|$ on this module, which gives an automorphism. So there is no nonzero module that is annihilated by both $1-\gamma$ and $\sum_{j=1}^{|\gamma|} \gamma^j$. Then because $\Q_p[\Gamma]=(1-\gamma)\Q_p[\Gamma] \oplus (\sum_{j=1}^{|\gamma|} \gamma^j)\Q_p[\Gamma]$, where the two direct summands are annihilated by $\sum_{j=1}^{|\gamma|} \gamma^j$ and $1-\gamma$ respectively, the simple module $e\Q_p[\Gamma]$ is a submodule of exactly one of these two summands, so it is annihilated by exactly one of $\sum_{j=1}^{|\gamma|} \gamma^j$ and $1-\Gamma$. Then the first claim in the lemma follows by $e\Z_p[\Gamma]\otimes_{\Z_p} \Q_p \simeq e\Q_p[\Gamma]$. 

		Consider the case when $\Gamma$ is an abelian $p$-group. The Grothendieck group of $\Q_p[\Gamma]$-modules is generated by $\Ind_C^{\Gamma}\Q_p$ where $C$ runs over all cyclic subgroups of $\Gamma$ (for example, one may show that by following the proof of \cite[Theorem~30]{Serre} with $\Q$ replaced with $\Q_p$ and using the fact that $\Gal(\Q_p(\zeta_{p^m})/\Q_p)\simeq (\Z/p^m\Z)^\times$). Since $\Gamma$ is abelian, $\Ind_C^{\Gamma} \Q_p \simeq \Q_p[\Gamma]^{C}=(\sum_{j=1}^{|\gamma|} \gamma^j)\Q_p[\Gamma]$ for a generator $\gamma$ of the cyclic subgroup $C$. So, for $e_1\neq e_2$ in $\Idem(\F_p)$, there must be an element $\gamma \in \Gamma$ acting trivially on exactly one of $e_1$ and $e_2$. So, the map \eqref{eq:idem-cyclicquot} is injective. On the other hand, if $C'$ is a cyclic quotient of $\Gamma$, then $\Q_p[C']$ contains an irreducible faithful $\Q_p[C']$-module, so $C'$ is in the image of \eqref{eq:idem-cyclicquot}, and hence \eqref{eq:idem-cyclicquot} is surjective.
		
		Consider the general case: $\Gamma=\Gamma_p\times \Gamma'$ where $\Gamma_p$ is the Sylow $p$-subgroup of $\Gamma$. For a simple $\F_p[\Gamma]$-module $A$, recall that $P_A=\Z_p[\Gamma_p] \otimes_{\Z_p} P$, where $P$ is the unique projective $\Z_p[\Gamma']$-module such that $P/pP\simeq A$, and recall that $P_A \otimes \Q_p = \oplus_{e \in \Idem(A)} e\Q_p[\Gamma]$. So there is a bijective correspondence between $\Idem(A)$ and the set of primitive idempotent of $\Q_p[\Gamma_p]$, defined by sending $e\in \Idem(A)$ to the primitive idempotent $f$ of $\Q_p[\Gamma_p]$ such that $e\Q_p[\Gamma]=f\Q_p[\Gamma_p] \otimes_{\Q_p} (P \otimes \Q_p)$. Since $\Gamma_p$ acts trivially on $P$, a subgroup of $\Gamma_p$ acts trivially on $e\Z_p[\Gamma]$ if and only if it acts trivially on $f\Q_p[\Gamma_p]$, so the bijectivity of \eqref{eq:idem-cyclicquot} follows by the special case above.
	\end{proof}
	
	\begin{proposition}\label{prop:dvr}
		The local ring $e\Z_p[\Gamma]$ is a complete discrete valuation ring.
	\end{proposition}
	
	\begin{proof}
		Since $e\Q_p[\Gamma]$ has no nonzerodivisor and $e\Z_p[\Gamma] \subset e\Q_p[\Gamma]$, $e\Z_p[\Gamma]$ is an integral domain. Then by Lemma~\ref{lem:max-ideal}, $e\Z_p[\Gamma]$ is a Noetherian local domain whose maximal ideal is principal, so it is a discrete valuation domain. By definition of $e\Z_p[\Gamma]$, one see that it is completed with respect to the ideal $p(e\Z_p[\Gamma])$, so by \cite[\href{https://stacks.math.columbia.edu/tag/0319}{Lemma 0319}]{stacks-project} it is complete with respect to its maximal ideal. Therefore, $e\Z_p[\Gamma]$ is a complete discrete valuation ring.
	\end{proof}

\subsection{Structure of $e\Z_p[\Gamma]$-modules and decomposition of $\Z_p[\Gamma]$-modules as $e\Z_p[\Gamma]$-modules}
\hfill

	Because $e\Z_p[\Gamma]$ is a discrete valuation ring, the $e\Z_p[\Gamma]$-modules can be classified using the lemma below.
		
	\begin{lemma}\label{lem:DVR-module}
		\begin{enumerate}
			\item\label{item:DVR-1}
				Every finitely generated $e\Z_p[\Gamma]$-module is isomorphic to a finite direct sum of modules of the form $e\Z_p[\Gamma]/\frakm_e^k$ for positive integers $k$.
			\item\label{item:DVR-2}
				For any nonzero ideal $I$ of $e\Z_p[\Gamma]$ and any finite $e\Z_p[\Gamma]$-module $H$, $H[I]$ is isomorphic to $H_{/I}$ as $e\Z_p[\Gamma]$-module. 
			\item \label{item:DVR-3}
				For any positive integer $n$ and any $e\Z_p[\Gamma]$-submodule $H$ of $e\Z_p[\Gamma]^{\oplus n}$ of finite index, we have $H \simeq e\Z_p[\Gamma]^{\oplus n}$.
		\end{enumerate}
	\end{lemma}
	
	\begin{proof}
		The statements~\eqref{item:DVR-1} follows by Proposition~\ref{prop:dvr} and the classification of finite modules over discrete valuation rings; and \eqref{item:DVR-1} implies \eqref{item:DVR-2}.
		
		Let $H$ be a submodule of $e\Z_p[\Gamma]^{\oplus n}$ of finite index. There exists a positive integer $m$ such that $(\frakm_e^{m-1})^{\oplus n} \subset H$. Then $M:=H/(\frakm_e^m)^{\oplus n}$ is a finite module. By \eqref{item:DVR-2}, $H_{/\frakm_e} = M_{/\frakm_e} \simeq M[\frakm_e]\simeq \frakm_e^{\oplus n}$, so $H$ is a $n$-generated module. Since $H$ has finite index in $e\Z_p[\Gamma]^{\oplus}$, $H$ is a free $\Z_p$-module whose rank is the same as the $\Z_p$-rank of $e\Z_p[\Gamma]^{\oplus n}$, so $H\simeq e\Z_p[\Gamma]^{\oplus n}$.
	\end{proof}

	\begin{definition}\label{def:rank}
		Define the following notation of ranks of $\Z_p[\Gamma]$-modules.
		\begin{itemize}
			\item For a simple $\F_p[\Gamma]$-module $A$ and a finitely generated $\Z_p[\Gamma]$-module $H$, the \emph{$A$-rank of $H$}, denoted by $\rk_A H$, is the maximal interger $r$ such that $A^{\oplus r}$ is a quotient of $H$. 
			\item For a nonzero proper ideal $I$ of $e\Z_p[\Gamma]$, let $d$ be the integer such that $I=\frakm_e^d$, and let $A:= e\Z_p[\Gamma]/\frakm_e$.
			Then, for a finitely generated $e\Z_p[\Gamma]$-module $H$, the \emph{$I$-rank of $H$}, denoted by $\rk_I H$, is defined to be $\rk_A (\frakm_e^{d-1}H)$.
		\end{itemize}
	\end{definition}
	\begin{remark}
		Throughout this paper, for an elementary abelian $p$-group $M$, whether it is a $\F_p[\Gamma]$-module or not, we let $\rk_{\F_p}M$ denote the rank of $M$ as an $\F_p$-module. When we want to refer to the $A$-rank of a $\F_p[\Gamma]$-module $M$ for $A=\F_p$, we will always write ``$\rk_A M$ for $A=\F_p$''.
	\end{remark}
	For a finitely generated $\Z_p[\Gamma]$-module $H$ and a simple $\F_p[\Gamma]$-module $A$,  
	\begin{equation}\label{eq:e-rk}
		\rk_A H=\frac{\dim_{\F_p} \Hom_{\Z_p[\Gamma]}(H, A)}{\dim_{\F_p} \End_{\Z_p[\Gamma]}(A)}.
	\end{equation}
	For an $e\Z_p[\Gamma]$-module $H$, there is a filtration 
	\[
		H \supset \frakm_e H \supset \frakm_e^2 H \supset \ldots.
	\]
	From the above definition, for each positive integer $i$,
	\[
		\faktor{\frakm_e^{i-1} H}{\frakm_e^i H} \simeq \left(\faktor{e\Z_p[\Gamma]}{\frakm_e}\right)^{\oplus \rk_{\frakm_e^i} H}.
	\]
	Therefore, if $I \subsetneq J$ are two ideals of $e\Z_p[\Gamma]$, then $\rk_I H \geq \rk_J H$ for any $e\Z_p[\Gamma]$-module $H$. Moreover, the isomorphism class of $H$ is uniquely determined by its $I$-ranks for all ideals $I$.

	\begin{notation}\label{not:idem}
		For any $\Z_p[\Gamma]$-module $M$ and $e \in \calE$, let 
		\[
			\rho_{M, e}: M \longrightarrow e M
		\]
		denote the quotient map obtained by taking tensor product of $M$ with $\Z_p[\Gamma] \twoheadrightarrow e \Z_p[\Gamma]$, and denote
		\[
			 \rho_{M} = \bigoplus_{e \in \calE} \rho_{M, e} : M \longrightarrow \bigoplus_{e \in \calE} e M.
		\]
	\end{notation}

	When $p\nmid |\Gamma|$, the map $\rho_M$ is always an isomorphism because $\Z_p[\Gamma] \simeq \bigoplus_{e \in \calE} e\Z_p[\Gamma]$ by \cite[Proposition~43]{Serre}.
	When $p \mid |\Gamma|$, $P_A \to \bigoplus_{e \in \Idem(A)} e\Z_p[\Gamma]$ is not an isomorphism because $P_A$ is indecomposible but $\bigoplus_{e \in \Idem(A)} e\Z_p[\Gamma]$ is not. The map $\rho_M$ is not necessarily surjective or injective: for example, assume $\Gamma=\Z/3\Z$ is generated by an element $\gamma$. Consider the module $M$ such that: $M$ is isomorphic to $\Z/9\Z$ as a group and $\gamma(x)=x^4$ for every $x \in M$. There are two primitive idempotents $e_0=(\sum_{i=1}^9 \gamma^i)/9$ and $e_1=1-e_0$. One can check that $e_0 M \simeq e_1 M = \F_3$, so $\rho_M$ is neither surjective or injective.

	We end this section with the following lemma about simple $\F_p[\Gamma]$-modules for abelian $\Gamma$.

	\begin{lemma}\label{lem:End=A}
		For a simple $\F_p[\Gamma]$-module $A$,
		\[
			\dim_{\F_p}A=\dim_{\F_p}\End_{\Gamma}(A).
		\]
	\end{lemma}
	\begin{proof}
		By \cite[Remark~5.2]{Liu-Wood}, $\frac{\dim_{\F_p} A}{\dim_{\F_p} \End_{\Gamma}(A)}$ is the maximal number $m$ such that $A^{\oplus m}$ can be generated by one element as a $\Z_p[\Gamma]$-module, i.e., it is the maximal number $m$ such that $A^{\oplus m}$ is a quotient module of $\Z_p[\Gamma]$. By the decomposition \eqref{eq:PA}, we have $m=1$.
	\end{proof}

\section{Main results and outline of the paper}\label{sec:MainResults}
	
	In this section, we list definitions and notations that will be used throughout the paper, and list the main theorems in the most general form.
	
	Let $\Gamma$ be a finite abelian group and $p$ a prime. 
	Let $Q$ be a global field whose characteristics does not divide $p|\Gamma|$. For a finite group $G$, a \emph{$G$-extension of $Q$} is a surjective homomorphism $G_Q \to G$, and equivalently, is a pair $(K, \iota)$ where $K/Q$ is a Galois extension and $\iota$ is an isomorphism $\Gal(K/Q) \to G$. We will omit $\iota$ from the notation when the isomorphism is not explicitly used. Two $G$-extensions $(K_1, \iota_1)$ and $(K_2, \iota_2)$ of $Q$  are isomorphic if there exists an isomorphism $\phi:K_1 \to K_2$ fixing $Q$ such that the induced isomorphism $\phi_*:\Gal(K_1/Q) \to \Gal(K_2/Q)$ satisfies $\iota_1=\iota_2 \circ \phi_*$. 
	For a set $S$ of primes of $Q$ and an extension $K/Q$, let $S(K)$ denote the set of all primes of $K$ lying above primes in $S$.
	When $K$ is a number field, let $S_p(K)$ denote the set of all primes of $K$ that lies above the prime $(p)$ of $\Q$.
	Throughout this paper, we always let $S$ and $T$ denote two finite sets of primes of $Q$.  
	Let $Q_S^T$ denote the maximal extension of $Q$ that is unramified away from $S$ and split completely at primes in $T$, and for an extension $K$ of $Q$, let $K_S^T:=K_{S(K)}^{T(K)}$. Then denote 
	\[
		G_S^T(Q):=\Gal(Q_S^T/Q) \quad \text{and}  \quad G_S^T(K):=\Gal(K_S^T/K).
	\]
	For a $\Gamma$-extension $K/Q$, we define
	\begin{eqnarray*}
		E_S^T(K) &:=& \text{the maximal abelian $p$-extension of $K$ that is contained in $K_S^T$}, \\
		C_S^T(K) &:=& \Gal(E_S^T(K)/K)=G_S^T(K)^{\ab}(p).
	\end{eqnarray*}
	Then the short exact sequence
	\[
		1 \longrightarrow C_S^T(K) \longrightarrow \Gal(E_S^T(K)/Q) \longrightarrow \Gal(K/Q) \longrightarrow 1
	\]
	defines a natural $\Gamma$-action on $C_S^T(K)$ via the conjugation of $\Gal(E_S^T(K)/Q)$, which defines a $\Z_p[\Gamma]$-module structure on $C_S^T(K)$. 
		Let $\infty$ denote the set of primes of $Q$ that lie above the unique archimedean prime of $\Q$ when $Q$ is a number field and lie above the unique infinite place of $\F_q(t)$ when $Q$ is an extension of $\F_q(t)$. Denote
		\[
			\Cl(K):=C_{\O}^{\{\infty\}}(K) \quad \text{and} \quad \Cl_T(K):=C_{\O}^T(K) .
		\]
	The $S$-unit group is $\calO_{K,S} = \{x \in K \mid v_{\frakp}(x) \geq 0 \text{ for all } \frakp \not\in S(K)\}$, and denote $\calO_{K}=\calO_{K,\O}$. Then $\calO_K:=\calO_{K,\O}$ is the ring of integers when $K$ is a number field, and is the finite field of constant when $K$ is a function field.

	Let $e$ be a primitive idempotent of $\Q_p[\Gamma]$. Retain the notation from Section~\ref{sec:module}. We let 
	\[
		e C_S^T(K):= e\Z_p[\Gamma] \otimes_{\Z_p[\Gamma]} C_S^T(K);
	\]
	in particular, $e C_S^T(K)$ is a quotient $\Z_p[\Gamma]$-module of $C_S^T(K)$. We let $e E_S^T(K)$ denote the subfield of $E_S^T(K)$ fixed by $\ker(C_S^T(K) \to eC_S^T(K))$, so $\Gal(eE_S^T(K)/ K)$ is $eC_S^T(K)$. Note that $e E_S^T(K)$ is Galois over $Q$. 
	We define 
	\[
		\rho_S^T(K, e) : C_S^T(K) \longrightarrow eC_S^T(K) \quad \text{and} \quad \rho_S^T(K) = \bigoplus_{e \in \calE} \rho_S^T(K, e):  C_S^T(K) \longrightarrow \bigoplus_{e \in \calE} e C_S^T(K)
	\]
	to be the maps $\rho_{M,e}$ and $\rho_M$ in Notation~\ref{not:idem} for the module $M=C_S^T(K)$.

	Recall that $e\Z_p[\Gamma]$ is a discrete valuation ring with the maximal ideal $\frakm_e$.

	\begin{definition}\label{def:thresfold-ideal}
		For each idempotent $e\in \calE$, define an ideal $I_e$ of $e\Z_p[\Gamma]$ as 
		\[
			I_e:= \bigcap_{ 1\neq \gamma \in \Gamma} \rho_{\Z_p[\Gamma], e}\bigg( \big(1-\gamma\, ,\,  \,\sum_{j=1}^{|\gamma|} \gamma^j \big)\bigg),
		\]
		where $( 1-\gamma, \sum_{j=1}^{|\gamma|}\gamma^j )$ is the ideal of $\Z_p[\Gamma]$ generated by $1-\gamma$ and $\sum_{j=1}^{|\gamma|}\gamma^j$.
	\end{definition}
	
	\begin{lemma}\label{lem:Ie-proper}
		If $\gamma\in \Gamma$ is a nontrivial element such that $\rho_{\Z_p[\Gamma], e}((1-\gamma, \sum_{j=1}^{|\gamma|} \gamma^j))$ is a proper ideal of $e\Z_p[\Gamma]$, then $p \mid |\gamma|$. In particular, the ideal $I_e$ is proper if and only if $p\mid |\Gamma|$. 
	\end{lemma}
	
	\begin{proof}
		Let $A:=e\Z_p[\Gamma]/\frakm_e$, and let $\gamma$ be as described in the lemma. Then both $1-\gamma$ and $\sum_{j=1}^{|\gamma|} \gamma^j$ annihilate $A$. So $\gamma$ acts trivially on $A$, and then $\sum_{j=1}^{|\gamma|} \gamma^j(x)=|\gamma|x$ for any $x \in A$, which implies that $|\gamma|$ must be divisible by $p$. 
		
		By Definition~\ref{def:thresfold-ideal}, there exists $\gamma \in \Gamma$ such that $I_e= \rho_{\Z_p[\Gamma], e}((1-\gamma, \sum_{j=1}^{|\gamma|} \gamma^j))$. So if $I_e$ is proper then $p \mid |\Gamma|$. On the other hand, 
		if $p \mid |\Gamma|$, then $\Gamma_p$ acts trivially on $A$. Then for a nontrivial element $\gamma \in \Gamma_p$, both $1-\gamma$ and $\sum_{j=1}^{|\gamma|} \gamma^j$ annihilate $A$, so $I_e \subseteq \frakm_e$. 
	\end{proof}

	\begin{definition}\label{def:dagger}
		Let $(K,\iota)$ be a $\Gamma$-extension of $Q$ and $e \in \calE$. Given an ideal $I$ of $e\Z_p[\Gamma]$, we let $\scrR_I(K/Q)$ denote the set of primes of $Q$ satisfying the following conditions.
		\begin{enumerate}
			\item \label{item:dagger-RA} $\frakp \not\in S_p(Q)$.
			
			\item \label{item:dagger-1}  As a subgroup of $\Gamma$, the inertia subgroup $\iota(\calT_{\frakp}(K/Q))$ can be generated by a nontrivial element $\gamma \in \Gamma$, such that the image of the ideal 
			\[
				\left(1-\gamma, \, \sum_{j=1}^{|\gamma|} \gamma^j \right) \subset \Z_p[\Gamma]
			\]
			is contained in $I$ under the quotient map $\rho_{e,\Z_p[\Gamma]}: \Z_p[\Gamma] \to e\Z_p[\Gamma]$.
			
			\item \label{item:dagger-2} As a subgroup of $\Gamma$, the decomposition subgroup $\iota(\calG_{\frakp}(K/Q))$ acts trivially on $e\Z_p[\Gamma]/I$.
		\end{enumerate}
		Note that, by definition of $I_e$ in Definition~\ref{def:thresfold-ideal}, if $I \subset I_e$, then $\scrR_I(K/Q)$ is empty.
	\end{definition}
	
	\begin{remark}
	\begin{itemize}
		\item Because $\Gamma$ is assumed to be abelian, the inertia (resp. decomposition) subgroup of $\Gal(K/Q)$ at $\frakp$ does not depend on the choice of primes of $K$ lying above $\frakp$.
		\item Because $e\Z_p[\Gamma]$ is a discrete valuation ring, by definition of $I_e$, there exist elements $\gamma\in \Gamma$ such that the image of the ideal $(1-\gamma, \sum_{j=1}^{|\gamma|} \gamma^j)$ is $I_e$. For such an element $\gamma$, $1-\gamma$ annihilates $e\Z_p[\Gamma]/I_e$, so the subgroup $\langle \gamma \rangle$ of $\Gamma$ acts trivially on $e\Z_p[\Gamma]/I_e$.
	\end{itemize}
	\end{remark}

	\begin{theorem}\label{thm:lb-rank}
		Let $e\in \calE$ and $I$ be a proper ideal of $e\Z_p[\Gamma]$ such that $I_e \subseteq I$ (so $p\mid |\Gamma|$ by Lemma~\ref{lem:Ie-proper}). For any $\Gamma$-extension $K$ of $Q$, there is a lower bound of the $I$-rank of $eC_S^T(K)$:	
		\begin{equation}\label{eq:lb-rank}
			\rk_{I} eC_S^T(K) \geq \#\scrR_I(K/Q)- c,
		\end{equation}
		where $c$ is a constant depending on $Q$, $S$, $T$, $\Gamma$ and $e$, but not on the field $K$.
	\end{theorem}
	
	We will prove Theorem~\ref{thm:lb-rank} in Section~\ref{sect:proof-main}. The following corollary is an immediate consequence of Theorem~\ref{thm:lb-rank}.
	\begin{corollary}\label{cor:lb-rank}
		Let $e\in \calE$ and $I$ be a proper ideal of $e\Z_p[\Gamma]$ such that $I_e \subseteq I$.
		Assume $\calF$ is a family of $\Gamma$-extensions of $Q$, and there is an invariant $H(K)\in \R$ defined for every $K\in \calF$ such that the set 
		\[
			\calB_{\calF}(X):=\{K \in \calF \mid H(K) \leq X\}
		\]
		is finite for every $X \in \Z_{\geq 0}$.
		If
		\begin{equation}\label{eq:avg-np}
			\lim_{X \to \infty} \frac{ \sum_{K \in \calB_{\calF}(X)} \#\scrR_I(K/Q)}{\# \calB_{\calF}(X)} = \infty,
		\end{equation}
		then
		\[
			\lim_{X \to \infty} \frac{ \sum_{K \in \calB_{\calF}(X)} \rk_{I} eC_S^T(K)}{\# \calB_{\calF}(X)} = \infty.
		\]
	\end{corollary}
	
	When $Q$ is a number field and the extensions are ordered by the absolute norm of the radical of the discriminant ideal (which is the product of ramified primes if $Q=\Q$), then \eqref{eq:avg-np} holds.
	
	\begin{definition}\label{def:calA}
		Given a global field $Q$, for an extension $K/Q$, let $\rDisc K$ denote the absolute norm of the radical of the discriminant ideal $\Disc(K/Q)$. We say \emph{a family of sets of $\Gamma$-extensions $\{ \calA_{\Gamma}(X,Q) \mid X \in \Z\}$ satisfies ramification restriction at finitely many primes} if there exists 
		\begin{enumerate}
			\item a finite set $\calZ$ of primes of $Q$, and
			\item for each $\frakp \in \calZ$, there is a set $U_{\frakp}$ of Galois \'etale algebra over $Q_{\frakp}$ of Galois group $\Gamma$,
		\end{enumerate}
		such that
		\[
			\calA_{\Gamma}(X,Q)=\{\text{$\Gamma$-extensions }K/Q \mid \rDisc(K/Q)\leq X \text{ and } K_{\frakp} \in U_{\frakp}, \forall \frakp \in \calZ\}.
		\]
		Here $K_{\frakp}:= \prod_{\frakP \mid \frakp} K_{\frakP}$, where the product is taken over all primes $\frakP$ of $K$ lying above $\frakp$, is naturally a Galois \'etale algebra over $Q$ with Galois group $\Gal(K/Q)\simeq\Gamma$.
		
	\end{definition}
	
	\begin{theorem}\label{thm:conductor}
		Let $Q$ be a number field. Let $e\in \calE$ and $I$ be a proper ideal of $e\Z_p[\Gamma]$ such that $I_e \subseteq I$. Assume $\calA_{\Gamma}(X, Q), X\in \Z$ satisfies ramification restriction at finitely many primes and is non-empty when $X$ is sufficiently large. Then
		\[
			\lim_{X \to \infty} \frac{\sum\limits_{K\in \calA_{\Gamma}(X,Q)}\, \rk_{I} eC_S^T(K)}{ \# \calA_{\Gamma}(X,Q)}= \infty.
		\]
	\end{theorem}
	
	Theorem~\ref{thm:main-A} follows by applying Theorem~\ref{thm:lb-rank} to $Q=\Q$ or $\F_q(t)$ and $S=\O$, $T=\{\infty\}$; and Theorem~\ref{thm:main-B}\eqref{item:main-B-1} is a special case of Theorem~\ref{thm:conductor} because $\cup_{D\leq X} \calA_{\Gamma}^+(D,\Q)$, $X \in \Z$ satisfies ramification restriction at only $\infty$.
	
	\begin{proof}[Proof of Theorem~\ref{thm:conductor}]
		By definition of $I_e$ in Definition~\ref{def:thresfold-ideal}, there exists a nontrivial element $\gamma \in \Gamma$ such that $I_e=\rho_{\Z_p[\Gamma], e} ((1-\gamma, \sum_{j=1}^{|\gamma|} \gamma^j))$. Let $\Gamma_0$ be the cyclic subgroup of $\Gamma$ generated by $\gamma$. By Definition~\ref{def:dagger}, if a prime $\frakp \not \in S_p(Q)$ and $\calT_{\frakp}(K/Q)=\calG_{\frakp}(K/Q)=\Gamma_0$, then $\frakp \in \scrR_{I}$. So 
		\[
			\#\scrR_I(K/Q) \geq \#\{\frakp \subset Q \mid \calT_{\frakp}(K/Q) = \calG_{\frakp}(K/Q)= \Gamma_0\}-[Q:\Q].
		\]
		For every tuple $t=(t_{\frakp})_{\frakp \in \calZ} \in \prod_{\frakp \in \calZ} U_{\frakp}$, we define $\calA^t_{\Gamma}(X, Q):= \{K \in \calA_{\Gamma}(X,Q) \mid K_{\frakp}=t_{\frakp}, \forall \frakp \in \calZ \}$. If $\calA_{\Gamma}^t(X,Q)$ is not empty when $X$ is large, then by Corollary~\ref{cor:lb-rank} and Theorem~\ref{tCond}, we have
		\[
			\lim_{X \to \infty} \frac{\sum\limits_{K\in \calA^t_{\Gamma}(X,Q)}\, \rk_{I} eC_S^T(K)}{ \# \calA^t_{\Gamma}(X,Q)}= \infty.
		\]
		The proof is completed, noting that $\prod_{\frakp \in \calZ}U_{\frakp}$ must be a finite set since there are only finitely many Galois \'etale algebra over $Q_{\frakp}$ of Galois group $\Gamma$. 
	\end{proof}
	
	\begin{remark}
		When all the $\Gamma$-extensions of $Q$ are ordered by absolute discriminant, then the condition \eqref{eq:avg-np} can fail: for example, when $Q=\Q$, $\Gamma=\Z/6\Z$ and $p=3$, as discussed in Appendix Remark~\ref{rmk:disc}. However, if $\ell$ is the minimal prime divisor of $|\Gamma|$, and there exists $\gamma \in \Gamma$ of order $\ell$ such that $I_e=\rho_{\Z_p[\Gamma],e}((1-\gamma, \sum_{j=1}^{\ell} \gamma^j))$,  then by the same argument in the above proof and applying Theorem~\ref{thm:A.2}, one can show that Theorem~\ref{thm:conductor} still holds when ordering by absolute discriminant.	
	\end{remark}
	
	Writing $A:=e\Z_p[\Gamma]/\frakm_e$, by \eqref{eq:e-rk}, for any positive integer $d$,
	\[
		\rk_{\frakm_e^d} eC_S^T(K)=\rk_A \frakm_e^{d-1} \cdot eC_S^T(K)=\frac{\log_p (\#\Sur_{\Gamma}(\frakm_e^{d-1} \cdot eC_S^T(K), A)+1)}{\dim_{\F_p} \End_{\Z_p[\Gamma]} (A)}.
	\]
	So Theorem~\ref{thm:conductor} implies that, for any ideal $I \supsetneq I_e$ of $e\Z_p[\Gamma]$,
	\begin{equation}\label{eq:lb-rank-moment}
		\lim_{X \to \infty} \frac{\sum_{K \in \calA_{\Gamma}(X,Q)} \#\Sur_{\Gamma}(I \cdot eC_S^T(K), A)}{\#\calA_{\Gamma}(X, Q)}=\infty.
	\end{equation}
	On the other hand, from the proof of Theorem~\ref{thm:lb-rank}, one will see that there exists a lower bound of the rank in terms of the number of primes ramified in $K/Q$ as \eqref{eq:lb-rank} if and only if $I\subseteq I_e$. In fact, when $I\subseteq I_e$, one should not expect \eqref{eq:lb-rank-moment} to hold, c.f., Theorem~\ref{thm:main-B}\eqref{item:main-B-2}.

	Finally, we state the general form of Theorem~\ref{thm:main-C}.
	
	\begin{theorem}\label{thm:kernel-limit}
		Let $Q$ be a number field. Assume $p^2 \mid |\Gamma|$, and $\calA_{\Gamma}(X, Q), X\in \Z$ satisfies ramification restriction at finitely many primes and is non-empty when $X$ is sufficiently large. Then for every simple $\F_p[\Gamma]$-module $A$, 
		\[
			\lim_{X \to \infty} \frac{\sum\limits_{K\in \calA_{\Gamma}(X,Q)}\, \rk_{A} \ker\rho_S^T(K)}{ \# \calA_{\Gamma}(X,Q)}= \infty.
		\]
	\end{theorem}

	\begin{remark}
		Theorems~\ref{thm:conductor} and \ref{thm:kernel-limit} can be generalized to function fields if one can recover the results in Appendix~\ref{ss:appendix}.
	\end{remark}
	
	The proof of Theorem~\ref{thm:kernel-limit} will be given in Section~\ref{sect:Proof-kernel-limit}. We end this section with proving Theorem~\ref{thm:main-C}.

	\begin{proof}[Proof of Theorem~\ref{thm:main-C}]
		The statement~\eqref{item:C-2} in Theorem~\ref{thm:main-C} follows directly from Theorem~\ref{thm:kernel-limit}, because $\rho_K=\rho_{\O}^{\{\infty\}}(K)$ and $\cup_{D\leq X} \calA_{\Gamma}^+(D,Q)$ satisfies ramification restriction at finitely many primes (in fact, at only $\infty$). So it suffices to prove \eqref{item:C-1}.
		
		When $p \nmid |\Gamma|$, $\Z_p[\Gamma]=\bigoplus_{e \in \calE} e\Z_p[\Gamma]$, so $M=\bigoplus_{e \in \calE} eM$ for any finite module $M$, and hence $\ker \rho_{K}=0$. For the rest, assume  $\Gamma = \Z/p\Z$, and let $\gamma$ be a generator of $\Gamma$. Since $\Cl(Q)=0$ when $Q=\Q$ or $\F_q(t)$, the norm map annihilates the class group, so $\sum_{i=1}^p \gamma^i$ annihilates the $\Z_p[\Gamma]$-module $\Cl(K)(p)$. Note that $e_1:=1-\frac{\sum_{i=1}^p \gamma^i}{p}$ is a primitive idempotent, and $e_1\Z_p[\Gamma]=\Z_p[\Gamma]/(\sum_{i=1}^p \gamma^i)$ by Lemma~\ref{lem:prop_eM}\eqref{item:prop_eM-2}. So the desired result follows by $\Cl(K)(p)=e\Cl(K)$.
	\end{proof}

\section{Cohomological invariant {$\B_{S \backslash T}^{S \cup T} (Q, A)$}.} \label{sect:B}

	Let $Q$ be a global field and $p$ be a prime number such that $p \neq \Char(Q)$.
	Associated to a finite $\F_p[G_Q]$-module $A$, there is a cohomological invariant $\B_{S\backslash T}^{S\cup T}(Q, A)$ (defined in \cite[Definition~3.1]{Liu2022b}), which is the cokernel of the following composite map
	\begin{equation}\label{eq:def-B}
		\prod_{\frakp \in S\backslash T} H^1(\calG_{\frakp}, A) \times \prod_{\frakp \not \in S\cup T} H^1_{\nr}(\calG_{\frakp}, A) \xhookrightarrow{\quad} \prod_{\frakp} H^1(\calG_{\frakp}, A) \xrightarrow{\quad} \prod_{\frakp} H^1(\calG_{\frakp}, A')^{\vee} \xrightarrow{\quad} H^1(G_Q, A')^{\vee}.
	\end{equation}
	Here $\calG_{\frakp}$ is the absolute Galois group of the local field $Q_{\frakp}$, $A'$ is $\Hom(A, \overline{Q}^{\times})$, $M^{\vee}$ is the Pontryagin dual of a module $M$, and $H^1_{\nr}(\calG_{\frakp}, A):=\ker(H^1(\calG_{\frakp},A)\to H^1(\calT_{\frakp},A)^{\calG_{\frakp}})$ is the unramified cohomology group. The second and the third terms in the maps are products over all primes of $Q$. The first map is the natural embedding, the second map is the product of isomorphisms obtained by the local Tate duality, and the last map is the Pontryagin dual of the product of restriction maps.

	\begin{lemma}\label{lem:B}
		Let $L$ be a Galois extension of $Q$ such that $p \nmid [L:Q]$. Then $\Gal(L/Q)$ acts on $\B_{S\backslash T(L)}^{S\cup T (L)} (L, A)$ via the conjugation action on cohomology groups, and
		\[
			\B_{S\backslash T}^{S\cup T}(Q, A) \simeq \B_{S\backslash T(L)}^{S\cup T (L)} (L, A)^{\Gal(L/Q)}.
		\]
	\end{lemma}
	
	\begin{proof}
		Fix a prime $\frakp$ of $Q$ and a prime $\frakP$ of $L$ lying above $\frakp$, and denote 
		\[
			\Delta:=\Gal(L_{\frakP}/Q_{\frakp}). 
		\]
		Let $\calG_{\frakp}:=\calG_{\frakp}(Q)$, $\calT_{\frakp}:=\calT_{\frakp}(Q)$, $\calG_{\frakP}:=\calG_{\frakP}(L)$ and $\calT_{\frakP}:=\calG_{\frakP}(L)$. Note that $\calG_{\frakP} \unlhd \calG_{\frakp}$ and $\calT_{\frakP} \unlhd \calT_{\frakp}$ by our definition in Section~\ref{ss:notation}.
				
		Because of $p\nmid |\Delta|$,
		the following restriction map and corestriction map 
		\begin{equation}\label{eq:B-1}
			H^1(\calG_{\frakp}, A) \xrightarrow{\,\res\,} H^1(\calG_{\frakP}, A)^{\Delta} \quad \text{and} \quad H^1(\calG_{\frakP}, A)_{\Delta} \xrightarrow{\,\cor\,} H^1(\calG_{\frakp}, A)
		\end{equation}
		are isomorphisms.
		Therefore, by taking product of all primes above $\frakp$, one obtain the following commutative diagram by \cite[Proposition~(1.5.6)]{NSW}
		\begin{equation}\label{eq:B-2}\begin{tikzcd}
			H^1(\calG_{\frakp}, A')^{\vee} \arrow["\res^{\vee}"]{r} \arrow["\cor^{\vee}", "\sim"']{d} & H^1(G_Q, A')^{\vee} \arrow["\cor^{\vee}", "\sim"']{d} \\
			\left(\prod\limits_{\frakP\mid \frakp} H^1(\calG_{\frakP}, A')^{\vee} \right)^{\Gal(L/Q)} \arrow["\res^{\vee}"]{r} & \left(H^1(G_L, A')^{\vee} \right)^{\Gal(L/Q)}.
		\end{tikzcd}\end{equation}
		For the unramified cohomology groups, consider the diagram
		\[\begin{tikzcd}
			H_{\nr}^1(\calG_{\frakp}, A) \arrow["\res"]{d} \arrow[hook, "\inf"]{r} & H^1(\calG_{\frakp}, A) \arrow[two heads, "\res"]{r} \arrow["\res"]{d} & H^1(\calT_{\frakp}, A)^{\calG_{\frakp}} \arrow["\res"]{d} \\
			H_{\nr}^1(\calG_{\frakP}, A) \arrow[hook, "\inf"]{r} & H^1(\calG_{\frakP}, A) \arrow[two heads, "\res"]{r} & H^1(\calT_{\frakP}, A)^{\calG_{\frakP}},
		\end{tikzcd}\]
		where the two horizontal restriction maps are surjective because $\calG_{\frakp}/\calT_{\frakp}$ and $\calG_{\frakP}/\calT_{\frakP}$ are both isomorphic to $\hZ$, the right square commutes by the definition of restriction map, and the left square commutes by applying \cite[Proposition~(1.5.5)(i)]{NSW} to $\calT_{\frakP} \unlhd \calG_{\frakP} \unlhd \calG_{\frakp}$. Since $p\nmid |\Delta|$, the middle and the right vertical restriction maps are injective and send the upper entries isomorphically to the $\Delta$-invariant of the lower entries. So by the snake lemma, the diagram implies an isomorphism
		\begin{equation}\label{eq:B-3}
			H_{\nr}^1(\calG_{\frakp}, A) \xrightarrow{\,\res\,} H_{\nr}^1(\calG_{\frakP}, A)^{\Delta}.
		\end{equation}

		Next, we study how the Tate Duality is compatible with base field change between $Q_\frakp$ and $L_{\frakP}$. First, assume $\frakp$ is nonarchimedean. Because the Tate Duality for nonarchimedean primes \cite[Theorem~(7.2.6)]{NSW} is a special case of the Tate spectral sequence \cite[Theorem~(2.5.3)]{NSW}, which is functorial in the sense that it is well-behaved under taking open subgroups. By \cite[p.122-123]{NSW} and the fact that $p \nmid |\Delta|$, we have the following commutative diagram   
		\begin{equation}\label{eq:B-4}\begin{tikzcd}
			H^1(\calG_{\frakp}, A) \arrow["\res", "\sim"']{d} \arrow["\text{TD}", "\sim"']{r} & H^1(\calG_{\frakp}, A')^{\vee} \arrow["\cor^{\vee}", "\sim"']{d} \\
			H^1(\calG_{\frakP}, A)^{\Delta} \arrow["\text{TD}", "\sim"']{r} & \left( H^1(\calG_{\frakP}, A')^{\vee}\right)^{\Delta}.
		\end{tikzcd}\end{equation}
		For an archimedean prime $\frakp$, if $\calG_{\frakp}\neq \calG_{\frakP}$, then $[L:Q]$ is even and hence $p$ is odd, in which case, every entry in \eqref{eq:B-4} is zero; otherwise, $\calG_{\frakp} = \calG_{\frakP}$ and the diagram \eqref{eq:B-4} obviously commutes. So for any prime $\frakp$ (archimedean or not), the commutative diagram \eqref{eq:B-4} always holds.
		
		Finally, comparing the definition of $\B_{S\backslash T}^{S\cup T}(Q,A)$ and $\B_{S\backslash T(L)}^{S\cup T(L)}(L, A)$ in \eqref{eq:def-B}, the desired isomorphism in the lemma follows by \eqref{eq:B-1}, \eqref{eq:B-2}, \eqref{eq:B-3} and \eqref{eq:B-4}. 
	\end{proof}

	\begin{lemma}\label{lem:B-QvsL}
		Let $L:=Q(A, \mu_p)$ denote the minimal trivializing extension of $Q$ for the modules $A$ and $\mu_p$, and $S$, $T$ be finite sets of primes of $Q$. Let $r_S^T(L,A)$ be the maximal integer such that $\B_{S\backslash T(L)}^{S\cup T(L)}(L, \F_p)$ has a $\Gal(L/Q)$-equivariant quotient isomorphic to $(A^{\vee})^{\oplus r_S^T(L,A)}$. Then 
		\[
			\B_{S\backslash T}^{S \cup T} (Q, A)\simeq \End_{G_Q} (A^{\vee}) ^{\oplus r_S^T(L, A)}.
		\]
	\end{lemma}
	
	\begin{proof}
		By Lemma~\ref{lem:max-ideal}, the Sylow $p$-subgroup of $\Gamma$ acts trivially on $A$, so $[Q(A):Q]$ is prime to $p$. Also, $[Q(\mu_p):Q]$ is prime to $p$, so $L=Q(A)Q(\mu_p)$ is an abelian extension of $Q$ of degree prime to $p$. 
		Because $G_L$ acts trivially on $A$ and $A'$, for any prime $\frakP$ of $L$, $G_{\frakP}$ acts trivially on $A$, so the cup product induces the following $\Gal(L/Q)$-equivariant isomorphisms
		\[
			H^1(G_{\frakP}, \F_p) \otimes A \xrightarrow{\, \sim \,}H^1(G_{\frakP}, A) \quad \text{and} \quad H_{\nr}^1(G_{\frakP}, \F_p) \otimes A \xrightarrow{\, \sim \,}H_{\nr}^1(G_{\frakP}, A).
		\]
		For $G$ being either $G_{\frakP}$ or $G_L$, for the same reason, we have a $\Gal(L/Q)$-equivariant isomorphism
		\[
			H^1(G, \mu_p) \otimes \Hom(A, \F_p) \xrightarrow{\,\sim\,} H^1(G, A')
		\]
		defined by the cup product associated to the bilinear map
		\begin{eqnarray*}
			\mu_p \times \Hom(A, \F_p) &\longrightarrow& \Hom(A, \mu_p) \\
			(\xi, f) &\longmapsto& (x \mapsto \xi^{f(x)}).
		\end{eqnarray*}
		So we have functorial isomorphisms
		\begin{eqnarray*}
			H^1(G, A')^{\vee} &\simeq& \Hom ( H^1(G, \mu_p) \otimes \Hom(A, \F_p) , \F_p) \\
			&\simeq & \Hom\left( H^1(G, \mu_p), \Hom(A, \F_p)^{\vee} \right) \\
			&\simeq & \Hom \left( H^1(G, \mu_p), A \right) \\
			&\simeq & H^1(G, \mu_p)^{\vee} \otimes A,
		\end{eqnarray*}
		where the second isomorphism follows by the Tensor-Hom adjunction. Moreover, one can check that the diagram
		\[\begin{tikzcd}
			H^1(\calG_{\frakP}, A) \arrow["\sim"]{d} \arrow["\text{TD}"]{rr} && H^1(\calG_{\frakP}, A')^{\vee} \arrow["\sim"]{d} \\
			H^1(\calG_{\frakP}, \F_p) \otimes A \arrow["\text{TD}\otimes \id"]{rr} && H^1(\calG_{\frakP}, \mu_p)^{\vee} \otimes A
		\end{tikzcd}\]
		commutes. So by definition of $\B_{S\backslash T(L)}^{S\cup T (L)}(L, A)$, we obtain a $\Gal(L/Q)$-equivariant isomorphism $\B_{S\backslash T(L)}^{S\cup T(L)}(L, A)\simeq \B_{S\backslash T(L)}^{S \cup T(L)}(L, \F_p) \otimes A$.
		
		By Lemma~\ref{lem:B},
		\begin{eqnarray*}
			\B_{S\backslash T}^{S\cup T}(Q, A) &\simeq& \B_{S\backslash T(L)}^{S \cup T(L)} (L, A)^{\Gal(L/Q)} \\
			&\simeq & \left( \B_{S\backslash T(L)}^{S \cup T(L)}(L, \F_p) \otimes A \right)^{\Gal(L/Q)} \\
			&\simeq& \Hom\left(\B_{S\backslash T(L)}^{S \cup T(L)}(L, \F_p) \otimes A , \F_p\right)^{\Gal(L/Q)}\\
			&\simeq & \Hom_{\Gal(L/Q)} \left( \B_{S\backslash T(L)}^{S \cup T(L)}(L, \F_p), A^{\vee} \right)\\
			&\simeq& \End_{\Gal(L/Q)}(A^{\vee})^{\oplus r_S^T(L,A)}.
		\end{eqnarray*}
		Here, the third isomorphism uses the fact that $\F_p[\Gal(L/Q)]$ is semisimple and $M^{\Gal(L/Q)} \simeq (M^{\vee})^{\Gal(L/Q)}$ for any $\F_p[\Gal(L/Q)]$-module.  
		Then the proof is completed.
	\end{proof}

		\begin{lemma}\label{lem:es-B}
		Let $k$ be a Galois extension of $Q$, $S_1\subset S_2$ and $T$ finite sets of primes of $Q$, and $A$ a finite $\F_p[\Gal(k_{S_1}^T/Q)]$-module. Then there exists a $\Gal(k/Q)$-equivariant exact sequence
		\[
			H^1(G_{S_1}^T(k), A) \hookrightarrow H^1(G_{S_2}^T(k), A) \rightarrow \bigoplus_{\frakP \in S_2 \backslash (S_1\cup T) (k)} H^1(\calT_{\frakP}, A)^{\calG_{\frakP}} \rightarrow \B_{S_1\backslash T(k)}^{S_1 \cup T(k)}(k,A) \twoheadrightarrow \B_{S_2 \backslash T(k)}^{S_2 \cup T(k)}(k,A).
		\]
	\end{lemma}
	
	\begin{proof}
		This lemma is a generalization of Lemma 8.4 in \cite{Liu-presentation} and the proof is the same, despite that one need to appropriately change the sets of primes that the product of local cohomology groups is taken over in the proof of \cite[Lemma~8.4]{Liu-presentation}.
	\end{proof}

	\begin{lemma}\label{lem:frakS}
		Let $A$ be a finite simple $\F_p[G_Q]$-module such that $\Gal(Q(A)/Q)$ is abelian . Let $S$ and $T$ be two sets of primes of $Q$.
		Then there exists a set $\frakS$ of primes of $Q$ such that 
		\begin{enumerate}
			\item\label{item:frakS-1} $S \subset \frakS$, and $S_{\ell}(Q) \subset \frakS$ for all $\ell \mid p|\Gamma|$,
%			\item\label{item:frakS-2} $S_{\ell}(Q) \subset \frakS$ for every prime $\ell \mid (p|\Gamma|)$,
			\item\label{item:frakS-3} $\B_{\frakS \backslash T}^{\frakS \cup T}(Q, A)=0$,
			\item\label{item:frakS-4} the set $\frakS \backslash (\bigcup_{\ell\mid p|\Gamma|}S_{\ell}(Q) \cup S \cup T)$ has cardinality 
			\[
				\frac{\dim_{\F_p} \B_{S\backslash T}^{S \cup T}(Q, A)}{\dim_{\F_p}\End_{\Gamma}(A)}.
			\]
		\end{enumerate}
	\end{lemma}
	
	\begin{proof}
		Let $L:=Q(A,\mu_p)$ and let $\Ram(Q(A)/Q)$ denote the set of primes of $Q$ ramified in $Q(A)/Q$.
		Let $T'=T \cup\Ram(Q(A)/Q)\cup S_p(Q)$. 			
		Consider the following diagram of $\Gal(L/Q)$-modules. 
		\[\begin{tikzcd}
			\prod\limits_{\frakP \in S \backslash T (L)} H^1(\calG_{\frakP}, A) \times \prod\limits_{\frakP \not \in S \cup T(L)} H_{\nr}^1(\calG_{\frakP}, A) \arrow[hook]{d} \arrow["\alpha"]{r} & H^1(G_L, A')^{\vee}\arrow[equal]{d} \arrow[two heads]{r} & \B_{S\backslash T(L)}^{S \cup T(L)}(L, A) \\
			\prod'\limits_{\frakP \not \in T(L)} H^1(\calG_{\frakP}, A) \times \prod\limits_{\frakP \in T'\backslash(S\cup T)(L)} H_{\nr}^1(\calG_\frakP, A)\arrow[two heads]{d} \arrow[two heads]{r} & H^1(G_L, A')^{\vee} & \\
			\bigoplus\limits_{\frakP \not \in S \cup T'(L)} H^1(\calT_{\frakP}, A)^{\calG_{\frakP}}
		\end{tikzcd}\]
		The first row is from definition of $\B_{S\backslash T(L)}^{S \cup T(L)}(L, A)$.
		In the second row, the product $\prod'_{\frakP \not \in T(L)} H^1(\calG_{\frakP},A)$ is the restricted product, consisting of all elements in $\prod_{\frakP \not \in T(L)} H^1(\calG_{\frakP},A)$ such that the image under the restriction map $H^1(\calG_{\frakP}, A) \to H^1(\calT_{\frakP},A)$ is nonzero at only finitely many primes $\frakP$.
		Since $\Gal(Q(A)/Q)$ is abelian, $L$ is an abelian extension of $Q$, so $H^1(G_L,A')\to \prod_{\frakP \not\in T'(L)}H^1(\calG_{\frakP}, A')$ is injective by \cite[Theorem~(9.1.15)(ii)]{NSW},
		 and therefore the second row is surjective. By the snake lemma, we obtain a surjection 
		\[
			\bigoplus_{\frakP \not \in S \cup T'(L)} H^1(\calT_{\frakP}, A)^{\calG_{\frakP}} \xtwoheadrightarrow{} \B_{S\backslash T(L)}^{S\cup T(L)}(L, A).
		\]

		By Lemma~\ref{lem:B}, 
		if $\B_{S\backslash T}^{S \cup T}(Q, A)\neq 0$, then there exists a prime $\frakp \not\in S \cup T'(Q)$ such that the image of $(\oplus_{\frakP \in \frakp(L)}H^1(\calT_{\frakP}, A)^{\calG_\frakP})^{\Gal(L/Q)}$ in $\B_{S \backslash T}^{S\cup T}(Q, A)$ is nontrivial. If we enlarge $S$ by including $\frakp$, then the cokernel of $\alpha$ gets smaller; in other words, the map $\beta: \B_{S \backslash T}^{S \cup T}(Q, A) \twoheadrightarrow \B_{S \cup\{\frakp\} \backslash T}^{S \cup \{\frakp\} \cup T}(Q,A)$ (obtained by taking $\Gal(L/Q)$-equivariant of the last map in the exact sequence in Lemma~\ref{lem:es-B} for $S_1=S(L)$, $S_2=S\cup \{\frakp\}(L)$ and $k=L$) has nontrivial kernel. By Lemma~\ref{lem:B-QvsL}, we have
		\begin{equation}\label{eq:beta>}
			\dim_{\F_p} \ker \beta \geq \dim_{\F_p}\End_{G_Q}(A^{\vee}).
		\end{equation}
		For every prime $\frakP$ of $L$ lying above $\frakp$, since $\frakp \not\in T'$, $\frakp$ is unramified in $Q(A)/Q$ and the residue characteristic of $\frakp$ is prime to $p$, so $\calT_{\frakp}$ acts trivially on $A$ and $\calT_{\frakP}/p\calT_{\frakP}$ is isomorphic to $\Z/p\Z$ as groups. Then
		\begin{equation}\label{eq:beta<}
			\dim_{\F_p} \left(\bigoplus_{\frakP \in \frakp(L)} H^1(\calT_{\frakP}, A)^{\calG_{\frakP}} \right)^{\Gal(L/Q)} = \dim_{\F_p} H^1(\calT_{\frakp}, A)^{\calG_{\frakp}}=\dim_{\F_p}\Hom_{\calG_{\frakp}}(\calT_{\frakp}, A) \leq \dim_{\F_p} A.
		\end{equation}
		Then by Lemma~\ref{lem:End=A}, \eqref{eq:beta<} and \eqref{eq:beta>}, we have $\dim_{\F_p} \ker \beta = \dim_{\F_p} \End_{\Gamma}(A^{\vee}) = \dim_{\F_p} A$.
		So including an appropriate prime in $S$ can reduce the $\dim_{\F_p} \B$ by at least $\dim_{\F_p}\End_{G_Q}(A^{\vee})=\dim_{\F_p}\End_{G_Q}(A)$. By repeating this process, and finally including $\bigcup_{\ell \mid p|\Gamma|} S_{\ell}(Q)$ in $\frakS$, we obtain a set $\frakS$ satisfying all of \eqref{item:frakS-1}, \eqref{item:frakS-3} and \eqref{item:frakS-4}. 
	\end{proof}

\section{Bounds of $A$-rank of $C_S^T(K)$}\label{sect:A-rk-bounds}

	Let $K$ be a $\Gamma$-extension of $Q$. In this section, we will estimate the $A$-rank of $C_S^T(K)$ for any simple $\F_p[\Gal(K/Q)]$-module $A$. For a group $G$ and an $\F_p[G]$-module $M$, denote
	\[
		h^i(G,M):= \dim_{\F_p} H^i(G,M).
	\]
	When $G$ is a subgroup of $H$, for an $\F_p[H]$-module $M$, $H$ acts on $H^i(G,M)$ by conjugation, and we denote
	\[
		h^i(G,M)^H:=\dim_{\F_p} H^i(G, M)^H.
	\]

	\begin{definition}\label{def:RA}
		Let $(K,\iota)$ be a $\Gamma$-extension of $Q$, and $A$ a simple $\F_p[\Gamma]$-module. We let $\scrR_A(K/Q)$ denote the set of primes of $Q$ satisfying the following conditions.
		\begin{enumerate}
			\item \label{item:dagger-0} The inertia subgroup $\calT_{\frakp}(K/Q)$ of $\Gal(K/Q)$ at $\frakp$ has order divisible by $p$.
			\item \label{item:dagger-3} $\frakp \not \in S_p(Q)$. 
			\item \label{item:RA} As a subgroup of $\Gamma$ via the isomorphism $\iota:\Gal(K/Q) \simeq \Gamma$, the decomposition subgroup $\calG_{\frakp}(K/Q)$ of $\Gal(K/Q)$ acts trivially on $A$.		
		\end{enumerate}
	\end{definition}
	
	Comparing this definition with Definition~\ref{def:dagger}, $\scrR_I$ is a subset of $\scrR_A$ for any ideal $I$ of $e\Z_p[\Gamma]$ when $e\Z_p[\Gamma]/\frakm_e \simeq A$.
	
	\begin{lemma}\label{lem:base_change}
		Let $(K,\iota)$ be a $\Gamma$-extension of $Q$ and $A$ a simple $\F_p[\Gamma]$-module. Then $A$ is an $\F_p[\Gal(K/Q)]$-module by $\iota: \Gal(K/Q) \to \Gamma$. Denote $L:=Q(A,\mu_p)$ and $\scrS_A:=S\cup \scrR_A(K/Q)$. Then there exists a constant $c_0$ depending on $\#T$, $Q$, $\Gamma$, $p$ and the $\Gamma$-module structure of $A$ such that
		\[
			|\rk_A C_S^T(K)-\frac{h^1(G^T_{\scrS_A}(L), A)^{\Gal(L/Q)}}{\dim_{\F_p} \End_{G_Q}(A)}| \leq c_0.
		\] 
	\end{lemma}
		
	\begin{proof}
		Let $D:=Q(A)$. So $D$ is contained in $K \cap L$ and $p\nmid [D:Q]$. By applying the Hochschild--Serre exact sequence to the short exact sequence $1 \to G_S^T(K) \to \Gal(K_S^T/D)\to \Gal(K/D) \to 1$ and the module $A$, we obtain an exact sequence of $\F_p[\Gal(D/Q)]$-modules
		\begin{equation}\label{eq:HS-D}
			H^1(\Gal(K/D), A) \hookrightarrow H^1(\Gal(K_S^T/D), A) \to H^1(G_S^T(K), A)^{\Gal(K/D)} \to H^2(\Gal(K/D), A).
		\end{equation}
		Because $p\nmid[D:Q]$, taking $\Gal(D/Q)$-invariant is an exact functor on $\F_p[\Gal(D/Q)]$-modules. So by taking $\Gal(D/Q)$-invariants on \eqref{eq:HS-D} it follows that 
		\begin{eqnarray}
			&&-h^1(\Gal(K/D), A)^{\Gal(D/Q)} \nonumber\\
			&\leq&h^1(G_S^T(K),A)^{\Gal(K/Q)}-h^1(\Gal(K_S^T/D),A)^{\Gal(D/Q)}  \nonumber\\
			&\leq& h^2(\Gal(K/D), A)^{\Gal(D/Q)} -h^1(\Gal(K/D), A)^{\Gal(D/Q)}, \label{eq:ineq-D1}
		\end{eqnarray}
		where both the first and the last lines are determined by the $\Gamma$-module structure of $A$. By a similar argument, one see that $h^1(G_{\scrS_A}^T(L), A)^{\Gal(L/Q)} - h^1(\Gal(L_{\scrS_A}^T/D), A)^{\Gal(D/Q)}$ is bounded (above and below) by constants determined by only the $\Gamma$-module structure of $A$. 
		
		Since the degree of $L=D(\mu_p)$ over $D$ is prime to $p$ and $\Gal(L/D)$ acts trivially on $A$, 
		\begin{eqnarray}
			H^1(\Gal(L_{\scrS_A}^T/D), A)^{\Gal(D/Q)} &=& \Hom_{\Gal(D/Q)}(\Gal(L_{\scrS_A}^T/D), A) \nonumber\\
			&=& \Hom_{\Gal(D/Q)}(\Gal(D_{\scrS_A}^T/D), A). \label{eq:ineq-D2}
		\end{eqnarray} 
		Let $F_D/D$ be the maximal abelian subextension of $D_{\scrS_A}^T/D$ such that $\Gal(F_D/D)$ is $\Gal(D/Q)$-equivariant isomorphic to a direct product of $A$. Let $E/D$ be the maximal abelian subextension of $K_S^T/D$ such that $\Gal(E/D)$ is $\Gal(D/Q)$-equivariant isomorphic to a direct product of $A$.
		In other words, $F_D$ (resp. $E$) is the subfield fixed by the intersection of kernels of all $\Gal(D/Q)$-equivariant surjections from $\Gal(D_{\scrS_A}^T)^{\ab}$ to $A$ (resp. from $\Gal(K_S^T/D)^{\ab}$ to $A$).

		Let $\Ram_p(K/D)$ be the set of primes of $D$ at which the inertia subgroup of $K/D$ has order divisible by $p$. Then by definition of $E$, we see that $E/D$ is unramified outside $S(D) \cup \Ram_p(K/D)$. 
		Let $\frakP \in \Ram_p(K/D)$ be a prime that is ramified in $E/D$, and assume $\frakP \not\in S_p(D)$. Because the inertia subgroup at a tamely ramified prime is cyclic, the inertia subgroups $\calT_{\frakP}(K/D)$ and $\calT_{\frakP}(E/D)$ are both cyclic. Then as $\Gal(E/D)$ is elementary abelian-$p$, any element of $\Gal(K/D)$ of order divisible by $p$ cannot be lifted to an element of $\Gal(EK/D)$ with larger order. Thus, $EK/K$ must be unramified at primes above $\frakP$, and equivalently, $\calT_{\frakP}(E/D)$ embeds into $\calT_{\frakP}(K/D)$. Let $\frakp$ be the prime of $Q$ lying below $\frakP$. Since $\Gamma$ is abelian, the conjugation action of $\calG_{\frakp}(K/Q)$ on $\calT_{\frakp}(K/Q)$ is trivial. Then we see that $\calG_{\frakp}(EK/Q)$ acts trivially on $\calT_{\frakp}(EK/Q)$, and hence $\calG_{\frakp}(K/Q)$ acts trivially on $A$ because $\calT_{\frakP}(E/D)\subset \Gal(E/D)\simeq A^{\oplus r}$ for some $r$. So we conclude that $\frakp \in \scrR_A(K/Q)$. In summary, we proved above that if a prime $\frakP$ is ramified in $E/D$ and $\frakP \not\in S_p(D) \cup S(D)$, then $\frakp \in \scrR_A(K/Q)$.
		
		So, $E/D$ is unramified outside $\scrS_A(D)\cup S_p(D)$. Thus, the quotient of $\Gal(E/D)$ by its decomposition subgroups at primes in $T(D)$ and inertia subgroups at primes in $S_p(D)$ is a quotient of $G_{\scrS_A}^T(D)$, and hence 
		\begin{equation}\label{eq:bound-E2}
			h^1(\Gal(E/D),A)^{\Gal(D/Q)} \leq h^1(\Gal(F_D/D),A)^{\Gal(D/Q)}+k_1\cdot \#T(D)+k_2\cdot \#S_p(D),
		\end{equation}
		where $k_1$ is the maximum of $\dim_{\F_p}\Hom(\calG_{\frakP},A)$ for $\frakP \in T(D)$ and $k_2$ is the maximum of $\dim_{\F_p}\Hom(\calT_{\frakP},A)$ for $\frakP\in S_p(D)$. Although $k_1$ and $k_2$ are defined in terms of the primes of $T(D)$ and $S_p(D)$, because the generator ranks of $\calG_{\frakP}(p)$ and $\calT_{\frakP}(p)$ are determined by the degree of $D_{\frakP}$ over the base local field ($\Q_\ell$ or $\F_q((t))$, depending on what $Q$ and $\frakP$ are) \cite[Theorems (7.5.3) and (7.5.11)]{NSW}, both $k_1$ and $k_2$ are bounded above by a constant depending on $Q$, $\Gamma$, $p$, and the module structure of $A$.
		
		Considering $F_D/D$, by the same reason, since $\Gal(F_D/D)$ is elementary abelian-$p$, any element of $\Gal(K/D)$ of order divisible by $p$ cannot be lifted to an element of $\Gal(F_DK/D)$ of larger order. If a prime $\frakP$ of $D$ is tamely ramified in both $F_D/D$ and $K/D$ such that $\calT_{\frakP}(K/D)$ has order divisible by $p$, then $F_DK/K$ is unramified at every prime above $\frakP$. 
		Therefore, by definition of $D$ and $\scrS_A$, $F_DK/K$ is unramified outside $S(K)\cup S_p(K)$ and splits completely at $T(K)$, which shows that after taking quotient of $\Gal(F_DK/D)$ by appropriate inertia subgroups of primes in $S_p(K)$ we obtain a subfield of $G_S^T(K)$. So 
		\begin{equation}\label{eq:bound-E3}
			h^1(\Gal(F_D/D),A)^{\Gal(D/Q)} \leq h^1(\Gal(F_DK/D),A)^{\Gal(D/Q)} \leq h^1(\Gal(E/D),A)^{\Gal(D/Q)} + k_3 \cdot \# S_p(K),
		\end{equation}
		where $k_3$ is the maximum of $\dim_{\F_p} \Hom(\calG_{\frakP},A)$ for $\frakP \in S_p(K)$, and $k_3$ and $\#S_p(K)$ are bounded above by constants depending on $Q$, $\Gamma$ and $p$.
		
		By \eqref{eq:bound-E2} and \eqref{eq:bound-E3}, 
		\begin{eqnarray*}
			&&h^1(\Gal(K_S^T/D), A)^{\Gal(D/Q)}-h^1(\Gal(D_{\scrS_A}^T/D), A)^{\Gal(D/Q)} \\
			&=&h^1(\Gal(E/D),A)^{\Gal(D/Q)}-h^1(\Gal(F_D/D),A)^{\Gal(D/Q)}
		\end{eqnarray*}
		is bounded above and below by constants depending on $\#T$, $Q$, $\Gamma$, $p$ and the $\Gamma$-module structure of $A$.
		Then the proposition follows by the argument from \eqref{eq:ineq-D1} to \eqref{eq:ineq-D2}, and the formula \eqref{eq:e-rk}.
	\end{proof}

	The following lemma generalizes \cite[Proposition~(10.7.2)]{NSW}.
	
	\begin{lemma}\label{lem:B-empty}
		Retain the notation from above and let $L$ be $Q(A,\mu_p)$. Then 
		\[
			\B_{\O}^{T(L)}(L, \F_p)^{\vee} \simeq \faktor{\calO_{L,T(L)}^{\times}}{\calO_{L,T(L)}^{\times p}} \oplus \Cl_{T(L)}(L)_{/p}
		\]
		as $\F_p[\Gal(L/Q)]$-modules.
	\end{lemma}
	
	\begin{proof}
	The group $\B_{\O}^{T(L)}(L, \F_p)$ is the pontryagin dual of $V_{\O}^{T}(L):=W/L^{\times p}$ with 
		\[
			W:=\left\{ a\in L^{\times} \, :\, a\in U_{\frakP}L_{\frakP}^{\times p} \text{ for all } \frakP \not \in T(L) \right\},
		\]
		where $U_{\frakP}$ is the group of units of $\calO_{L_{\frakP}}$.
		Consider the homomorphism
		\begin{eqnarray*}
			W &\longrightarrow& \Cl_{T(L)}(L)[p] \\
			a &\longmapsto& \fraka \text{  with }(a)=\fraka^p.
		\end{eqnarray*}
		This homomorphism is equivariant under the action by $\Gal(L/Q)$, and induces a map from $V_{\O}^T(L) \to \Cl_{T(L)}(L)[p]$ with kernel equal to $\calO_{L,T(L)}^{\times} / \calO_{L,T(L)}^{\times p}$. The lemma follows since $\F_p[\Gal(L/Q)]$ is semisimple and $\Cl_{T(L)}(L)[p]\simeq_{\Gal(L/Q)} \Cl_{T(L)}(L)_{/p}$.
	\end{proof}

	\begin{proposition}\label{prop:e-rk}
		Retain the notation from above. There exists a constant $c_1$ depending on $\Gamma$, $p$, $Q$, $S$, $T$ and the $\Gamma$-module structure of $A$ such that
		\[
			|\rk_A C_S^T(K) - \frac{\dim_{\F_p} \B_{\scrS_A \backslash T}^{\scrS_A\cup T}(Q, A)+ \sum\limits_{\frakp \in \scrS_A \backslash T} h^1(\calT_{\frakp}, A)^{\calG_{\frakp}} }{\dim_{\F_p}\End_{\Gamma}(A)}| \leq c_1.
		\]
	\end{proposition}
	
	\begin{proof}
		Applying Lemma~\ref{lem:es-B} to $S_1=\O$, $S_2=\scrS_A$, and $k=L$ gives the $\Gal(L/Q)$-equivariant sequence
		\begin{equation}\label{eq:es-B-L}
			H^1(G_{\O}^T(L), A) \hookrightarrow H^1(G_{\scrS_A}^T(L), A) \to \bigoplus_{\frakP \in \scrS_A \backslash T(L)} H^1(\calT_{\frakP}, A)^{\calG_{\frakP}} \to \B_{\O}^{T(L)}(L, A) \twoheadrightarrow \B_{\scrS_A \backslash T(L)}^{\scrS_A \cup T(L)}(L, A).
		\end{equation}
		Since $p\nmid[L:Q]$, taking $\Gal(L/Q)$-invariants is an exact functor, so we obtain an exact sequence of $\F_p$-modules after taking $\Gal(L/Q)$-invariants of \eqref{eq:es-B-L}. Note that 
		\[
			H^1(G_{\O}^T(L),A)^{\Gal(L/Q)}=\Hom_{\Gal(L/Q)}(G_{\O}^T(L), A)=\Hom_{\Gal(L/Q)}(\Cl_{T(L)}(L), A).
		\]
		For each prime $\frakp$ of $Q$, let $\frakp(L)$ denote the primes of $L$ above $\frakp$. 
		Because $p\nmid [L:Q]$, for any $\frakP \in \frakp(L)$, $\calT_{\frakP}$ is a normal subgroup of $\calT_{\frakp}$ of index prime to $p$, so by the Hochschild--Serre exact sequence, we have
		\[
			H^1(\calT_{\frakp}, A) \simeq H^1(\calT_{\frakP}, A)^{\calT_\frakp}.
		\]
		Therefore,
		\[
			\left(\bigoplus_{\frakP \in \frakp(L)} H^1(\calT_{\frakP}, A)^{\calG_{\frakP}} \right)^{\Gal(L/Q)}= \left(H^1(\calT_{\frakP}, A)^{\calG_{\frakP}}\right)^{\calG_{\frakp}(L/Q)} = H^1(\calT_{\frakP}, A)^{\calG_{\frakp}} =H^1(\calT_{\frakp}, A)^{\calG_{\frakp}}.
		\]
		By Lemmas~\ref{lem:B}
, \ref{lem:B-QvsL} and \ref{lem:B-empty}, we have $\B_{\scrS_A\backslash T(L)}^{\scrS_A \cup T(L)}(L,A)^{\Gal(L/Q)} \simeq \B_{\scrS_A \backslash T}^{\scrS_A \cup T}(Q, A)$ and 
		\begin{eqnarray*}
			\B_{\O}^{T(L)}(L,A)^{\Gal(L/Q)} &\simeq& \B_{\O}^T(Q, A) \\
			 &\simeq& \Hom_{\Gal(L/Q)}\left(\B_{\O}^{T(L)}(L, \F_p), A^{\vee} \right) \\
			&=& \Hom_{\Gal(L/Q)}\left(\Cl_{T(L)}(L),A\right) \oplus \Hom_{\Gal(L/Q)}\left(\calO_{L,T(L)}^{\times}, A\right).
		\end{eqnarray*}
		Now we have evaluated the $\Gal(L/Q)$-invariants of terms in \eqref{eq:es-B-L}, from which we have
		\begin{eqnarray}
			h^1(G_{\scrS_A}^T(L), A)^{\Gal(L/Q)} &=& \dim_{\F_p} \B_{\scrS_A \backslash T}^{\scrS_A \cup T}(Q, A) + \sum_{\frakp \in \scrS_A \backslash T} h^1(\calT_{\frakp}, A)^{\calG_{\frakp}}  \nonumber \\
			&& - \dim_{\F_p} \Hom_{\Gal(L/Q)}\left( \calO_{L,T(L)}^{\times}, A \right). \label{eq:calO}
		\end{eqnarray}
		Note that $[L:Q]$ can be bounded from above by a constant depending on only $\Gamma$ and $Q$, but not on the choice of $K$ and how $G_Q$ acts on $A$. So, the last term in \eqref{eq:calO}, which is at most $\dim_{\F_p} (\calO^{\times}_{L,T(L)}/\calO^{\times p}_{L,T(L)}) \cdot \dim_{\F_p}A$, can be bounded a constant depending only on $\Gamma$, $Q$, $T$ and the module structure of $A$. 
		Finally, the lemma follows from Lemma~\ref{lem:base_change}.
	\end{proof}

	\begin{proposition}\label{prop:e-rk-arbitrary}
		For any set $\scrS$ satisfying $S\subseteq \scrS \subseteq \scrS_A$, 
		\[
			\rk_A C_S^T(K) \geq \frac{\dim_{\F_p} \B_{\scrS \backslash T}^{\scrS\cup T}(Q, A)+ \sum\limits_{\frakp \in \scrS \backslash T} h^1(\calT_{\frakp}, A)^{\calG_{\frakp}} }{\dim_{\F_p}\End_{\Gamma}(A)} - c_1,
		\]
		where $c_1$ is the constant in Proposition~\ref{prop:e-rk}.
	\end{proposition}
	
	\begin{proof}
		Repeating the proof of Lemma~\ref{lem:base_change} by replacing $\scrS_A$ with $\scrS$, one see that the inequality \eqref{eq:bound-E2} still holds (but \eqref{eq:bound-E3} might fail), so
		\[
			\rk_AC_S^T(K) \geq \frac{h^1(G_{\scrS_A}^T(L), A)^{\Gal(L/Q)}}{\dim_{\F_p}\End_{G_Q}(A)}-c_0.
		\]
		Then following the proof of Proposition~\ref{prop:e-rk}, one obtain the lower bound for $\rk_AC_S^T(K)$ in the proposition.
	\end{proof}

\section{Embedding Problems and Presentations}\label{sect:presentation}

\subsection{Embedding problems}
\hfill
	
	\begin{lemma}\label{lem:EP}
		Let $k$ be a finite Galois extension of $Q$ and $p$ a prime number such that $p\neq \Char(Q)$. Let $\rho:\widetilde{G} \to G$ be a surjection of profinite groups such that $M:=\ker \rho$ is a finite abelian $p$-group, and let $\varphi: G_Q \to G$ be a homomorphism. For each prime $\frakp$ of $Q$, let $\varphi_{\frakp}$ be defined by restricting $\varphi$ to $\calG_{\frakp}$. Consider the global and local embedding problems below.
		\[\begin{tikzcd}
			 & & G_Q \arrow["\varphi"]{d} \arrow[dashed, "\psi"']{dl} \\
			 M \arrow[hook]{r} & \widetilde{G} \arrow["\rho", two heads]{r} & G
		\end{tikzcd}
		\quad \quad \quad \quad
		\begin{tikzcd}
			& & \calG_{\frakp} \arrow["\varphi_{\frakp}"]{d} \arrow[dashed, "\psi_{\frakp}"']{dl}\\
			M \arrow[hook]{r} & \widetilde{G} \arrow["\rho", two heads]{r} & G
		\end{tikzcd}\]
		Assume that $M$ is a simple $\F_p[G_Q]$-module, where the $G_Q$-action on $M$ is defined via $\varphi$ and the conjugation of $\widetilde{G}$. Let $S$ be a set of primes of $Q$ such that $M$, with the above $G_Q$-action, satisfies $\B_{S}^{S}(Q, M)=0$.
		If 
		\begin{enumerate}
			%\item\label{item:EP-0} the supernatural number $|\widetilde{G}|$ is coprime to $\Char Q$,
			\item \label{item:EP-1} $\varphi$ factors through $\Gal(k_S/Q)$,
			\item \label{item:EP-2} $\psi_{\frakp}$ in the right diagram exists for every $\frakp \in S$, and
			\item \label{item:EP-3} when $Q$ is a number field, $S_\ell(Q) \subset S$ for every $\ell \mid p[k:Q]$,
		\end{enumerate}
		then there exists a map $\psi$ in the left diagram that factors through $\Gal(k_S/Q)$.
	\end{lemma}
	
	\begin{proof}
		By definition, $\B_{\{\text{all primes}\}}^{\{\text{all primes}\}}(Q,M)$ is a quotient of $\B_S^S(Q,M)$, so it is 0; and then the Shaferevich group $\Sha^2(Q,M)=0$ by \cite[Proposition~8.5]{Liu-presentation}. 
		By \cite[Lemma~3.7]{Liu2022b}, there exists a map $\psi: G_Q\to \widetilde{G}$ fitting into the left diagram if and only if the map $\psi_\frakp$ exists for every prime $\frakp$ of $Q$.
		
		We first show the existence of $\psi_{\frakp}$ for every $\frakp \not\in S$. If $\varphi_{\frakp}$ is unramified, then $\varphi_{\frakp}$ factors through $\calG_{\frakp}/\calT_{\frakp} \simeq \hat{\Z}$, and it can always be lifted to a map $\hat{\Z} \to \widetilde{G}$, which gives an unramified $\psi_{\frakp}$ fitting into the right diagram. Suppose $\varphi_{\frakp}$ is ramified for some $\frakp \not\in S$. By the condition~\eqref{item:EP-1}, any prime of $k$ above $\frakp$ is unramified in the field $\overline{Q}^{\ker \varphi}$. Then it follows by the condition~\eqref{item:EP-3} that $\frakp$ is tamely ramified in $k/Q$ and $\varphi_{\frakp}(\calT_{\frakp})$ has order pro-prime-to-$p$. By the result of Iwasawa \cite{Iwasawa}, $\varphi_{\frakp}(\calG_{\frakp})$ can be generated by two elements $t, s \in G$ such that 
		\begin{equation}\label{eq:Iwa}
			sts^{-1}=t^{\Nm(\frakp)}
		\end{equation}
		and the cyclic subgroup generated by $t$ is $\varphi_{\frakp}(\calT_{\frakp})$. Since $p\nmid |t|$ and $M$ is elementary abelian-$p$, there exists $\tilde{t} \in \rho^{-1}(t)$ such that $|\tilde{t}|=|t|$. Let $x\in \widetilde{G}$ be an element of $\rho^{-1}(s)$. By \eqref{eq:Iwa},
		\begin{equation}\label{eq:Iwa-lift}
			x \tilde{t} x^{-1} = {\tilde{t}}^{\,\Nm(\frakp)} m,
		\end{equation}
		for some $m \in M$. Since $|x\tilde{t} x^{-1}|=|\tilde{t}|=|t|=|sts^{-1}|=|t^{\Nm(\frakp)}|$, we see that $\tilde{t}^{\, \Nm(\frakp)}$ and $\tilde{t}^{\, \Nm(\frakp)} m$ have the same order that is prime to $|M|$, so by the Schur--Zassenhaus theorem, $\tilde{t}^{\, \Nm(\frakp)}$ and $\tilde{t}^{\, \Nm(\frakp)}$ are conjugate, i.e., there exists $g \in M$ such that $g \tilde{t}^{\, \Nm(\frakp)} m g^{-1}= \tilde{t}^{\, \Nm(\frakp)}$. Then \eqref{eq:Iwa-lift} implies
		\[
			(gx) \tilde{t} (gx)^{-1} = \tilde{t}^{\, \Nm(\frakp)},
		\]
		thus $\tilde{t}$ and $\tilde{s}:=gx$ give lifts of $t$ and $s$ that satisfies the relator in the presentation of the Galois group of maximal tamely ramified extension given in \cite{Iwasawa}. So the subgroup of $\widetilde{G}$ generated by $\tilde{t}$ and $\tilde{s}$ defines a lift $\psi_{\frakp}$ of $\varphi_{\frakp}$.
		
		From the argument above, we see the condition \eqref{item:EP-2} in the lemma implies the existence of $\phi:G_Q \to \widetilde{G}$ such that $\rho \circ \phi= \varphi$. Next, we will show that the conditions \eqref{item:EP-1} and \eqref{item:EP-3} imply that there exists a 1-cocycle $\delta: G_Q \to M$ such that the group homomorphism, which is the twist of $\phi$ by $\delta$,
		\begin{eqnarray*}
			^{\delta}\phi : G_Q &\longrightarrow& \widetilde{G} \\
			g &\longmapsto& \delta(g) \phi(g)
		\end{eqnarray*}
		factors through $\Gal(k_S/Q)$. For each prime $\frakp$ of $Q$, let $\phi_{\frakp}: \calG_{\frakp} \to \widetilde{G}$ denote the composition of $\calG_{\frakp} \hookrightarrow G_Q$ and $\phi$. Consider a prime $\frakp \not\in S$, and pick a prime $\frakP$ of $\overline{Q}^{\ker \varphi}$ lying above $\frakp$. Let $-\phi_{\frakp}$ be the map from $\calG_{\frakp} \to \widetilde{G}$ such that $\phi_{\frakp}(x)^{-1}=-\phi_{\frakp}(x)$ for every $x \in \calG_{\frakp}$. The restriction of $-\phi_{\frakp}$ to $\calG_{\frakP}$ gives a 1-cocycle $\delta_{\frakP}$ in $H^1(\calG_{\frakP}, M)^{\calG_{\frakp}}$, and its further restriction to $\calT_{\frakP}$ gives a 1-cocycle $\delta_{\frakP} |_{\calT_{\frakP}}$ in $H^1(\calT_{\frakP}, M)^{\calG_{\frakp}}$. Recall that we showed, because of the conditions~\eqref{item:EP-1} and \eqref{item:EP-3}, if $\varphi_\frakp$ is ramified, then it has to be tamely ramified. So $\calT_{\frakP}$ is a subgroup of $\calT_{\frakp}$ of index not divisible by $p$. So by \cite[Corollary~(2.4.2)]{NSW},
		\[
			H^i(\calG_{\frakp}/\calT_{\frakp}, M^{\calT_{\frakp}}) \overset{\sim}{\longrightarrow} H^i(\calG_{\frakp}/\calT_{\frakP}, M^{\calT_{\frakP}}), \quad \text{for }i>0.
		\]
		Then we have the following commutative diagram
		\[\begin{tikzcd}
			0 \arrow{r} & H^1(\calG_{\frakp}/\calT_{\frakp}, M^{\calT_{\frakp}}) \arrow{r} \arrow["\sim"]{d} & H^1(\calG_{\frakp}, M) \arrow{r}\arrow[equal]{d} & H^1(\calT_{\frakp}, M)^{\calG_{\frakp}} \arrow[dashed, "\sim"]{d} \arrow{r} & 0 \\
			0 \arrow{r} & H^1(\calG_{\frakp}/\calT_{\frakP}, M^{\calT_{\frakP}}) \arrow{r} & H^1(\calG_{\frakp}, M) \arrow{r} & H^1(\calT_{\frakP}, M)^{\calG_{\frakp}} \arrow{r} & 0,
		\end{tikzcd}\]
		where the rows are inflation-restriction exact sequences, and the last entries are zero because $H^2(\calG_{\frakp}/\calT_{\frakP}, M^{\calT_{\frakP}}) \simeq H^2(\calG_{\frakp}/\calT_{\frakp}, M^{\calT_{\frakp}})=0$ as $\calG_{\frakp}/\calT_{\frakp}\simeq \hat{\Z}$. From the diagram, we see that the right dashed arrow exists and is an isomorphism. 
		Via this isomorphism, we consider 
		\[
			\prod_{\frakp \not \in S} \delta_{\frakP}|_{\calT_{\frakP}} \in \bigoplus_{\frakp \not \in S} H^1(\calT_{\frakP}, M)^{\calG_{\frakp}} \overset{\sim}{\longrightarrow} \bigoplus_{\frakp \not \in S} H^1(\calT_{\frakp}, M)^{\calG_{\frakp}}.
		\]
		By the assumption $\B_S^S(Q,M)=0$ and \cite[Lemma~3.3]{Liu2022b}, there exists $\delta \in H^1(G_Q, M)$ such that the restriction of $\delta$ induced by $\calT_{\frakP} \hookrightarrow \calG_{\frakp} \hookrightarrow G_Q$ is $\delta_{\frakP}|_{\calT_{\frakP}}$for all $\frakp \not\in S$. Then $\calT_{\frakP} \subset \ker {^{\delta}\phi}$ by our construction of $\delta_{\frakP}$, so the map $^{\delta}\phi$ gives a lift of $\phi$ that does not further ramified at $\frakP$. This holds for all primes outside $S$, so $^{\delta}\phi$ fits into the global diagram in the lemma and factors through $\Gal(k_S/Q)$, and then the proof is completed.
	\end{proof}

\subsection{Maximal split subextension}
\hfill

	In this subsection, we study the basic properties of the maximal split subextensions for a given group extension, which will be used in the proof of the main theorems later.
	
	\begin{definition}\label{def:max-split}
		Given a profinite group extension 
		\begin{equation}\label{eq:ses}
			1 \longrightarrow M \longrightarrow \widetilde{G} \longrightarrow G \longrightarrow 1
		\end{equation}
		and a normal subgroup $N$ of $\widetilde{G}$ that is contained in $M$, we say \emph{$N$ defines a maximal split subextension of \eqref{eq:ses}} if the group extension 
		\[
			1 \longrightarrow M/N \longrightarrow \widetilde{G}/N \longrightarrow G \longrightarrow 1
		\]
		splits, and for any proper subgroup $N_0\subsetneq N$ that is normal in $\widetilde{G}$, the group extension 
		\[
			1 \longrightarrow M/N_0 \longrightarrow \widetilde{G}/N_0 \longrightarrow G \longrightarrow 1
		\]
		is nonsplit.
	\end{definition}
	
	\begin{lemma}\label{lem:max-split-abelian}
		Consider the extension \eqref{eq:ses} and let $\rho$ denote the surjection $\widetilde{G}\to G$. Assume $M$ is abelian. Then a normal subgroup $N$ of $\widetilde{G}$ defines a maximal split extension if and only if there exists a subgroup $H$ of $G$ such that $N=H\cap M$, $\rho(H)=G$, and the group extension $N \hookrightarrow H \twoheadrightarrow G$ defined by $\rho|_{H}$ is completely nonsplit (that is, if $\rho(E)=G$ for a subgroup $E\subset H$, then $E=H$). 
	\end{lemma}

	\begin{proof}
		For a normal subgroup $N$ of $\widetilde{G}$, if the group extension $M/N\hookrightarrow \widetilde{G}/N \twoheadrightarrow G$ splits, then let $H$ be the full preimage of the subgroup $G$ of $\widetilde{G}/N$ (defined by a splitting) under the quotient map $\widetilde{G}\to \widetilde{G}/N$, and we have $N=H\cap M$ and $\rho(H)=G$. On the other hand, suppose $H$ is a subgroup of $G$ such that $\rho(H)=G$. Let $N=H\cap M$. Note that the conjugation action of $\widetilde{G}$ on $M$ factors through $G$ because $M$ is abelian. This $G$-action preserves $N$ because $\rho(H)=G$. Then $N$ is normal and $H/N$ defines a section of $\widetilde{G}/N \twoheadrightarrow G$, so $\widetilde{G}/N \twoheadrightarrow G$ splits. So we showed that $N$ defines a split subextension if and only if there exists $H\subset G$ such that $N=H\cap M$ and $\rho(H)=G$. Therefore, $N=H\cap M$ defines a maximal split subextension if and only if $H$ does not contain any proper subgroup $E$ such that $\rho(E)=G$.
	\end{proof}
	
	\begin{lemma}\label{lem:rank-maxsplit}
		Consider \eqref{eq:ses}, and assume $G=\Gamma$ is finite abelian and $M$ is a finitely generated abelian pro-$p$ group. Assume a normal subgroup $N \subset \widetilde{G}$ defines a maximal split subextension of \eqref{eq:ses}. Let $A$ be a simple $\F_p[\Gamma]$-module.
		\begin{enumerate}
			\item If $A\neq\F_p$, then $\rk_A N=0$ and $\rk_A M/N=\rk_A M$.
			\item If $A=\F_p$, then $\rk_{A} N \leq h^2(\Gamma, \F_p)$ and $\rk_{A} M/N \geq \rk_{A} M - h^2(\Gamma, \F_p)$.
		\end{enumerate}
	\end{lemma}

	\begin{proof}
		Recall that, by Lemma~\ref{lem:max-ideal}, $\Gamma_p$ acts trivially on $A$ and hence $A$ is a simple $\F_p[\Gamma']$-module. By the Hochschild--Serre spectral sequence (for example \cite[Corollary~(2.4.2)]{NSW}), since $H^i(\Gamma', A)=0$ for $i>0$, we have 
		\begin{equation}\label{eq:HS-isom}
			H^i(\Gamma_p, A^{\Gamma'}) \simeq H^i(\Gamma, A) \quad \text{for all }i.
		\end{equation}
		Let $H$ be as described in Lemma~\ref{lem:max-split-abelian}, and then $N \hookrightarrow H \to \Gamma$ is a completely nonsplit extension.
		
		Assume $\Gamma$ acts nontrivially on $A$. Then $A^{\Gamma'}=A^{\Gamma}=1$, and it follows by \eqref{eq:HS-isom} that $H^2(\Gamma, A)=0$. So $H^2(\Gamma,A)=0$ implies that $\rk_A N=0$, and 
		\begin{equation}\label{eq:HSses-split}
			0 \longrightarrow H^1(M/N, A)^{\Gamma} \longrightarrow H^1(M, A)^{\Gamma} \longrightarrow H^1(N, A)^{\widetilde{G}}
		\end{equation}
		implies $\Hom_{\Gamma}(M/N, A)\simeq \Hom_{\Gamma}(M,A)$, and hence $\rk_A M/N=\rk_A M$ follows by \eqref{eq:e-rk}.
		
		Assume $A=\F_p$. The exact sequence $N \hookrightarrow H \twoheadrightarrow \Gamma$ implies
		\[
			0 \longrightarrow H^1(\Gamma, \F_p) \longrightarrow H^1(H,\F_p) \longrightarrow H^1(N,\F_p)^{H} \longrightarrow H^2(\Gamma, \F_p).
		\]
		Since $N \hookrightarrow H \twoheadrightarrow \Gamma$ is completley nonsplit, $h^1(\Gamma,\F_p)=h^1(H,\F_p)$, so $\rk_{A} N \leq h^2(\Gamma, \F_p)$. Finally, the last inequality in the lemma follows by \eqref{eq:HSses-split} for $A=\F_p$.
	\end{proof}

\subsection{Presentations of maximal split subextensions of $\Gal(E_S^T(K)/Q)\to \Gal(K/Q)$.}\label{ss:presentaiton-split}
\hfill

	Throughout this subsection, we fix a simple $\F_p[\Gal(K/Q)]$-module $A$ and a finite set $\scrS$ of primes of $Q$ such that $S\subset \scrS$, and let $R$ be a quotient ring of $\Z_p[\Gamma]$ such that every composition factor of $R$ is isomorphic to $A$ and $\rk_A R=1$. Later in Section~\ref{sect:proof-main}, we will apply the results in this section to $R=P_A$ and $R=e\Z_p[\Gamma]$ for $e\in\Idem(A)$.
	
	Let $S$ and $T$ be the sets in Theorem~\ref{thm:lb-rank}. Recall $E_S^T(K)$ and $C_S^T(K)$ defined in Section~\ref{sec:MainResults}. Let 
	\[
		RC_S^T:=C_S^T(K) \otimes_{\Z_p[\Gamma]} R.
	\]
	Because $R$ is a quotient ring of $\Z_p[\Gamma]$, $RC_S^T$ is a $\Gamma$-equivariant quotient of $C_S^T(K)$. We define
	\[
		RE_S^T:=E_S^T(K)^{\ker(C_S^T(K)\to RC_S^T)},
	\]
	so $RE_S^T$ is the extension of $K$ with Galois group $RC_S^T$.

	By Lemma~\ref{lem:frakS}, there exists a set $\frakS$ of primes of $Q$ such that
	\begin{enumerate}
		\item $\scrS \subset \frakS$,
		\item $\B_{\frakS \backslash T}^{\frakS \cup T}(Q, A)=0$, and
		\item $\# \frakS \backslash (\cup_{\ell \mid (p|\Gamma|)} S_{\ell}(Q) \cup \scrS \cup T) = \dfrac{\dim_{\F_p} \B_{\scrS \backslash T}^{\scrS \cup T}(Q,A)}{\dim_{\F_p} \End_{\Gamma}(A)}$.
	\end{enumerate}
	We pick and then fix such a set $\frakS$. The motivation for defining $\frakS$ is: we want to enlarge the set $\scrS$ by including sufficiently many primes to make $\B_{\frakS}^{\frakS}(Q,A)$ zero, so that we can apply the embedding problem result Lemma~\ref{lem:EP}. 
	
	Define $RC_{\frakS}$ and $RE_{\frakS}$, by replacing $S$ with $\frakS$ and $T$ with $\O$ in the definition of $RC_S^T$ and $RE_S^T$. 
	Then consider the short exact sequence
	\begin{equation}\label{eq:RC-ses}
		1\longrightarrow RC_{\frakS} \longrightarrow \Gal(RE_{\frakS}/Q) \longrightarrow \Gal(K/Q) \longrightarrow 1,
	\end{equation} 
	and choose a normal subgroup $N$ of $\Gal(RE_{\frakS}/Q)$ that defines a maximal split subextension of \eqref{eq:RC-ses}. We denote by
	\[
		R\mathcal{C}_{\frakS}:=RC_{\frakS}/N \quad \text{and} \quad R\calE_{\frakS}:=(RE_{\frakS})^{N},
	\]
	and then by Definition~\ref{def:max-split} we have a split short exact sequence.
	\begin{equation}\label{eq:split-RC-ses}
		1 \longrightarrow R\mathcal{C}_{\frakS} \longrightarrow \Gal(R\mathcal{E}_{\frakS}/Q) \longrightarrow \Gal(K/Q) \longrightarrow 1.
	\end{equation}
	Note that $RC_S^T$ is a $\Gal(K/Q)$-equivariant quotient of $RC_{\frakS}$. By Lemma~\ref{lem:max-split-abelian}, one can check that the image of $N$ in $RC_S^T$ defines a maximal split subextension of 
	\begin{equation}\label{eq:RC-ses-little}
		1\longrightarrow RC_S^T \longrightarrow \Gal(RE_S^T/Q)\longrightarrow \Gal(K/Q) \longrightarrow 1.
	\end{equation}
	So we define
	\[
		R\calC_S^T:=\faktor{RC_{\frakS}}{N\ker (RC_{\frakS} \to RC_S^T)} \quad \text{and}\quad R\calE_S^T:=(RE_{\frakS})^{N\ker (RC_{\frakS} \to RC_S^T)},
	\]
	and then obtain a maximal split subextension of \eqref{eq:RC-ses-little}
	\begin{equation}\label{eq:split-RC-ses-little}
		1 \longrightarrow R\calC_S^T \longrightarrow \Gal(R\calE_S^T/Q) \longrightarrow \Gal(K/Q) \longrightarrow 1.
	\end{equation}
	The goal of this subsection is to give presentations of $R\calC_\frakS$  and $R\calC_S^T$ using the local relators (relators in terms of only local information such as inertia subgroups and Frobenius elements). 
	
	Let 
	\[
		r:=\rk_A R\calC_{\frakS}.
	\]
	Because \eqref{eq:split-RC-ses} splits, there exists a surjective group homomorphism
	\begin{equation}\label{eq:def-kappa}
		\kappa: R^{\oplus r} \rtimes \Gamma \xtwoheadrightarrow{} R\calC_{\frakS} \rtimes \Gal(K/Q) \simeq \Gal(R\calE_{\frakS}/Q),
	\end{equation}
	whose restriction to $\Gamma$ is the inverse of the chosen isomorphism $\iota: \Gal(K/Q)\overset{\sim}{\to} \Gamma$ for the $\Gamma$-extension $(K,\iota)$.
	
		For a prime $\frakp$ of $Q$, if $\frakp$ is tamely ramified or unramified in $R\calE_{\frakS}/Q$, then, by \cite{Iwasawa}, we let $t_{\frakp}, s_{\frakp} \in \Gal(R\calE_{\frakS}/Q)$ denote a set of generators of $\calG_{\frakp}(R\calE_{\frakS}/Q)$ such that $t_{\frakp}$ generates $\calT_{\frakp}(R\calE_{\frakS}/Q)$, $s_{\frakp}$ is a Frobenius element, and $t_{\frakp}$ and $s_{\frakp}$ are compatible in the sense that
	\begin{equation}\label{eq:tame-rel}
		s_{\frakp} t_{\frakp} s_{\frakp}^{-1} = t_{\frakp} ^{\Nm(\frakp)}.
	\end{equation}
	(So $t_{\frakp}$ is trivial when $\frakp$ is unramified.) We fix a choice of preimages
	\[
		x_{\frakp} \in \kappa^{-1}(t_{\frakp}) \quad \text{and} \quad y_{\frakp} \in \kappa^{-1}(s_{\frakp}).
	\]

	\begin{proposition}\label{prop:pres-scrS}
		There exists a constant $c_2$ depending on $\Gamma$, $p$, $Q$, $R$ and the $\Gamma$-module structure of $A$ such that $\ker \kappa$ is the smallest closed normal subgroup of $R^{\oplus r} \rtimes \Gamma$ containing elements of the following types:
		\begin{itemize}
			\item {\bf Tame Type:} 
			\[
				x_{\frakp}^{\Nm(\frakp)} y_{\frakp} x_{\frakp}^{-1} y_{\frakp}^{-1}
			\]
			for each prime $\frakp \in \frakS \backslash (\cup_{\ell | (p|\Gamma|)} S_{\ell}(Q))$, and
			
			\item {\bf Wild Type:} additionally at most $c_2$ elements.
		\end{itemize}
	\end{proposition}

	\begin{proof}
		Let $\varphi_{\frakp}: \calG_{\frakp} \to \Gal(R\calE_{\frakS}/Q)$ denote the composition of the local inclusion $\calG_{\frakp} \hookrightarrow G_Q$ and $\varphi: G_Q \twoheadrightarrow \Gal(R\calE_{\frakS}/Q)$.
		When $\frakp \in \frakS \backslash ( \cup_{\ell |(p|\Gamma|)} S_{\ell}(Q))$, $R\calE_{\frakS}/Q$ must be tamely ramified or unramified at $\frakp$, so the map $\varphi_{\frakp}: \calG_{\frakp} \to \Gal(R\calE_{\frakS}/Q)$ factors through the Galois group of the maximal tamely ramified extension of $Q_{\frakp}$. Because $t_{\frakp}$ and $s_{\frakp}$ satisfy the relation \eqref{eq:tame-rel}, we obtain the relation of tame type as described in the lemma
		\[
			x_{\frakp}^{\Nm(\frakp)} y_{\frakp} x_{\frakp}^{-1} y_{\frakp}^{-1} \in \ker \kappa.
		\]
		Define $M$ be to the smallest closed normal subgroup of $R^{\oplus r} \rtimes \Gamma$ containing all the elements of tame type. If $M=\ker\kappa$, then we are done. Otherwise, $M\subsetneq \ker \kappa$, and we let $M_1$ be the smallest closed normal subgroup of $R^{\oplus r}\rtimes \Gamma$ such that $M \subset M_1\subset \ker \kappa$ and $\ker \kappa/M_1\simeq_{\Gamma} A^{\oplus d}$ for some integer $d$; equivalently, $M_1:=\cap_{\alpha} \ker \alpha$ where $\alpha$ varies in $\Hom_{\Gamma}(\ker \kappa/M ,A)$. 

	For each $\frakp \in \cup_{\ell \mid (p|\Gamma|)} S_{\ell}(Q)$, $\frakp$ can be wildly ramified in $R\calE_{\frakS}/Q$, and we will define a submodule $N_{\frakp}$ of $R^{\oplus r} /M_1$ as follows. First, $\kappa$ and $M_1$ define the short exact sequence below, in which we denote the surjection by $\varrho$.
	\begin{equation}\label{eq:ses-varrho}
		1 \longrightarrow \ker\kappa/M_1 \longrightarrow (R^{\oplus r} \rtimes \Gamma)/M_1 \overset{\varrho}{\longrightarrow} \Gal(R\calE_{\frakS}/Q) \longrightarrow 1
	\end{equation}
	The local Galois group $\calG_{\frakp}(R\calE_{\frakS}/Q)$ is $\im \varphi_{\frakp}$. Let $\frakP$ be a prime of $K$ lying above $\frakp$. By \cite[Theorem~(7.5.11)]{NSW} if $\frakp \in S_p(Q)$ and by \cite[Theorem~(7.5.3)]{NSW} if $\frakp\not\in S_p(Q)$, the pro-$p$ completion of $\calG_{\frakP}$ is finitely generated whose generator rank is bounded above by $Q$ and the size of $|\Gamma|$, so $d_{\frakp}:=\rk_A\calG_{\frakP}(R\calE_{\frakS}/K)$ is bounded above. Let $\gamma_{\frakp, 1}, \gamma_{\frakp,2},\ldots, \gamma_{\frakp,d_{\frakp}}$ be a minimal set of generators of the $R$-module $\calG_{\frakP}(R\calE_{\frakS}/K)$. For each $i=1, \ldots, d_{\frakp}$, pick a preimage $\widetilde{\gamma}_{\frakp, i} \in \varrho^{-1}(\gamma_{\frakp, i})$, then define $N_{\frakp}$ to be the submodule of $R^{\oplus r}/M_1$ generated by $\widetilde{\gamma}_{\frakp, 1}, \ldots, \widetilde{\gamma}_{\frakp, d_{\frakp}}$. 
	
	We claim that the submodule of an $R$-module $M$ generated by one (arbitrary) element $x \in M$ is a quotient module of $R$ (i.e., a one-generated $R$-module has $A$-rank at most 1). To see this, by the Nakayama's lemma, it suffices to show that $A^{\oplus n}$ cannot be generated by one element when $n\geq 2$, and this follows by \cite[Remark~5.2]{Liu-Wood} and Lemma~\ref{lem:End=A}.
	
	Therefore, for every $i$,  $\rk_{\F_p} \langle \widetilde{\gamma}_{\frakp, i} \rangle _{/p} \leq \dim_{\F_p} R_{/p}$. So we have 	
	\begin{equation}\label{eq:rk_Np}
		\rk_A(N_{\frakp} \cap (\ker \kappa/M_1))\leq d_{\frakp}\frac{\dim_{\F_p} R_{/p}}{\dim_{\F_p}A}.
	\end{equation}
	Moreover,
	\[
		1 \longrightarrow \frac{\ker \kappa/M_1}{N_{\frakp} \cap (\ker\kappa /M_1)} \longrightarrow \frac{\varrho^{-1}(\calG_{\frakp}(R\calE_{\frakS}/Q))}{N_{\frakp} \cap (\ker\kappa/M_1)}\longrightarrow \calG_{\frakp}(R\calE_{\frakS}/Q) \longrightarrow 1
	\]
	is a split group extension.
	
	Define $N_{p|\Gamma|}$ to be the intersection of $\ker \kappa/M_1$ and the product of $N_{\frakp}$ over all $\frakp \in \cup_{\ell \mid (p|\Gamma|)} S_{\ell}(Q)$. After taking quotient of \eqref{eq:ses-varrho} by $N_{p|\Gamma|}$, we obtain an embedding problem
	\begin{equation}\label{eq:embedprob-split}
	\begin{tikzcd}
		& & & G_Q \arrow["\varphi", two heads]{d} \arrow[dashed]{dl}& \\
		1 \arrow{r} & \dfrac{\ker \kappa /M_1}{N_{p|\Gamma|}} \arrow{r} &\dfrac{(R^{\oplus r} \rtimes \Gamma)/M_1}{N_{p|\Gamma|}} \arrow{r} &\Gal(R\calE_{\frakS}/Q) \arrow{r} & 1,
	\end{tikzcd}
	\end{equation}
	The induced local embedding problem at every $\frakp \in \cup_{\ell \mid (p|\Gamma|)} S_{\ell}(Q)$ is split (an embedding problem is split if and only if the horizontal group extension is split), so they are solvable. For each prime $\frakp \in \frakS \backslash \cup_{\ell \mid(p|\Gamma|)} S_{\ell}(Q)$, the images of $x_{\frakp}$ and $y_{\frakp}$ define a solution to the induced local embedding problem. By Lemma~\ref{lem:EP}, the global embedding problem \eqref{eq:embedprob-split} has a solution factoring through $\Gal(K_{\frakS}/Q)$. By definition of $R\calE_{\frakS}$, $(\ker \kappa/M_1)/N_{p|\Gamma|}$ must be trivial, so
	\[
		\rk_A \ker\kappa/M_1 = \rk_A N_{p|\Gamma|} \leq \sum_{\frakp \in \cup_{\ell \mid(p|\Gamma|)} S_{\ell}(Q)} N_{\frakp} \cap (\ker \kappa/M_1) \leq \frac{\dim_{\F_p} R_{/p}}{\dim_{\F_p} A}\sum_{\frakp \in \cup_{\ell \mid(p|\Gamma|)} S_{\ell}(Q)} d_{\frakp},
	\]
	Therefore, $\rk_A \ker\kappa/M_1$ is bounded above by a constant depending on $\Gamma$, $p$, $Q$, $R$ and $A$, and we denote this upper bound by $c_2$. Then $\ker \kappa/M_1$ is generated by $c_2(A)$ elements, and by Nakayama's lemma $\ker \kappa/M$ is generated by $c_2:= \max\{c_2(A)\mid A \in \calM_{\F_p[\Gamma]}\}$ elements, so the proof is completed.
\end{proof}

	\begin{corollary}\label{cor:pres}
		Let $\overline{t}_{\frakp}$ denote the image of $t_{\frakp}$ in $\Gal(K/Q)\simeq \Gamma$, and define $\varkappa$ to be the composite map
		\[
			R^{\oplus r} \rtimes \Gamma \xtwoheadrightarrow{\kappa} \Gal(R\calE_{\frakS}/Q) \xtwoheadrightarrow{} \Gal(R\calE_S^T/Q).
		\]
		There exists a constant $c_3$ depending on $|\Gamma|$, $p$, $Q$, $S$, $T$, $R$ and the $\Gamma$-module structure of $A$ such that $\ker \varkappa$ is the smallest closed normal subgroup of $R^{\oplus r} \rtimes \Gamma$ containing elements of the following types:
		\begin{enumerate}
			\item\label{item:pres-1}
				\[
					x_{\frakp}^{\Nm (\frakp)} y_{\frakp} x_{\frakp}^{-1} y_{\frakp}^{-1}
				\]
				for each prime $\frakp \in S \backslash (\cup_{\ell \mid (p |\Gamma|)} S_{\ell}(Q) \cup T)$,
			
			\item\label{item:pres-2}
				\[
					x_{\frakp}^{\Nm (\frakp)} y_{\frakp} x_{\frakp}^{-1} y_{\frakp}^{-1} \quad \text{and} \quad x_{\frakp}^{|\overline{t}_{\frakp}|}
				\]
				for each prime $\frakp \in \frakS \backslash (\cup_{\ell \mid (p |\Gamma|)} S_{\ell}(Q) \cup S \cup T)$, and 
			\item\label{item:pres-3} additionally at most $c_3$ elements.
		\end{enumerate}
	\end{corollary}
	
	\begin{proof}
		By definition of $R\calE_S^T$, $\Gal(R\calE_S^T/Q)$ is the quotient of $\Gal(R\calE_{\frakS}/Q)$ modulo $\calT_{\frakP}(R\calE_{\frakS}/K)$ for each $\frakP \in \frakS\backslash (S\cup T)(K)$ and $\calG_{\frakP}(R\calE_{\frakS}/K)$ for each $\frakP \in T(K)$.
		For $\frakp \in \frakS\backslash( \cup_{\ell |(p|\Gamma|)} S_{\ell}(Q) \cup S \cup T)$ and a prime $\frakP$ of $K$ lying above $\frakp$, because an inertia subgroup $\calT_{\frakp}(R\calE_{\frakS}/Q)$ is generated by $t_{\frakp}$, $\calT_{\frakP}(R\calE_{\frakS}/K)$ is conjugate to the (pro)-cyclic subgroup of $\Gal(R\calE_{\frakS}/Q)$ generated by $t_{\frakp}^{|\overline{t}_{\frakp}|}$.
		So by Proposition~\ref{prop:pres-scrS}, we see that $\ker \varkappa$ is the smallest closed normal subgroup of $R^{\oplus r} \rtimes \Gamma$ containing elements in \eqref{item:pres-1}, \eqref{item:pres-2}, and 
		\begin{eqnarray}
			\calT_{\frakP}(R\calE_{\frakS}/K) &\text{for}& \frakP \in \cup_{\ell |(p|\Gamma|)} S_{\ell}(Q) \backslash (S \cup T) (K),\label{eq:pres-1}\\
			\calG_{\frakP}(R\calE_{\frakS}/K) &\text{for}&  \frakP \in T(K), \label{eq:pres-2}
		\end{eqnarray}
		\begin{equation}\label{eq:pres-3}
			 \text{the $c_2$ elements of wild type in Proposition~\ref{prop:pres-scrS}}.
		\end{equation}
		Note that, in \eqref{eq:pres-1}, $\calT_{\frakP}(R\calE_{\frakS}/K)$ is a (pro-)$p$-group (for $\ell=p$) or a (pro-)cyclic group (for $\ell \neq p$), so by \cite[Theorem~(7.5.11)]{NSW}, the minimal number of generators of $\calT_{\frakP}(R\calE_{\frakS}/K)$ can be bounded from above by a constant depending on $|\Gamma|$ and $Q$. Similarly, the minimal number of generators of $\calG_{\frakP}(R\calE_{\frakS}/K)$ in \eqref{eq:pres-2} can be bounded by a constant depending on $|\Gamma|$ and $Q$. Also, the number of primes in \eqref{eq:pres-1} and \eqref{eq:pres-2} is bounded by a constant depending of $|\Gamma|$, $S$, $T$ and $Q$ (recall both $S$ and $T$ are given and fixed). The number of elements in \eqref{eq:pres-3} is at most $c_2$ by Proposition~\ref{prop:pres-scrS}.
	\end{proof}

\section{Proof of Theorem~\ref{thm:lb-rank}}\label{sect:proof-main}

	In this section, we give the proof of Theorem~\ref{thm:lb-rank}. We apply the result in Section~\ref{ss:presentaiton-split} to the ring $R=e\Z_p[\Gamma]$ for $e\in\calE$ and let $A:=e\Z_p[\Gamma]/\frakm_e$. Let $\frakS$ be as defined in \S\ref{ss:presentaiton-split} for $\scrS=S \cup \scrR_I(K/Q)$,
	and let $eE_S^T, eC_S^T, e\calE_S^T, e\calC_S^T, e\calE_{\frakS}$ denote $RE_S^T, RC_S^T, R\calE_S^T, R\calC_S^T, R\calE_{\frakS}$ respectively.

	Note that $e\calE_S^T$ is a subfield of $eE_S^T$, so 
	\[
		\rk_{I} eC_S^T \geq \rk_{I} e\calC_S^T. 
	\]
	We will show that there exists a constant $c$ depending on $Q$, $S$, $T$, $\Gamma$, $p$ and $e$ such that 
	\begin{equation}\label{eq:pf-goal}
		\rk_{I}e\calC_S^T \geq \# \scrR_I(K/Q)- c,
	\end{equation}
	for any $\Gamma$-extension $K/Q$, and then Theorem~\ref{thm:lb-rank} immediately follows.
	
	By Proposition~\ref{prop:e-rk-arbitrary} and Lemma~\ref{lem:rank-maxsplit} applied to \eqref{eq:RC-ses-little} and \eqref{eq:split-RC-ses-little}, we have the following lower bound for $r:=\rk_A e\calC_{\frakS}$
	\begin{eqnarray}
		r &\geq& \rk_{A} e\calC_S^T \nonumber\\
		 &\geq& \frac{\dim_{\F_p} \B_{\scrS \backslash T}^{\scrS \cup T} (Q, A) + \sum\limits_{\frakp \in \scrS \backslash T} h^1(\calT_{\frakp}, A)^{\calG_{\frakp}}}{\dim_{\F_p} \End_{\Gamma}(A)}-c_4,\label{eq:No-trees}
	\end{eqnarray}
	where $c_4$ is a constant depending on $\Gamma$, $e$, $Q$, $S$ and $T$.
	
	For a prime $\frakp$ of $Q$ such that $\frakp \not \in S_p(Q)$, let $\calT_{\frakp}^{\tr}$ denote the maximal tame inertia subgroup (i.e., the pro-prime-to $\Nm(\frakp)$ completion of $\calT_{\frakp}$, which is a pro-cyclic group), and let $\calT_{\frakp}^{\tr}(p)$ denote the pro-$p$ completion of $\calT_{\frakp}^{\tr}$. Then 
	\begin{equation*}
		H^1(\calT_{\frakp}, A)^{\calG_{\frakp}} = H^1(\calT_{\frakp}^{\tr}(p), A)^{\calG_{\frakp}} 
		= \Hom_{\calG_{\frakp}}(\calT_{\frakp}^{\tr}(p), A).
	\end{equation*}
	When $\frakp \in \scrR_I(K/Q)$, the inertia subgroup $\calT_{\frakp}(K/Q)$ has order divisible by $p$ and $\calG_{\frakp}(K/Q)$ acts trivially on $A$. Because $\Gamma$ is abelian, $\calG_{\frakp}(K/Q)$ acts trivially on $\calT_{\frakp}^{\tr}(p)/p\calT_{\frakp}^{\tr}(p)\simeq \F_p$ and $A$.  Therefore, we have
	\begin{equation*}
		h^1(\calT_{\frakp}, A)^{\calG_{\frakp}} = \dim_{\F_p}\Hom_{\calG_{\frakp}}(\calT_{\frakp}^{\tr}(p), A) = \dim_{\F_p} A.
	\end{equation*}
	By \eqref{eq:No-trees} and Lemma~\ref{lem:End=A}, we have the following lower bound for $r$,
	\begin{equation}\label{eq:lb-r}
		r \geq \frac{\dim_{\F_p} \B_{\scrS \backslash T}^{\scrS \cup T}(Q, A)}{\dim_{\F_p} \End_{\Gamma}(A)} + \# \scrR_I(K/Q) \backslash T- c_4.
	\end{equation}

	We consider the surjection
	\[
		\varkappa: e\Z_p[\Gamma]^{\oplus r} \rtimes \Gamma \xtwoheadrightarrow{}\Gal(e\calE_S^T/Q)= \Gal(e\calE_S^T/K) \rtimes \Gamma
	\]
	defined in Corollary~\ref{cor:pres}. Taking the tensor products of the first components (i.e., $e\Z_p[\Gamma]^{\oplus r}$ and $\Gal(e\calE_S^T/K)$) with $e\Z_p[\Gamma]/ I$, we obtain the following surjective map
	\[
		\overline{\varkappa}: \left( \faktor{e\Z_p[\Gamma]}{I} \right)^{\oplus r} \rtimes \Gamma \xtwoheadrightarrow{}\left( \faktor{\Gal(e\calE_S^T/K)}{I \Gal(e\calE_S^T/K)} \right) \rtimes \Gamma.
	\]
	Then $\ker \overline{\varkappa}$ is the smallest normal subgroup of $(e\Z_p[\Gamma]/I)^{\oplus r} \rtimes \Gamma$ containing the images of the elements as described in Corollary~\ref{cor:pres}. For each $\frakp \in \frakS \backslash (\bigcup_{\ell | (p|\Gamma|)} S_{\ell}(Q) \cup S \cup T)$, we let $\frakx_{\frakp}$ and $\fraky_{\frakp}$ denote the images of $x_{\frakp}$ and $y_{\frakp}$ in $(e\Z_p[\Gamma]/I)^{\oplus r} \rtimes \Gamma$ respectively, and let $\overline{t}_{\frakp}$ and $\overline{s}_{\frakp}$ denote the images of $t_{\frakp}$ and $s_{\frakp}$ in $\Gal(K/Q)=\Gamma$ respectively. Then because $\kappa(x_{\frakp})=t_{\frakp}$ and $\kappa(y_{\frakp})=s_{\frakp}$, we can write 
	\[
		\frakx_{\frakp} = (a_{\frakp}, \overline{t}_{\frakp}), \quad \text{and} \quad \fraky_{\frakp} = ( b_{\frakp}, \overline{s}_{\frakp}),
	\]
	for some $a_{\frakp}, b_{\frakp} \in \Gal(e\calE_S^T/K)/I \Gal(e\calE_S^T/K)$, as represented using the notation of semidirect product. Then compute
	\begin{equation}\label{eq:comp-power}
		\frakx_{\frakp}^{|\overline{t}_{\frakp}|}=(a_{\frakp}, \overline{t}_{\frakp})^{|\overline{t}_{\frakp}|} = \left(a_{\frakp} \cdot \overline{t}_{\frakp}(a_{\frakp}) \cdot \overline{t}^2_{\frakp}(a_{\frakp}) \cdots  \overline{t}^{|\overline{t}_{\frakp}|-1}_{\frakp}(a_{\frakp}), 1 \right), \quad\text{and}
	\end{equation}
	\begin{eqnarray}
		\frakx_{\frakp}^{\Nm(\frakp)} \fraky_{\frakp} \frakx_{\frakp}^{-1} \fraky_{\frakp}^{-1} &=&  \frakx_{\frakp}^{\Nm(\frakp)-1} \frakx_{\frakp} \fraky_{\frakp} \frakx_{\frakp}^{-1} \fraky_{\frakp}^{-1} \nonumber \\
		&=& \frakx_{\frakp}^{\Nm(\frakp)-1}  \left(a_{\frakp}, \overline{t}_{\frakp}\right) \left(b_{\frakp}, \overline{s}_{\frakp}\right) \left(\overline{t}_{\frakp}^{-1}(a_{\frakp})^{-1}, \overline{t}_{\frakp}^{-1}\right) \left( \overline{s}_{\frakp}^{-1} (b_{\frakp})^{-1}, \overline{s}_{\frakp}^{-1} \right) \nonumber \\
		&=& \frakx_{\frakp}^{\Nm(\frakp)-1} \left(a_{\frakp}  \cdot \overline{s}_{\frakp} (a_{\frakp})^{-1} \cdot \overline{t}_{\frakp}(b_{\frakp}) \cdot b_{\frakp}^{-1} ,1 \right),
		\label{eq:comp-comm}
	\end{eqnarray}
	where the last uses the fact that $\Gamma$ is abelian.
	
	Suppose $\frakp\in\scrR_I (K/Q) \backslash ( \bigcup_{\ell | (p|\Gamma|)} S_{\ell}(Q) \cup S \cup T)$. By definition of $\scrR_I(K/Q)$, $\calG_{\frakp}(K/Q)$ acts trivially on $e\Z_p[\Gamma]/I$, so $\overline{s}_{\frakp}(a_{\frakp}) = a_{\frakp}$ and $\overline{t}_{\frakp}(b_{\frakp}) = b_{\frakp}$. Also, because the inertia subgroup $\calT_{\frakp}(K/Q) \subset \Gal(K/Q)$ has order divisible by $p$ and $\Gamma$ is abelian, by the presentation of Galois group of the maximal tamely ramified extension of $Q_{\frakp}$, we see that $\Nm(\frakp)-1$ is divisible by $|\overline{t}_{\frakp}|$. 
	Moreover, both $1-\overline{t}_{\frakp}$ and $\sum_{j=1}^{|\Gamma|} \overline{t}_{\frakp}^j$ annihilate $e\Z_p[\Gamma]/I$.
	Thus, from \eqref{eq:comp-power} and \eqref{eq:comp-comm}, we see that both $\frakx_{\frakp}^{| \overline{t}_{\frakp}|}$ and $\frakx_{\frakp}^{\Nm(\frakp)} \fraky_{\frakp} \frakx_{\frakp}^{-1} \fraky_{\frakp}^{-1}$ are trivial.
	
	We denote 
	\[
		S':= \frakS \backslash (\bigcup_{\ell| (p|\Gamma|)} S_{\ell}(Q) \cup \scrS \cup T),
	\]
	and then we have 
	\[
		\# S' \leq \frac{\dim_{\F_p} \B_{\scrS \backslash T}^{\scrS \cup T}(Q, A)}{\dim_{\F_p} \End_{\Gamma}(A)},
	\]
	by definition of $\frakS$. Note that both $\frakx_{\frakp}^{|\overline{t}_{\frakp}|}$ and $\frakx_{\frakp}^{\Nm(\frakp)} \fraky_{\frakp} \frakx_{\frakp}^{-1} \fraky_{\frakp}^{-1}$ are contained in the normal subgroup generated by $\frakx_{\frakp}$.
	
	Then, by the argument above and Corollary~\ref{cor:pres}, $\ker \overline{\varkappa}$ is contained in the smallest normal subgroup of $(e\Z_p[\Gamma]/I)^{\oplus r} \rtimes \Gamma$ containing
	\[
		\frakx_{\frakp} \quad \text{for each $\frakp \in S'$}
	\]
	and additionally $c_5$ many elements, where $c_5$ is a constant depending on $\Gamma$, $p$, $Q$, $S$ and $T$. These additional $c_5$ elements are those in Corollary~\ref{cor:pres} \eqref{item:pres-1} and \eqref{item:pres-3}. These additional elements together with the elements $\frakx_{\frakp}$ for $\frakp \in S'$ are all contained in the subgroup $(e\Z_p[\Gamma]/I)^{\oplus r}$ of $(e\Z_p[\Gamma]/I)^{\oplus r} \rtimes \Gamma$, because $\ker \overline{\varkappa}$ intersects trivially with the subgroup $\Gamma$. So the smallest normal subgroup containing these elements is exactly the $e\Z_p[\Gamma]$-submodule of $(e\Z_p[\Gamma]/I)^{\oplus r}$ generated by these elements. 
	
	Recall that the submodule of an $e\Z_p[\Gamma]$-module $M$ generated by one (arbitrary) element $x \in M$ is a quotient module of $e\Z_p[\Gamma]$. 	Finally, applying all the arguments, we have
	\begin{eqnarray*}
		 \rk_{I}  eC_S^T(K) &\geq & \rk_{I} \Gal(e\calE_S^T/K) \\
		&\geq & r - \#S' -c_5 \\
		&\geq & \frac{\dim_{\F_p} \B_{\scrS \backslash T}^{\scrS \cup T} (Q,A)}{\dim_{\F_p} \End_{\Gamma}(A)} + \# \scrR_I(K/Q) \backslash  T - c_4 - \# S' - c_5 \\
		&\geq & \# \scrR_I(K/Q) - \left(c_4+c_5  +\#T \right).
	\end{eqnarray*}
	Here the first step is because $\Gal(e\calE_S^T/K)$ is a quotient of $eC_S^T(K)$, the second step uses the presentation of $\overline{\varkappa}$ we discusses above, the third step uses \eqref{eq:lb-r}, and the last step uses the upper bound for $\#S'$ that follows from the definition of $\frakS$. Then the proof of Theorem~\ref{thm:lb-rank} is completed.

\section{Proof of Theorem~\ref{thm:kernel-limit}}\label{sect:Proof-kernel-limit}

\subsection{Preparation for the proof}
\hfill

	In this subsection, we prove Proposition~\ref{prop:tildeM}.

	\begin{lemma}\label{lem:comm_quot}
		Let $\pi_1:N_1 \to N_3$ and $\pi_2:N_2 \to N_3$ be two surjections of $P_A$-modules such that the $A$-ranks of $N_1, N_2, N_3$ are the same. Then there exist a unique maximal quotient $\overline{N_1}$ of $N_1$ and a unique maximal quotient $\overline{N_2}$ of $N_2$ such that the dashed isomorphic arrow in the following commutative diagram exists. 
		\[\begin{tikzcd}
			&&N_1 \arrow[two heads]{d}\arrow[two heads, bend left=45, "\pi_1"]{dd}\\
			&&\overline{N_1} \arrow[two heads]{d} \arrow[dashed, "\sim"']{dl} \\
			N_2 \arrow[two heads]{r} \arrow[two heads, bend right=45, "\pi_2"]{rr}&\overline{N_2} \arrow[two heads]{r} & N_3
		\end{tikzcd}\] 
	\end{lemma}
	
	\begin{proof}
		It is enough to show that there exists a maximal quotient $\overline{N_2}$ of $N_2$ such that both $\pi_1$ and $\pi_2$ factor through $\overline{N_2}$. We will prove it by showing if $U_1$ and $U_2$ are two submodules of $\ker \pi_2$ such that $\pi_1$ factors through $N_2/U_i$ for both $i=1,2$, then $\pi_1$ also factors through $N_2/(U_1\cap U_2)$.
		
		Note that $N_2/(U_1\cap U_2)$ is the fiber product of $N_2/U_1 \twoheadrightarrow N_2/(U_1U_2)$ and $N_2/U_2 \twoheadrightarrow N_2/(U_1U_2)$. By the assumption that $\pi_1$ factors through $N_2/U_i$ for $i=1,2$, we see that $\pi_1$ also factors through $N_2/(U_1U_2) \twoheadrightarrow N_3$. Then by the universal property of fiber product, there exists a homomorphism $\phi: N_1 \to N_2/(U_1 \cap U_2)$ such that $\pi_1$ is the composition of $\phi$, $N_2/(U_1 \cap U_2) \twoheadrightarrow N_2/(U_1U_2)$ and $N_2/(U_1U_2) \twoheadrightarrow N_3$. Finally by Nakayama's lemma, since $\rk_A N_1=\rk_A N_3=\rk_A N_2/(U_1 \cap U_2)$ and $\pi_1$ is surjective, we obtain that $\phi$ is surjective, which implies that $\pi_1$ factors through $N_2/(U_1 \cap U_2)$.
	\end{proof}

	\begin{definition}
		Given extensions $\pi_1: N_1 \to N_3$ and $\pi_2: N_2 \to N_3$ as described in Lemma~\ref{lem:comm_quot}, we denote
		\[
			N_1 \boxtimes_{N_3}N_2:= N_1\times_{\overline{N_2}} N_2,
		\]
		where the fiber product on the right-hand side is defined by the surjections $N_1\twoheadrightarrow \overline{N_1} \overset{\sim}{\to} \overline{N_2}$ and $N_2 \twoheadrightarrow \overline{N_2}$ in the diagram in Lemma~\ref{lem:comm_quot}. We call the isomorphism class of the extension $\overline{N_2} \to N_3$ \emph{the maximal common quotient of $\pi_1$ and $\pi_2$}.
	\end{definition}
	
	\begin{lemma}\label{lem:nonsplit}
		In the setting of Lemma~\ref{lem:comm_quot}, the $A$-rank of the fiber product $N_1 \times_{N_3} N_2$ defined by $\pi_1$ and $\pi_2$ equals $\rk_A N_3$ if and only if $\overline{N_1} \simeq \overline{N_2} \simeq N_3$. In particular,
		\[
			\rk_A N_1\boxtimes_{N_3} N_2 = \rk_A N_3.
		\]
	\end{lemma}
	
	\begin{proof}
		Denote $d:=\rk_A N_3$, and $\overline{N_1}, \overline{N_2}$ be as described in Lemma~\ref{lem:comm_quot}. 
		Assume $\overline{N_2}\not\simeq N_3$. Then $N_1 \to \overline{N_1} \overset{\sim}{\to} \overline{N_2}$ and $N_2 \to \overline{N_2}$ define a fiber product $N_1 \times_{\overline{N_2}} N_2$, and one can check that $N_1 \times_{\overline{N_2}} N_2$ is a proper submodule of $N_1 \times_{N_3} N_2$ that is mapped surjectively onto $N_3$. By Nakayama's lemma, $\rk_A(N_1 \times_{N_3} N_2) >d$, which completes the proof of the ``only if'' direction.
		
		For the ``if'' direction, assume $\rk_A N_1 \times_{N_3} N_2 >d$. Pick a generator set $z_1, \ldots, z_d$ of $N_3$, and pick $x_i \in \pi_1^{-1}(z_i)$ and $y_i \in \pi_2^{-1}(z_i)$ for each $i=1, \ldots, d$. Then by the assumption $\rk_A N_1=\rk_A N_2 =\rk_A N_3$, $x_1, \ldots, x_d$ form a generator set of $N_1$ and $y_1, \ldots, y_d$ form a generator set of $N_2$. For each $i$, let $w_i$ denote the element $(x_i, y_i)$ of the fiber product $N_1 \times_{N_3} N_2$. Let $N$ denote the submodule generated by $w_1, \ldots, w_d$. By our construction, the composite map $N \hookrightarrow N_1\times_{N_3} N_2 \twoheadrightarrow N_3$ is surjective and $N$ is a proper submodule of $N_1\times_{N_3} N_2$. Consider the following diagram
		\[\begin{tikzcd}
			N \arrow[hook]{r} \arrow[bend left=30, two heads, "\varphi_1"]{rr} \arrow[bend right=30, two heads, "\varphi_2", swap]{rd}& N_1 \times_{N_3} N_2 \arrow[two heads, "\phi_1"]{r} \arrow[two heads, "\phi_2"]{d} \arrow[two heads, "\varpi"]{dr}& N_1 \arrow[two heads, "\pi_1"]{d} \\
			& N_2 \arrow[two heads, "\pi_2"]{r} & N_3.
		\end{tikzcd}\]
		Define $\overline{N}:=\frac{N}{\ker \varphi_1 \ker \varphi_2}$, $\overline{M_1}:= \frac{N_1}{(\ker \varphi_1 \ker \varphi_2)/\ker \varphi_1}$ and $\overline{M_2}:= \frac{N_2}{(\ker \varphi_1 \ker \varphi_2)/\ker \varphi_2}$. Then because 
		\[
			\frac{N}{\ker \varphi_1 \ker \varphi_2} \simeq \frac{N/\ker \varphi_j}{(\ker \varphi_1 \ker \varphi_2)/\ker \varphi_j} \simeq \frac{N_j}{(\ker \varphi_1 \ker \varphi_2)/\ker \varphi_j} \text{  for  }j=1,2,
		\]
		 we see that the isomorphisms $\overline{M_1}\simeq \overline{N} \simeq \overline{M_2}$, and one can check that these isomorphisms are compatible with their quotients to $N_3$. Finally, let $u$ denote the index of $N$ in $N_1 \times_{N_3} N_2$, which is greater than 1. Then for $j=1,2$, $[\ker \phi_j: \ker \varphi_j]=u$ because both $\varphi_j$ and $\phi_j$ are surjective and $\ker \varphi_j= N \cap \ker \phi_j$. So $\ker \varphi_1 \ker \varphi_2=(\ker \phi_1 \cap N) \times (\ker \phi_2 \cap N)$ is of index $u^2$ in $\ker \varpi=\ker \phi_1 \times \ker \phi_2$. Therefore, $|\overline{M_1}|/|N_3|=u>1$, so for the module $\overline{N_1}$ described in Lemma~\ref{lem:comm_quot}, we have $|\overline{N_1}|/|N_3| \geq |\overline{M_1}|/|N_3|>1$, which implies $\overline{N_1}\not\simeq N_3$. So the proof of the ``if'' direction is completed.
	\end{proof}
	
	\begin{corollary}\label{cor:boxtime-UniversalProperty}
		Retain the notation and assumptions from Lemma~\ref{lem:comm_quot}, and further assume $\overline{N_1}=\overline{N_2}=N_3$. If there are surjections $\rho_1: N \to N_1$ and $\rho_2:N \to N_2$ such that $\pi_1\circ \rho_1=\pi_2\circ\rho_2$, then there is a unique surjection $\rho: N \to N_1 \times_{N_3} N_2$ such that $\pi_1 \circ \rho_1$ is the composition of $\rho$ and the natural surjection $N_1 \times_{N_3} N_2 \to N_3$.
	\end{corollary}
	
	\begin{proof}
		The existence and uniqueness of $\rho$ follow by the universal property of the fiber product, so it is enough to show $\rho$ is surjective. By Lemma~\ref{lem:nonsplit}, $\rk_A N_1 \times_{N_3} N_2=\rk_A N_3$; then because $\rho(N)$ maps surjectively onto $N_3$ under the map $N_1 \times_{N_3} N_2 \to N_3$, we have $\rho(N)=N_1 \times_{N_3} N_2$ by Nakayama's lemma.
	\end{proof}

	For a finite $P_A$-module $M$ such that $\ker \rho_{M}=0$, we define below a surjection $\overline{M} \twoheadrightarrow M$ of $P_A$-modules.
	
	Assume that the Sylow $p$-subgroup $\Gamma_p$ of $\Gamma$ has order at least $p^2$. Let 
	\[
		\Gamma_{(2)} \subset \Gamma_{(1)} 
	\]
	be subgroups of $\Gamma_p$ such that $[\Gamma_p: \Gamma_{(1)}]=p$, $[\Gamma_p: \Gamma_{(2)}]=p^2$, and $\Gamma_p/\Gamma_{(2)}$ is cyclic only when $\Gamma_p$ is cyclic. Let $\gamma_1$ and $\gamma_2$ be elements of $\Gamma_p$ such that $\gamma_1 \not\in \Gamma_{(1)}$ and $\gamma_2 \in \Gamma_{(1)} \backslash \Gamma_{(2)}$. 
	
	Recall that Lemma~\ref{lem:gamma-ann} establishes a correspondence between $\Idem(A)$ and the set of cyclic quotients of $\Gamma_p$. Let $e_0$ be the one corresponding to the trivial quotient of $\Gamma_p$, and $e_1$ the one corresponding to $\Gamma_p/\Gamma_{(1)}$. Note that $\Gamma_p/\Gamma_{(2)}$ is an abelian $p$-group of order $p^2$. If $\Gamma_p/\Gamma_{(2)}\simeq \Z/p^2\Z$ then let $N:=\Gamma_{(2)}$, and if $\Gamma_p/\Gamma_{(2)}\simeq \Z/p\Z \times \Z/p\Z$ then let $N$ be the smallest subgroup of $\Gamma_p$ containing $\Gamma_{(2)}$ and $\gamma_1$. Therefore, $\Gamma_p/N$ be a cyclic quotient. Let $e_2$ be the idempotent in $\Idem(A)$ corresponding to $\Gamma_p/N$. List the idempotents in $\Idem(A)\backslash\{e_0,e_1, e_2\}$ as $e_3, e_4, \ldots, e_n$. For a finite $P_A$-module $M$ and $i\in \{0, \ldots, n\}$, we can write 
	\[
		e_iM =  \bigoplus_{j=1}^r e_i\Z_p[\Gamma]/ \frakm_{e_i}^{d_{i,j}},
	\]
	where $r:=\rk_A M$ and $d_{i,j}$ is a positive integer for every $i,j$. Then we define
	\[
		\widetilde{e_iM} =  \bigoplus_{j=1}^r e_i\Z_p[\Gamma]/ \frakm_{e_i}^{d_{i,j}+1}.
	\]
	In other words, $\widetilde{e_iM}$ is the $e_i \Z_p[\Gamma]$-module that is an extension of $e_iM$ such that $\rk_A \widetilde{e_iM}=\rk_A M$ and $\ker( \widetilde{e_iM} \to e_iM)\simeq A^{\oplus r}$.
	We define $M_0:=e_0M$ and $\widetilde{M_0}:=\widetilde{e_0M}$, and for $i=1,\ldots, n$, define
	\[
		M_i:=e_iM \times_{e_i M_{i-1}} M_{i-1} \quad \text{and}\quad \widetilde{M_i}:=\widetilde{e_iM} \boxtimes_{e_i M_{i-1}} \widetilde{M_{i-1}},
	\]
	where the second one is defined by $\widetilde{e_iM} \to e_iM \to e_i M_{i-1}$ and $\widetilde{M_{i-1}} \to M_{i-1} \to e_i M_{i-1}$.

	\begin{proposition}\label{prop:tildeM}
		Assume $\Gamma_p$ is not trivial or $\Z/p\Z$. For a finite $P_A$-module $M$, define $\widetilde{M_i}$ as above. Then the following holds.
		\begin{enumerate}
			\item \label{item:tildeM-1new}
				For every $1\leq i \leq n$, $M_i$ is a quotient of $M$ and a quotient of $\widetilde{M_i}$; and $\rk_A \widetilde{M_i} = \rk_A M_i=\rk_A M$.
			\item \label{item:tildeM-2new} For every $2\leq i \leq n$, $\ker (\widetilde{M_i} \to M_i) \simeq A^{\oplus r_i}$ for some integer $r_i \geq 2\rk_A M + \rk_{I_{e_1}} e_1M$.
			\item \label{item:tildeM-3new} $\ker \rho_{M} = \ker (M \to M_n)$
		\end{enumerate}
	\end{proposition}

	\begin{proof}%[Proof of Proposition~\ref{prop:tildeM}]
		By the construction of $\widetilde{M_0}$ and $M_0$, we have 
		\[
			\ker(\widetilde{M_0} \to M_0) \simeq A^{\oplus r} \quad \text{and} \quad  \rk_A \widetilde{M_0} = \rk_A M_0.
		\] 
		Regarding $\widetilde{M_1}$ and $M_1$, by definition we have the following commutative diagram, where each arrow is surjective and the smaller square is cartesian.  
		\[\begin{tikzcd}
			\widetilde{M_1} \arrow{rr} \arrow{dd} \arrow[dashed]{dr}& & \widetilde{M_0} \arrow{d} \\
			& M_1 \arrow{r} \arrow{d} & M_0 \arrow{d} \\
			\widetilde{e_1M} \arrow{r} &e_1 M \arrow{r} &e_1M_0.
		\end{tikzcd}\]
		Since $\Gamma_{(1)}$ acts trivially on all the modules in the above diagram, we consider these modules as $P\otimes_{\Z_p} \Z_p[\Gamma_p/\Gamma_{(1)}]$-modules, where $P$ is the projective $\Z_p[\Gamma']$-module with $P/pP\simeq A$. Then one can check $e_0\Z_p[\Gamma]\simeq P$ and $e_1\Z_p[\Gamma] \simeq P \otimes_{\Z_p} (\Z_p[\Gamma_p/\Gamma_{(1)}]/\Z_p)$, where the ring $\Z_p[\Gamma_p/\Gamma_{(1)}]/\Z_p$ is the quotient of $\Z_p[\Gamma_p/\Gamma_{(1)}]$ by $\Z_p[\Gamma_p/\Gamma_{(1)}]^{\Gamma_p/\Gamma_{(1)}}\simeq \Z_p$. The $\Gamma_p/\Gamma_{(1)}$-coinvariant of $\Z_p[\Gamma_p/\Gamma_{(1)}]/\Z_p$ is isomorphic $\F_p$, so the $\Gamma_p/\Gamma_{(1)}$-coinvariant of $e_1\Z_p[\Gamma]$ is isomorphic $A$. Therefore, the $\Gamma_p$-coinvariant of $\widetilde{e_1M}$ is isomorphic to $A^{\oplus r}$, which is exactly $e_1M_0$. Because $\Gamma_p$ acts trivially on $\widetilde{M_0}$, the maximal common quotient of $\widetilde{M_0} \to e_1M_0$ and $\widetilde{e_1M} \to e_1M_0$ is $e_1M_0$.
		So, by Corollary~\ref{cor:boxtime-UniversalProperty}, there is a surjective map from $\widetilde{M_1} \to M_1$ and $\ker(\widetilde{M_1} \to M_1) \simeq \ker(\widetilde{M_0} \to M_0) \times \ker(\widetilde{e_1M} \to e_1M) \simeq A^{\oplus 2r}$. Lemma~\ref{lem:nonsplit} implies $\rk_A \widetilde{M_1} = \rk_A M_1 = \rk_A e_1 M_0=\rk_A M$. Similarly, applying Corollary~\ref{cor:boxtime-UniversalProperty} to the surjections $M \to M_0$ and $M \to e_1M$, we see that $M_1$ is a quotient of $M$, so we have \eqref{item:tildeM-1new} for $i=1$.
		
		For $1\leq i \leq n-1$, let $U_{i+1} \to e_{i+1} M_i$ be the maximal common quotient of $\widetilde{e_{i+1} M} \to e_{i+1} M_i$ and $\widetilde{M_i} \to e_{i+1} M_i$. 
		Consider the following commutative diagram in which all arrows are surjective and both the square containing $\widetilde{M_{i+1}}$ and the square containing $M_{i+1}$ are cartesian.
		\begin{equation}\label{eq:diag-tilde}\begin{tikzcd}
			\widetilde{M_{i+1}} \arrow{rr}\arrow{dd} & & \widetilde{M_i} \arrow{dd}\arrow{dr} & \\
			& M_{i+1} \arrow[crossing over]{rr}  & &M_i \arrow{dd} \\
			\widetilde{e_{i+1}M} \arrow{rr} \arrow{rd}&  & U_{i+1} \arrow{dr} &  \\
			 & e_{i+1} M\arrow{rr} \arrow[from=2-2, crossing over]  && e_{i+1}M_i.
		\end{tikzcd}\end{equation}
		If \eqref{item:tildeM-1new} holds for $i$, then $\rk_A e_{i+1}M_i = \rk_A M_i= \rk_A M$, which together with $\rk_A \widetilde{e_{i+1}M}=\rk_A M$ implies $\rk_A U_{i+1}=\rk_A M$. Note that any quotient of $e_{i+1}M$ is an $e_{i+1}\Z_p[\Gamma]$-module and $e_{i+1}M_i$ is the maximal quotient of $M_i$ that is an $e_{i+1}\Z_p[\Gamma]$-module, so $e_{i+1}M_i$ is the maximal common quotient of $M_i \to e_{i+1}M_i$ and $e_{i+1}M \to e_{i+1} M_i$. So by Lemma~\ref{lem:nonsplit} and the definition $\widetilde{M_{i+1}}= \widetilde{e_{i+1}M} \boxtimes_{e_{i+1}M_i} \widetilde{M_i}$ and $M_{i+1} = e_{i+1}M \times_{e_{i+1}M_i} M_i= e_{i+1}M \boxtimes_{e_{i+1}M_i} M_i$, we see that $\rk_A\widetilde{M_{i+1}}=\rk_A M_{i+1} = \rk_A M$. Furthermore, by Corollary~\ref{cor:boxtime-UniversalProperty}, there is a surjection $\widetilde{M_{i+1}} \to M_{i+1}$ that fits into the commutative diagram above; similarly, since $M_i$ and $e_{i+1}M$ are both quotients of $M$, there is a a surjection $M \to M_{i+1}$. So \eqref{item:tildeM-1new} can be proved by induction.

		Considering the diagram \eqref{eq:diag-tilde}, we have $\ker(\widetilde{M_{i+1}} \to M_{i+1})=\ker(\widetilde{M_{i+1}} \to e_{i+1}M) \cap \ker(\widetilde{M_{i+1}} \to M_i)$, so under the surjection $\widetilde{M_{i+1}}\to \widetilde{e_{i+1}M}$ (and resp. $\widetilde{M_{i+1}} \to \widetilde{M_{i}}$), $\ker(\widetilde{M_{i+1}} \to M_{i+1})$ is mapped to a submodule of $\ker(\widetilde{e_{i+1}M} \to e_{i+1}M)$ (resp. $\ker(\widetilde{M_{i}} \to M_i)$). As $\ker(\widetilde{M_{i+1}} \to \widetilde{e_{i+1}M}) \cap \ker (\widetilde{M_{i+1}} \to \widetilde{M_i})=0$, we see that 
		\begin{equation}\label{eq:ker-tilde}
			\ker(\widetilde{M_{i+1}} \to M_{i+1}) \hookrightarrow \ker(\widetilde{e_{i+1}M} \to e_{i+1}M) \times \ker(\widetilde{M_i} \to M_i).
		\end{equation}
		On the other hand, we compare $|\widetilde{M_{i+1}}|$ and $|M_{i+1}|$ as follows. Note that
		\begin{eqnarray*}
			&&\ker(\widetilde{M_{i+1}} \to e_{i+1}M_i)\\
			 &=& \{ (x,y) \in \ker(\widetilde{e_{i+1}M} \to e_{i+1}M_i) \times \ker(\widetilde{M_i} \to e_{i+1}M_i) \,:\, \text{images of $x, y$ in $U_{i+1}$ are equal}  \},
		\end{eqnarray*}
		so 
		\[
			\frac{|\widetilde{M_{i+1}}|}{|e_{i+1}M_i|}=\frac{|\ker(\widetilde{e_{i+1}M} \to e_{i+1}M_i)| |\ker(\widetilde{M_i} \to e_{i+1}M_i)|}{|\ker(U_{i+1} \to e_{i+1} M_i)|}.
		\]
		Then, since $|M_{i+1}|/|e_{i+1}M_i|=|\ker(e_{i+1}M \to e_{i+1}M_i)| |\ker(M_i \to e_{i+1} M_i)|$, we obtain
		\begin{equation}\label{eq:ker-tilde-2}
			|\ker(\widetilde{M_{i+1}} \to M_{i+1})|= \frac{|\ker(\widetilde{e_{i+1}M} \to e_{i+1}M)| |\ker(\widetilde{M_i} \to M_i)|}{|\ker(U_{i+1} \to e_{i+1}M_i)|}.
		\end{equation}
		By \eqref{eq:ker-tilde} and \eqref{eq:ker-tilde-2}, if $\ker(\widetilde{M_i} \to M_i) \simeq A^{\oplus r_i}$ for some integer $i$, then $\ker(\widetilde{M_{i+1}} \to M_{i+1}) \simeq A^{\oplus r_{i+1}}$ for 
		\begin{equation}\label{eq:r_i}
			r_{i+1}=r+r_i-\log_{|A|} |\ker(U_{i+1} \to e_{i+1}M_i)|.
		\end{equation}

		We are going to prove the inequality for $r_i$ in \eqref{item:tildeM-2new} by induction. Consider the diagram \eqref{eq:diag-tilde} for $i=1$. As $U_2$ is a quotient of $\widetilde{M_1}$, $\gamma_2$ acts trivially on $U_2$. If $\Gamma_p/\Gamma_{(2)}\simeq \Z/p\Z \times \Z/p\Z$, then the $\langle \gamma_2 \rangle$-coinvariant of $e_2\Z_p[\Gamma]$ is isomorphic to $\F_p$, so $U_2 \simeq e_2M_1\simeq \F_p^{\oplus r}$, and hence \eqref{eq:r_i} implies $\ker(\widetilde{M_2}\to M_2) \simeq A^{\oplus 3r}$. Otherwise, $\Gamma_p$ is cyclic generated by $\gamma_1$ and $\Gamma_p/\Gamma_{(2)} \simeq \Z/p^2\Z$, without loss of generality we assume $\gamma_2=\gamma_1^p$. In this case, one can explicitly write down the structure of $e\Z_p[\Z/p^2\Z]$ for all (the three) primitive idempotents $e$ of $\Q_p[\Z/p^2\Z]$; and one can see that if $V$ is an $e_2\Z_p[\Gamma]$-module such that $\gamma-2$ acts trivially on $V$, then $V=e_1V$ (when $V$ is viewed as a $\Z_p[\Gamma]$-module). Then, since $e_2M$ is a $e_2\Z_p[\Gamma]$-module and $\gamma_2$ acts trivially on $M_1$, $e_2 M_1$ equals $e_1(e_2M_1)$, so $M_1 \to e_2 M_1$ factors through $e_1 M_1$; also because $M_1 \to e_2M_1$ is defined by taking tensor product with $e_2\Z_p[\Gamma]$, $e_1M_1 \to e_2 M_1$ is also defined by $\otimes_{\Z_p[\Gamma]} e_2\Z_p[\Gamma]$.
		Similarly, $\widetilde{M_1} \to U_2$ factors through $e_1 \widetilde{M_1}$, and hence $U_2$ is a quotient of $e_2(e_1\widetilde{M_1})$.
		By the right exactness of tensor product $\ker(e_2(e_1\widetilde{M_1}) \to e_2M_1)$ is a quotient of $e_2 \ker(e_1\widetilde{M_1} \to e_1M) \simeq A^{\oplus r}$. Since $e_1\widetilde{M_1} = \widetilde{e_1M}$ and $e_1M_1=e_1M$, we have the following commutative diagram
		\[\begin{tikzcd}
			A^{\oplus r} \arrow[hook]{r} \arrow[equal]{d}&\widetilde{e_1M} \arrow[two heads]{rr} \arrow[two heads]{d} & & e_1M \arrow[two heads]{d} \\
			A^{\oplus r} \arrow{r} &e_2 (\widetilde{e_1M}) \arrow[two heads]{rr} \arrow[two heads]{rd}&  & e_2 M_1\\
			& & U_2 \arrow[two heads]{ru} &
		\end{tikzcd}\]
		where the two rows are exact.
		One can check in this case (when $\Gamma_p$ is cyclic) by definition that the ideal $I_{e_1}$ of $e_1\Z_p[\Gamma]$ is the image of $(1-\gamma_2, \sum_{j=1}^{|\gamma_2|} \gamma_2^j)$, so $U_2$ is an $e_1\Z_p[\Gamma]/I_{e_1}$-module, which implies that $I_{e_1}\cdot \widetilde{e_1M} \subseteq \ker( \widetilde{e_1M} \to U_2)$. Then by chasing the diagram above, we see that $\ker (U_2 \to e_2 M_1)$ is a quotient of $A^{\oplus r-\rk_{I_{e_1}} e_1M}$, so \eqref{eq:r_i} implies $\ker(\widetilde{M_2} \to M_2)\simeq A^{\oplus 2r+ \rk_{I_{e_1}} e_1M}$. Thus, \eqref{item:tildeM-2new} holds for $i=2$.
		
		Suppose \eqref{item:tildeM-2new} holds for $i\geq 2$. To prove \eqref{item:tildeM-2new} for $i+1$, by applying \eqref{eq:r_i}, it suffices to show $\ker(U_{i+1} \to e_{i+1}M_i)$ is a quotient of $A^{\oplus r}$. Because $\ker (\widetilde{M_i} \to M_i)$ is a direct product of copies of $A$, the kernel of the map $e_{i+1} \widetilde{M_i} \to e_{i+1} M_i$ induced by $\otimes_{\Z_p[\Gamma]} e_{i+1}\Z_p[\Gamma]$ is also a direct product of copies of $A$. As $U_{i+1}$ is a quotient of $e_{i+1} \widetilde{M_i}$, $\ker(U_{i+1} \to e_{i+1}M_i)$ is also a direct product of copies of $A$. Then, since $\rk_A \widetilde{e_{i+1}M} = \rk_A e_{i+1} M_i = \rk_A M$, from the diagram \eqref{eq:diag-tilde} we see that $\rk_A \ker(U_{i+1} \to e_{i+1}M_i) \leq r$. So the proof of \eqref{item:tildeM-2new} is completed.
		
		Finally, we prove \eqref{item:tildeM-3new}. By \eqref{item:tildeM-1new}, $M_n$ is a quotient of $M$, which induces quotient maps $e_i M \to e_i M_n$ for all $i$. Since $M_n$ is constructed in a way such that $e_iM$ is a quotient of $M_n$, so those quotient maps $e_iM \to e_iM_n$ are isomorphisms. If $x\in \ker(M \to M_n)$, then $x \in \ker \rho_{M,e_i}$ for all $i$. On the other hand, if $x \in \ker \rho_{M, e_i}$ for all $i$, then $x$ is in $\ker(M \to M_i)$ for all $i$. So $\ker \rho_M=\ker(M\to M_n)$.
	\end{proof}

\subsection{Proof of Theorem~\ref{thm:kernel-limit}}
\label{ss:proof-kernel-limit}
\hfill

	We apply the result in Section~\ref{ss:presentaiton-split} with 
	\[
		R=P_A\quad \text{and} \quad \scrS=\scrS_A:= S \cup \scrR_A(K/Q),
	\]
	and let $P_AE_S^T$, $P_AC_S^T$, $P_A\calE_S^T$, $P_A\calC_S^T$, $P_A\calE_{\frakS}$ and $P_A \calC_{\frakS}$ denote the notation $RE_S^T$, $RC_S^T$, $R\calE_S^T$, $R\calC_S^T$, $R\calE_{\frakS}$ and $R \calC_{\frakS}$ defined in \eqref{eq:RC-ses}-\eqref{eq:split-RC-ses-little}.
	For the module $P_A\calC_S^T$ and an idempotent $e \in \Idem(A)$,  taking tensor product with $\otimes_{\Z_p[\Gamma]} e\Z_p[\Gamma]$ defines a surjection $P_A\calC_S^T \to eP_A\calC_S^T$, so we define
	\begin{equation}\label{eq:theta}
		\theta_S^T(K): P_A \calC_S^T(K) \longrightarrow \bigoplus_{e\in \Idem(A)} eP_A\calC_S^T(K).
	\end{equation}
	Here we write $K$ explicitly since it worth pointing out that the map $\theta_S^T$ depends on $K$.
	
	Denote 
	\[
		r_{\frakS}:=\rk_A P_A \calC_{\frakS} \quad \text{and}\quad r_S^T:=\rk_A P_A \calC_S^T.
	\]
	
	\begin{lemma}\label{lem:frakS-ST}
		There exists a constant $C_1$ depending on $Q$, $\Gamma$ and $A$ such that 
		\[
			r_{\frakS} - r_S^T \leq C_1.
		\]
	\end{lemma}
	
	\begin{proof}
		Let $L:=Q(A,\mu_p)$. For every $\frakp$ of $Q$, since $p \nmid [L:Q]$, for every $\frakP \in \frakp(L)$, $\calT_{\frakP}$ is a normal subgroup of $\calT_\frakp$ of index prime to $p$, so
		\[
			\left(\bigoplus_{\frakP \in \frakp(L)} H^1(\calT_{\frakP}, A)^{\calG_{\frakP}} \right)^{\Gal(L/Q)} \simeq H^1(\calT_{\frakP}, A)^{\calG_{\frakp}}\simeq H^1(\calT_{\frakp}, A)^{\calG_{\frakp}} = \Hom_{\calG_{\frakp}}(\calT_{\frakp}, A).
		\]
		When $\frakp \not\in S_p(Q)$, $\dim_{\F_p}\Hom_{\calG_{\frakp}}(\calT_{\frakp},A) = \dim_{\F_p}\Hom_{\calG_{\frakp}}(\calT_{\frakp}^{tr}, A)\leq \dim_{\F_p} A$.
		So by definition of $\frakS$ and Lemma~\ref{lem:frakS},
		\begin{eqnarray*}
			&& \dim_{\F_p} \left(\bigoplus_{\frakP \in \frakS \backslash (\scrS_A \cup T)(L)} H^1(\calT_{\frakP}, A)^{\calG_{\frakP}} \right)^{\Gal(L/Q)} \\
			&=& \sum_{\frakp \in \frakS \backslash(\scrS_A \cup T)} \dim_{\F_p} \Hom_{\calG_{\frakp}}(\calT_{\frakp}, A) \\
			&\leq&\sum_{\frakp \in \cup_{\ell \mid p |\Gamma|} S_{\ell}(Q) \backslash (\scrS_A \cup T)} \Hom_{\calG_{\frakp}}(\calT_{\frakp},A)+ \frac{\dim_{\F_p} \B_{\scrS_A\backslash T}^{\scrS_A \cup T} (Q, A)}{\dim_{\F_p} \End_{\Gamma}(A) } \dim_{\F_p} A \\
			&\leq& C_0 + \dim_{\F_p} \B_{\scrS_A\backslash T}^{\scrS_A \cup T} (Q, A),
		\end{eqnarray*}
		where $C_0$ depends only on $A$, $Q$ and $p|\Gamma|$ and the last step uses Lemma~\ref{lem:End=A}.
		Then applying Lemma~\ref{lem:es-B} with $S_1=\scrS_A$, $S_2=\frakS$, $T=T$ and $k=L$, by Lemma~\ref{lem:B}, we have the equality above holds and
		\[
			0\leq h^1(G_{\frakS}^T(L), A)^{\Gal(L/Q)}-h^1(G_{\scrS_A}^T(L), A)^{\Gal(L/Q)} \leq C_0 .
		\]
		Then by Lemma~\ref{lem:base_change}, 
		\begin{equation}\label{eq:rk-S-frakS}
			|\rk_A C_S^T(K)-\rk_A C_{\frakS}^T (K)| \leq c_0 +\frac{C_0}{\dim_{\F_p} A}.
		\end{equation}
		Since $C_{\frakS}^T(K)$ is the quotient of $C_{\frakS}(K)$ by the Frobenius element at the primes in $T(K)$, and for each $\frakp \in T(Q)$, there are at most $|\Gamma|$ many primes of $K$ lying above $\frakp$, we have
		\begin{equation}\label{eq:frakS-frakST}
			0 \leq \rk_A C_{\frakS}(K) - \rk_A C_{\frakS}^T(K) \leq |\Gamma| \# T(Q).
		\end{equation}
		Finally, let $N$ denote the subgroup of $\Gal(P_A E_{\frakS}/Q)$ that defines the maximal split extension we are using (i.e., use the notation $N$ defined in Section~\ref{ss:presentaiton-split}), and recall 
		\[
			P_A \calC_{\frakS}:= P_A C_{\frakS} /N \quad \text{and} 
		\]
		\[
			P_A \calC_S^T:=\faktor{P_A C_{\frakS}}{N \ker(P_AC_{\frakS} \to P_AC_S^T)}\simeq \faktor{P_A C_S^T}{N/(\ker(P_A C_{\frakS} \to P_AC_S^T) \cap N)}.
		\]
		So 
		\[
			\rk_A C_{\frakS} -\rk_A N \leq \rk_A \calC _{\frakS} \leq \rk_A C_{\frakS}, \quad \text{and}
		\]
		\[
			\rk_A C_S^T - \rk_A N \leq \rk_A \calC_S^T \leq \rk_A C_S^T.
		\]
		Then the lemma follows by \eqref{eq:rk-S-frakS}, \eqref{eq:frakS-frakST}, and Lemma~\ref{lem:rank-maxsplit}.
	\end{proof}
	
	\begin{lemma}\label{lem:ker-theta}
		If the Sylow $p$-subgroup of $\Gamma$ is not trivial or $\Z/p\Z$, then
		\[
			\lim_{X \to \infty} \frac{\sum\limits_{K\in \calA_{\Gamma}(X,Q)}\, \rk_{A} \ker\theta_S^T(K)}{ \# \calA_{\Gamma}(X,Q)}= \infty.
		\]
	\end{lemma}
	
	\begin{proof}
		Use the notation $r_{\frakS}$ and $r_S^T$ defined in Lemma~\ref{lem:frakS-ST}.
		Then there exists a surjective group homomorphism
	\[
		\varkappa: P_A^{\oplus r_{\frakS}} \rtimes \Gamma \xtwoheadrightarrow{}  \Gal(P_A\calE_{\frakS}/Q) \xtwoheadrightarrow{} \Gal(P_A\calE_S^T/Q).
	\]
	By definition of $\frakS$, we have $\B_{\frakS \backslash T}^{\frakS \cup T}(Q,A)=0$, so it follows by \cite[Lemma~3.4]{Liu2022b} that $\B_{\frakS}^{\frakS}(Q,A)=0$.  Then by Corollary~\ref{cor:pres}, $\ker\varkappa$ is generated (as a $P_A$-module) by at most 
	\begin{equation}\label{eq:rel-count}
		m:=\# \left(S\backslash(\cap_{\ell \mid (p|\Gamma|)} S_{\ell}(Q) \cap T)\right) +2 \# \left(\frakS \backslash (\cap_{\ell \mid (p|\Gamma|)} S_{\ell}(Q) \cap S \cap T \right) + c_3
	\end{equation}
	many elements. So 
	\begin{equation}\label{eq:kappa-m}
		\rk_A \ker \varkappa \leq m.
	\end{equation}
	
	By Proposition~\ref{prop:e-rk}, Lemma~\ref{lem:rank-maxsplit} and $\B_{\frakS}^{\frakS}(Q,A)=0$,
	\begin{equation}\label{eq:r-lb}
		r_{\frakS} \geq \frac{\dim_{\F_p} \B_{\frakS}^{\frakS}(Q,A)+\sum_{\frakp \in \frakS} h^1(\calT_{\frakp}, A)^{\calG_{\frakp}}}{\dim_{\F_p}\End_{\Gamma}(A)}-c_1-h^2(\Gamma, \F_p)=\frac{\sum_{\frakp \in \frakS} h^1(\calT_{\frakp}, A)^{\calG_{\frakp}}}{\dim_{\F_p}\End_{\Gamma}(A)}-c_1-h^2(\Gamma, \F_p).
	\end{equation}
	We have
	\begin{eqnarray*}
		\# \left(\frakS \backslash (\cup_{\ell\mid p|\Gamma|} S_{\ell}(Q) \cup \scrS_A \cup T) \right) &=& \frac{\dim_{\F_p} \B_{\scrS_A\backslash T}^{\scrS_A\cup T} (Q,A)}{\dim_{\F_p} \End_{\Gamma}(A^{\vee})} \\
		&\leq& \frac{\dim_{\F_p}\B_{\frakS \backslash T}^{\frakS \cup T}(Q, A) + \sum_{\frakp \in \frakS \backslash (\scrS_A \cup T)(K)} h^1(\calT_{\frakp}, A)^{\calG_{\frakp}}}{\dim_{\F_p} \End_{\Gamma}(A)} \\
		&=&\frac{\sum_{\frakp \in \frakS \backslash (\scrS_A \cup T)(K)} h^1(\calT_{\frakp}, A)^{\calG_{\frakp}}}{\dim_{\F_p} \End_{\Gamma}(A)},
	\end{eqnarray*}
	where the first inequality and the last equality follow from the definition of $\frakS$ and the inequality uses Lemma~\ref{lem:es-B}. For every $\frakp \in \scrS_A$, since $\calG_{\frakp}$ acts trivially on $A$ and $\frakp\not\in S_p(K)$, it follows that
	$H^1(\calT_{\frakp}, A)^{\calG_{\frakp}}=\Hom(\calT_{\frakp}, A)^{\calG_{\frakp}}\simeq A,
	$
	so 
	\[
		\frac{h^1(\calT_{\frakp}, A)^{\calG_{\frakp}}}{\dim_{\F_p}\End_{\Gamma}(A)}=1.
	\]
	Therefore,
	\begin{equation}\label{eq:set2-ub}
		\#\left(\frakS \backslash (\cup_{\ell \mid p|\Gamma|} S_{\ell}(Q) \cup S \cup T)\right) \leq \frac{\sum_{\frakp \in \frakS \backslash T (K)} h^1(\calT_{\frakp}, A)^{\calG_{\frakp}}}{\dim_{\F_p} \End_{\Gamma}(A)}.
	\end{equation}
	So by \eqref{eq:rel-count}, \eqref{eq:r-lb} and \eqref{eq:set2-ub}, there exists a constant $D$ depending on $S, T, Q, A, \Gamma$ (not depending on $K$) such that
	\begin{equation}\label{eq:relation-r-m}
		m \leq 2r_{\frakS}+D.
	\end{equation}	
	
	Next, we consider $\im \theta_S^T$. Since the image of $\ker \theta_S^T$ under $P_A \calC_S^T(K) \to e P_A \calC_S^T(K)$ is trivial for any $e \in \Idem(A)$, taking quotient of $P_A \calC_S^T(K)$ by $\ker \theta_S^T$ does not change the $A$-rank, i.e., $\rk_A \im\theta_S^T = r_S^T$. Define a surjection $\alpha$ as
	\[
		\alpha: P_A^{\oplus r_{\frakS}} \rtimes \Gamma \xtwoheadrightarrow{\varkappa} \Gal(P_A \calE_S^T/Q) \xtwoheadrightarrow{/\ker \theta_S^T} \im \theta_S^T.
	\]
	Since the map $\im \theta_S^T \to \oplus_{e \in \Idem(A)} e\im\theta_S^T=\oplus_{e \in \Idem(A)} eP_A\calC_S^T(K)$ is injective, applying Proposition~\ref{prop:tildeM} to $M=\im \theta_S^T$, we have
	\begin{eqnarray}
		\rk_A \ker \alpha &\geq& 2\rk_A \im\theta_S^T +\rk_{I_{e_1}} e_1 \im \theta_S^T+r_{\frakS} - r_S^T \nonumber\\
		&=& r_S^T+r_{\frakS} +\rk_{I_{e_1}} e_1 P_A \calC_S^T(K). \label{eq:ker_alpha}
	\end{eqnarray}
	Then,
	\begin{eqnarray*}
		\rk_A \ker\theta_S^T &\geq & \rk_A \ker \alpha - \rk_A \ker \varkappa \\
		&\geq&  r_S^T+r_{\frakS} +\rk_{I_{e_1}} e_1 P_A \calC_S^T(K) -m \\
		&\geq& r_S^T-r_{\frakS} +\rk_{I_{e_1}} e_1 P_A \calC_S^T(K) -D \\
		&\geq& \rk_{I_{e_1}} e_1 P_A \calC_S^T(K) -D-C_1,
	\end{eqnarray*}
	where the first inequality follows by the definition of $\alpha$, the second uses \eqref{eq:kappa-m} and \eqref{eq:ker_alpha}, the third uses \eqref{eq:relation-r-m}, and the last follows from Lemma~\ref{lem:frakS-ST}. Since $D$ and $C_1$ are constants that are not depending on $K$, the proof is completed by Theorem~\ref{thm:conductor}.
	\end{proof}
	
	Since $\Z_p[\Gamma]=\oplus_{A\in \calM_{\F_p[\Gamma]}} P_A$ and for each $A$ every decomposition factor of $P_A$ is isomorphic to $A$, $\rk_A\ker \rho_S^T$ equals the $A$-rank of the kernel of 
	\[
		P_A\rho_S^T(K): P_AC_S^T(K) \longrightarrow \bigoplus_{e\in \Idem(A)} eC_S^T(K).
	\]
	By definition, $P_A\calC_S^T$ is a quotient of $P_AC_S^T$, and we denote the kernel of this quotient map by $\calN$. Then we have the following commutative diagram, in which the last two vertical maps are defined by taking direct sum of the tensor product maps and the rows are exact.
	\begin{equation}\label{eq:diag-N}
	\begin{tikzcd}
		1 \arrow{r} &\calN \arrow{r} \arrow["\rho"]{d} &P_AC_S^T \arrow{r} \arrow["P_A\rho_S^T"]{d} & P_A\calC_S^T \arrow["\theta_S^T"]{d} \arrow{r} & 1\\
		1 \arrow{r}&\bigoplus\limits_{e \in \Idem(A)} \ker(eC_S^T \to eP_A \calC_S^T) \arrow{r} & \bigoplus\limits_{e \in \Idem(A)} eC_S^T \arrow{r} & \bigoplus\limits_{e \in \Idem(A)} eP_A\calC_S^T \arrow{r} & 1
	\end{tikzcd}
	\end{equation}
	where $\ker(eC_S^T \to eP_A \calC_S^T)$ is a quotient of $e\calN$ by the right exactness of tensor product.
	Recall the definition in Section~\ref{ss:presentaiton-split}, $\Gal(P_A\calE_{\frakS}/Q) \twoheadrightarrow \Gal(K/Q)$ is a maximal split subextension of $\Gal(P_AE_{\frakS}/Q) \twoheadrightarrow \Gal(K/Q)$, so by Lemma~\ref{lem:rank-maxsplit}, $\rk_A \ker(P_AC_{\frakS}\to P_A\calC_{\frakS})$ is at most $h^2(\Gamma, \F_p)$ if $A\simeq\F_p$, and is 0 otherwise. Since $\calN$ is the image of $\ker(P_AC_{\frakS}\to P_A\calC_{\frakS})$ under the quotient map $P_AC_{\frakS} \to P_A C_S^T$, we have 
	\[
		\rk_A \calN \text{ is } \begin{cases}
			\leq h^2(\Gamma, \F_p) & \text{if } A\simeq\F_p \\
			=0 & \text{otherwise.}
		\end{cases}
	\]
	When $A\not\simeq \F_p$, $\calN$ is zero, so $P_AC_S^T=P_A\calC_S^T$ and $eC_S^T=e\calC_S^T$, and hence the claim in Theorem~\ref{thm:kernel-limit} follows by Lemma~\ref{lem:ker-theta} and the fact $\rk_A \ker \rho_S^T(K)=\rk_A\ker P_A\rho_S^T(K)=\rk_A\ker \theta_S^T(K)$. 
	
	For the rest of the proof, we assume $A\simeq \F_p$. Applying the snake lemma to \eqref{eq:diag-N}, we have the following exact sequence of $P_A$-modules.
	\begin{equation}\label{eq:es-sl}
		1\longrightarrow \ker \rho \longrightarrow \ker P_A \rho_S^T \longrightarrow \ker \theta_S^T \longrightarrow \coker \rho 
	\end{equation}
	Note that $\rk_{\F_p} \ker \rho \leq \rk_{\F_p} \calN \leq h^2(\Gamma, \F_p)$ and $\rk_{\F_p} \coker \rho \leq \sum_{e \in \Idem(A)} \rk_{\F_p} e\calN \leq h^2(\Gamma, \F_p)\#\Idem(\F_p) $, so the exact sequence \eqref{eq:es-sl} implies 
	\[
		|\rk_{\F_p} \ker P_A \rho_S^T/p\ker P_A \rho_S^T - \rk_{\F_p} \ker \theta_S^T / p\ker\theta_S^T| \leq h^2(\Gamma, \F_p)(\#\Idem(\F_p)+1)
	\]
	where the right-hand side depends only on $\Gamma$. Finally, note that $\frac{\rk_{\F_p}M/pM}{\rk_{\F_p} P_A/pP_A}\leq\rk_A M\leq \rk_{\F_p} M/pM$ for any $P_A$-module $M$, so the proof in this case is completed by applying Lemma~\ref{lem:ker-theta}.

\section{Preparation for function field moment counting}\label{sect:ff-moment-prep}
\subsection{$I$-closure of modules} 
\hfill
	
	Recall that $e\Z_p[\Gamma]$ is a discrete valuation ring, and its maximal ideal is denoted by $\frakm_e$.
	
	\begin{definition}[{$I$-closure of an $e\Z_p[\Gamma]$-module}]\label{def:Iclosure}
		Let $I$ be a nonzero proper ideal of $e\Z_p[\Gamma]$, and $d_I$ the integer such that 
		\[
			I=\frakm_e^{d_I}.
		\] 
		Given a finite $e\Z_p[\Gamma]$-module $M$ expressed as
		\[
			M\simeq \bigoplus_{i=1}^{r} e\Z_p[\Gamma]/\frakm_e^{n_i},
		\]	
		define the \emph{$I$-closure of $M$} to be 
		\[
			\bigoplus_{i=1}^r e\Z_p[\Gamma]/\frakm_e^{n_i+d_I}.
		\]
	\end{definition}
	
	\begin{lemma}\label{lem:Iclosure-prop}
		Let $H$ be a finite $e\Z_p[\Gamma]$-module such that $H$ is the $I$-closure of $IH$.
		\begin{enumerate}
			\item\label{item:Iclosure-1} If $M$ is a finite $e\Z_p[\Gamma]$-module such that $IM\simeq IH$, then there exists an $e\Z_p[\Gamma]/I$-module such that $M \simeq H \oplus B$.
			\item\label{item:Iclosure-2} Let $M$ be a finite $e\Z_p[\Gamma]$-module. If $\phi: IM \to IH$ is a surjection, then $\phi$ can be extended to a surjection from $M$ to $H$, i.e., there exists a surjection $\varphi: M \to H$ such that $\varphi|_{IM}=\phi$. 
		\end{enumerate}
	\end{lemma}
	
	\begin{proof}
		Write $M$ as 
		\[
			M\simeq \bigoplus_{i=1}^r e\Z_p[\Gamma]/\frakm_e^{m_i}.
		\]
		If $IM \simeq IH$, then $H$ is isomorphic to the direct sum of the summands in $M$ such that $m_i > d_I$. Define $B$ to be the direct sum of the summands with $m_i \leq d_I$. Then $M\simeq H \oplus B$, so \eqref{item:Iclosure-1} is proved. 
		
		Suppose $\phi: IM \to IH$ is a surjection. Since $\ker \phi$ is a submodule of $M$, we define $\overline{M}:=M/\ker \phi$ and then $\phi$ factor through $I\overline{M}$, where $I\overline{M}\simeq IH$. By \eqref{item:Iclosure-1}, there exists $B$ such that $\overline{M} \simeq H \oplus B$, so taking quotient by $B$ gives a surjection $\overline{M} \twoheadrightarrow H$. Then the composition $M \twoheadrightarrow \overline{M} \twoheadrightarrow H$ gives the desired $\varphi$.
	\end{proof}

	\begin{proposition}\label{prop:weight-formula}
		Let $H$ be a finite $e\Z_p[\Gamma]$-module such that $H$ is the $I$-closure of $IH$. For any finite $e\Z_p[\Gamma]$-module, denote
		\[
			w(M,H):=\begin{cases}
				\# \Hom_{e\Z_p[\Gamma]} (M, H[I]) & \text{ if } \Sur_{e\Z_p[\Gamma]}(M, H_{/I}) \neq \O \\
				0 & \text{ otherwise}.
			\end{cases}
		\]
		Then 
		\begin{equation}\label{eq:weight-Sur}
			\# \Sur_{e\Z_p[\Gamma]} (M, H) = w(M, H) \# \Sur_{e\Z_p[\Gamma]} (IM, IH).
		\end{equation}
	\end{proposition}
	
	\begin{proof}
		If $\Sur_{e\Z_p[\Gamma]}(M, H_{/I})=\O$, then $\Sur_{e\Z_p[\Gamma]}(M, H)$ must also be empty, so \eqref{eq:weight-Sur} holds in this case. For the rest of the proof, assume $\Sur_{e\Z_p[\Gamma]}(M, H_{/I})\neq\O$.
	
		Let $\varphi \in \Sur_{e\Z_p[\Gamma]}(M, H)$. By the right exactness of tensor product, the kernel of $M_{/I} \to H_{/I}$ is a quotient of $(\ker \varphi)_{/I}$, so we have the following commutative diagram 
		\[\begin{tikzcd}
			0 \arrow{r} & \ker \varphi \arrow{r} \arrow[two heads]{d} & M \arrow["\varphi"]{r} \arrow[two heads,"\otimes {e\Z_p[\Gamma]/I}"]{d} & H \arrow{r} \arrow[two heads,"\otimes {e\Z_p[\Gamma]/I}"]{d} & 0 \\
			0 \arrow{r} &\ker(M_{/I} \to H_{/I}) \arrow{r}& M_{/I} \arrow{r} & H_{/I} \arrow{r} & 0 
		\end{tikzcd}\]
		Then it follows by the snake lemma that $\varphi|_{IM}$ is a surjection from $IM$ to $IH$. So we obtain a map 
		\begin{eqnarray*}
			{}\Sur_{e\Z_p[\Gamma]}(M, H) &\overset{\beta}{\longrightarrow}& \Sur_{e\Z_p[\Gamma]}(IM, IH) \\
			{}\varphi &\longmapsto& \varphi|_{IM}.
		\end{eqnarray*}
		
		The map $\beta$ is surjective by Lemma~\ref{lem:Iclosure-prop}\eqref{item:Iclosure-2}, so it suffices to show $\#\ker \beta = w(M,H)$. Suppose $\varphi_1, \varphi_2 \in \Sur_{e\Z_p[\Gamma]}(M,H)$ such that $\beta(\varphi_1)=\beta(\varphi_2)$. Then the map from $M$ to $H$ that sends $x$ to $\varphi_1(x) \varphi_2(x)^{-1}$ is a module morphism that is a zero map when restricted to $IM$, so it belongs to $\Hom_{e\Z_p[\Gamma]}(M, H[I])$.
		 On the other hand, given $\varphi \in \Sur_{e\Z_p[\Gamma]}(M, H)$ and $\delta \in \Hom_{e\Z_p[\Gamma]}(M, H[I])$, we have a module homomorphism
		 \begin{eqnarray*}
		 	\varphi+\delta: M &\longrightarrow& H \\
			x & \longmapsto& \varphi(x)+\delta(x).
		 \end{eqnarray*}
		Taking the composition of $\varphi+\delta$ with the radical quotient map $H \twoheadrightarrow H_{/\frakm_e}$, we obtain a surjection $\xi: M \to H_{/\frakm_e}$. By the assumption that $H$ is the $I$-closure of $IH$, using Definition~\ref{def:Iclosure}, one can check that $H[I] \subset \frakm_e H$. Then since the image of $\delta$ is contained in $H[I]\subseteq \frakm_e H$ and $\varphi$ is surjective, we conclude that $\xi$ is surjective. Finally, by the Nakayama lemma, the surjectivity of $\xi$ implies the surjectivity of $\varphi+\delta$. So we see that $\# \ker \beta = \# \Hom_{e\Z_p[\Gamma]}(M, H[I])$, which completes the proof.
	\end{proof}

\subsection{Preparation for function field counting}
\hfill

	Throughout this subsection, let $H$ denote a finite $\Z_p[\Gamma]$-module and let $\gamma$ denote an element of the abelian group $\Gamma$. Given $H$ and $\gamma$, define the following sets of elements of $H$.
	\begin{eqnarray*}
		\frakA_{\gamma}^0(H) &:=& \{ h \in H \mid (1-\gamma)h=(1+\gamma+ \gamma^2 + \cdots + \gamma^{|\gamma|-1})h=0 \} \\
		\frakA_{\gamma}^-(H) &:=& \{ h \in H \mid (1+ \gamma + \gamma^2 + \cdots + \gamma^{|\gamma|-1}) h =0 \} \\
		\frakA_{\gamma}^+(H) &:=& \{ h \in H \mid (1-\gamma)h =0 \} \\
		\frakB_{\gamma}^-(H) &:=& \{(1-\gamma)h \mid h \in H\} \\
		\frakB_{\gamma}^+(H) &:=& \{(1+\gamma+ \gamma^2+ \cdots + \gamma^{|\gamma|-1})h \mid h \in H\}
	\end{eqnarray*}
	In other words, if we let $I$ denote the ideal of $\Z_p[\Gamma]$ generated by $1- \gamma$ and let $J$ denote the ideal generated by $1+\gamma+\gamma^2+ \cdots + \gamma^{|\gamma|-1}$, then the sets defined above are submodules of $H$:
	\[
		\frakA_{\gamma}^0(H)=H[I+J], \quad A_{\gamma}^+(H)=H[I], \quad A_{\gamma}^-(H)=H[J]
	\]
	\[
		\frakB_{\gamma}^-(H)=IH, \quad \text{and} \quad \frakB_{\gamma}^+(H)=JH.
	\]
	We summarize some basic properties of these submodules in the following lemma.
	
	\begin{lemma}\label{lem:basic-AB}
		\begin{enumerate}
			\item\label{item:basic-AB-1} $\frakA_{\gamma}^0(H) = \frakA_{\gamma}^-(H) \cap \frakA_{\gamma}^+(H)$, $\frakB_{\gamma}^-(H) \subset \frakA_{\gamma}^-(H)$, and $\frakB_{\gamma}^+(H) \subset \frakA_{\gamma}^+(H)$.
			\item\label{item:basic-AB-2} If $H_1$ is a sub-$\Z_p[\Gamma]$-module of $H$, then $\frakA_{\gamma}^0(H_1)=H_1 \cap \frakA_{\gamma}^0(H)$, $\frakA_{\gamma}^-(H_1)= H_1 \cap \frakA_{\gamma}^-(H)$ and $\frakA_{\gamma}^+(H_1)=H_1 \cap \frakA_{\gamma}^+(H)$.
			\item\label{item:basic-AB-3} If $\pi: H \to H_1$ is a quotient map of $\Z_p[\Gamma]$-modules, then $\frakB_{\gamma}^-(H_1)=\pi(\frakB_{\gamma}^-(H))$ and $\frakB_{\gamma}^+(H_1)= \pi(\frakB_{\gamma}^+(H))$.
		\end{enumerate}
	\end{lemma}
	
	\begin{proof}
		Statements \eqref{item:basic-AB-2}, \eqref{item:basic-AB-3} and the equality $\frakA_{\gamma}^0(H)=\frakA_{\gamma}^-(H) \cap \frakA_{\gamma}^+(H)$ in \eqref{item:basic-AB-1} follow immediately by definition. The rest of \eqref{item:basic-AB-1} follows by $(1+ \gamma+ \cdots +\gamma^{|\gamma|-1})(1-\gamma)=1- \gamma^{|\gamma|}=0$.
	\end{proof}

	\begin{definition}\label{def:c_gamma}
		Let $\pi:G \to \Gamma$ be a surjection of finite groups. For any $\gamma \in \Gamma$, let $c_{\gamma}(G,\pi)$ denote the set of elements of $G$ that map to $\gamma$ under $\pi$ and have the same order as $\gamma$, and let $d_{\gamma}(G,\pi)$ denote the number of conjugacy classes of elements in $c_{\gamma}(G,\pi)$.
	\end{definition}
	
	\begin{lemma}\label{lem:AB-c}
		Let $H$ be a finite $\Z_p[\Gamma]$-module and $G= H\rtimes \Gamma$, and let $\pi$ denote the natural surjection $G \to \Gamma$. Then, for any $\gamma \in \Gamma$ and $g \in c_{\gamma}(G, \pi)$, there is a bijection 
		\begin{eqnarray*}
			\frakA_{\gamma}^-(H) &\longrightarrow& c_{\gamma}(G, \pi) \\
			h &\longmapsto& (h,\gamma). 
		\end{eqnarray*}
		Moreover, two elements $(h_1, \gamma)$ and $(h_2, \gamma)$ in $c_{\gamma}(G,\pi)$ are conjugate in $G$ if and only if the images of $h_1$ and $h_2$ in $\frakA_{\gamma}^-(H)/\frakB_{\gamma}^-(H)$ are in the same $\Gamma$-orbit. 
	\end{lemma}
	
	\begin{proof}
		For $h \in H$, the element $(h, \gamma) \in G$ is contained in $c_{\gamma}(G, \pi)$ if and only if $(h, \gamma)^{|\gamma|}=1$. By the multiplication rule of semidirect products, we have $(h, \gamma)^{|\gamma|}=(h\gamma(h) \gamma^2(h) \cdots \gamma^{|\gamma|-1}(h), \gamma^{|\gamma|})$, which is trivial if and only if $h \in \frakA_{\gamma}^-(H)$, so we obtain the bijection in the lemma.
		
		For any $(a,b) \in G$, the conjugation of $(h_1, \gamma)$ by $(a,b)$ is $(a,b)^{-1} (h_1, \gamma) (a,b) = (b^{-1}(a)^{-1}, b^{-1})(h_1, \gamma) (a,b) = (b^{-1}(h_1)\cdot b^{-1}(a)^{-1}\cdot\gamma(b^{-1}(a))), \gamma)$. For a fixed $b$, any element of $\frakB_{\gamma}^{-}(H)$ can be written as $b^{-1}(a)^{-1} \cdot \gamma(b^{-1}(a))$ for some appropriate $a$. So given $h_1$ and $h_2$, $(h_1, \gamma)$ is conjugate to $(h_2, \gamma)$ if and only if there exists $b\in \Gamma$ such that $b^{-1}(h_1) \in h_2 \frakB_{\gamma}^-(H)$.
	\end{proof}
	
	\begin{lemma}\label{lem:isom-module}
		Let $e$ be a primitive idempotent of $\Q_p[\Gamma]$ and $\gamma$ an element of $\Gamma$. For any finite $e\Z_p[\Gamma]$-module $H$, the modules $\frakA_{\gamma}^{-}(H)/\frakB_{\gamma}^-(H)$, $\frakA_{\gamma}^0(H)$ and $H/\frakB_{\gamma}^-(H) \frakB_{\gamma}^+(H)$ are isomorphic.
	\end{lemma}
	
	\begin{proof}
		Let $I_{\gamma}$ denote the ideal of $e\Z_p[\Gamma]$ generated by the images of $1-\gamma$ and $\sum_{\i=1}^{|\gamma|} \gamma^i$ under the quotient map $\Z_p[\Gamma] \to e\Z_p[\Gamma]$. Then $H/\frakB_{\gamma}^-(H)\frakB_{\gamma}^+(H)=H_{/I_{\gamma}}$ and $\frakA_{\gamma}^0(H)=H[I_{\gamma}]$, so by Lemma~\ref{lem:DVR-module}\eqref{item:DVR-2},
		\[
			H/\frakB_{\gamma}^-(H) \frakB_{\gamma}^+(H) \simeq \frakA_{\gamma}^0(H).
		\]
		
		By Lemma~\ref{lem:gamma-ann}, one of $\frakB_{\gamma}^+(H)$ and $\frakB_{\gamma}^-(H)$ is zero. If $\frakB_{\gamma}^+(H)$ is zero, then $\frakA_{\gamma}^-(H) =H$, so $\frakA_{\gamma}^-(H)/\frakB_{\gamma}^-(H)=H/\frakB_{\gamma}^-(H)$ and then $\frakA_{\gamma}^-(H)/\frakB_{\gamma}^-(H)\simeq H/\frakB_{\gamma}^-(H) \frakB_{\gamma}^+(mH)$. If $\frakB_{\gamma}^-(H)=0$, then $\frakA_{\gamma}^-(H)=\frakA_{\gamma}^0(H)$ and hence $\frakA_{\gamma}^-(H)/\frakB_{\gamma}^-(H)\simeq \frakA_{\gamma}^0(H)$. 
	\end{proof}
	
	The corollary below follows immediately by Lemma~\ref{lem:AB-c} and Lemma~\ref{lem:isom-module}.
	
	\begin{corollary}\label{cor:d=isom-module}
		Let $e$ be a primitive idempotent of $\Q_p[\Gamma]$, $H$ a finite $e\Z_p[\Gamma]$-module and $G:=H \rtimes \Gamma$. Then $d_{\gamma}(G,\pi)$ is equal to
		\begin{enumerate}
			\item the number of the $\Gamma$-orbits of $\frakA_{\gamma}^-(H)/\frakB_{\gamma}^-(H)$;
			\item the number of the $\Gamma$-orbits of $H/\frakB_{\gamma}^+(H) \frakB_{\gamma}^-(H)$;
			\item the number of the $\Gamma$-orbits of $\frakA_{\gamma}^0(H)$.
		\end{enumerate}
	\end{corollary}
	
	When $\pi:G \to \Gamma$ is nonsplit, we have the following proposition about $d_{\gamma}(G,\pi)$.
	
	\begin{proposition}\label{prop:d-semi}
		Let $\pi: G \to \Gamma$ be a surjection of finite groups such that $H:=\ker \pi$ is an abelian $p$-group. The conjugation of $G$ on $H$ gives $H$ a $\Z_p[\Gamma]$-module structure. Then for each $\gamma \in \Gamma$, 
		\begin{equation}\label{eq:d-lb}
			d_{\gamma}(G, \pi) \leq \# \{\text{$\Gamma$-orbits of $\frakA_{\gamma}^-(H)/\frakB_{\gamma}^-(H)$}\}.
		\end{equation}
		Moreover, if the $\Z_p[\Gamma]$-action on $H$ factors through $e\Z_p[\Gamma]$ for a primitive idempotent $e$ of $\Q_p[\Gamma]$ and the equality in \eqref{eq:d-lb} holds for every $\gamma$, then $G$ is isomorphic to $H \rtimes \Gamma$.
		
	\end{proposition}
	
	\begin{proof}
		The inequality in \eqref{eq:d-lb} holds when $d_{\gamma}(G, \pi)=0$ because the right-hand side is always positive. Assume $d_{\gamma}(G, \pi)>0$, and let $g$ be an element of $c_{\gamma}(G, \pi)$. Then an element $x \in G$ is in $c_{\gamma}(G, \pi)$ if and only if $x=gh$ for some $h \in H$ such that $(gh)^{|\gamma|}=1$. Because $(gh)^{|\gamma|}=(ghg^{-1})(g^2hg^{-2}) \cdots (g^{|\gamma|} h g^{-|\gamma|})=\sum_{i=1}^{|\gamma|} \gamma^i(h)$, we see that $x=gh$ belongs to $c_{\gamma}(G, \pi)$ if and only if $h \in \frakA_{\gamma}^-(H)$.
		
		For any $a \in \frakB_{\gamma}^-(H)$, there exists $y \in H$ such that $a=(1-\gamma)y$, so $a$ as an element of $G$ is equal to the commutator $[g, y]$. Then for any element $h \in H$, $gh$ is conjugate to $gah$ since $y^{-1} ghy=g[g,y]h$. So for each $g \in c_{\gamma}(G, \pi)$, elements of the coset $g \frakB_{\gamma}^-(H)$ belong to the same conjugacy class of $G$. 
		
		Consider a fixed $g\in c_{\gamma}(G, \pi)$. 
		Let $\overline{G}:=G/\frakB^-_{\gamma}(H)$, and let $\overline{g}$ denote the image of $g$ in $\overline{G}$.
		We have
		\begin{eqnarray}
			d_{\gamma}(G, \pi) &=& \sum_{z \in \frakA^-_{\gamma}(H)/\frakB^-_{\gamma}(H)} ( \text{the size of conjugacy class of $\overline{g}z$ in $\overline{G}$} )^{-1} \nonumber \\
			&=& \sum_{z \in \frakA^-_{\gamma}(H)/\frakB^-_{\gamma}(H)} \frac{|Z_{\overline{G}}(\overline{g}z)|}{|\overline{G}|} \nonumber \\
			&=& \frac{1}{|\overline{G}|} \sum_{z \in \frakA^-_{\gamma}(H)/\frakB^-_{\gamma}(H)} \#\{s \in \overline{G} \mid [\overline{g}z, s]=1\} \nonumber \\
			&=& \frac{1}{|\overline{G}|} \sum_{s \in \overline{G}}\# \{z \in \frakA^-_{\gamma}(H)/\frakB^-_{\gamma}(H) \mid [\overline{g}z,s]=1\} \label{eq:dgamma-1},
		\end{eqnarray}
		where $Z_{\overline{G}}(\overline{g}z)$ denotes the centralizer of $\overline{g}z$ in $\overline{G}$.
		Since $H$ and $\Gamma$ are abelian, we have $[\overline{g},s]\in H/\frakB^-_{\gamma}(H)$ for any $s \in \overline{G}$, and $[\overline{g}z, s]=[\overline{g},s]^z[z,s]=[\overline{g},s][z,s]$ for any $s\in \overline{G}$, $z \in H/\frakB^-_{\gamma}(H)$. Moreover, for any $s$, the following map is a homomorphism of abelian groups.
		\begin{eqnarray}
			\alpha_{s,\gamma} : H/\frakB^-_{\gamma}(H) &\longrightarrow& H/\frakB^-_{\gamma}(H) \label{eq:alpha}\\
			z &\longmapsto& [z,s] \nonumber
		\end{eqnarray}
		So $\#\{ z \in \frakA^-_{\gamma}(H)/\frakB^-_{\gamma}(H) \mid [\overline{g}z,s]=1\}$ is $\#\ker \alpha_{s,\gamma}$ if $[\overline{g},s]\in \im \alpha_{s,\gamma}$, and is 0 otherwise. Then \eqref{eq:dgamma-1} is 
		\begin{eqnarray}
			&\leq & \frac{1}{|\overline{G}|} \sum_{s \in \overline{G}} \#\{z \in \frakA^-_{\gamma}(H)/\frakB^-_{\gamma}(H) \mid [z,s]=1\} \label{eq:ineq}\\
			&=& \frac{1}{|\Gamma|} \sum_{\sigma \in \Gamma} \#\{ z \in \frakA^-_{\gamma}(H)/\frakB^-_{\gamma}(H) \mid z=\sigma(z)\} \nonumber\\
			&=& \text{the number of $\Gamma$-orbits of $\frakA^-_{\gamma}(H)/\frakB^-_{\gamma}(H)$}, \nonumber
		\end{eqnarray}
		so we proved \eqref{eq:d-lb}.
			
		For the rest of the proof, we assume that the $\Z_p[\Gamma]$-action on $H$ factors through $e\Z_p[\Gamma]$ for a primitive idempotent $e$ of $\Q_p[\Gamma]$ and $d_{\gamma}(G,\pi)$ equals the number of $\Gamma$-orbits of $\frakA^-_{\gamma}(H)/\frakB^-_{\gamma}(H)$ for every $\gamma\in \Gamma$. Then $c_{\gamma}(G,\pi)\neq \O$ for any $\gamma$, because $\frakA_{\gamma}^-(H)/\frakB^-_{\gamma}(H)$ is nonempty. 
		By Lemma~\ref{lem:max-ideal}, the $\Gamma$-action on $H$ factors through a cyclic quotient $C$ of $\Gamma$. 
		So we can pick a set of generators $\gamma_1, \ldots, \gamma_d$ of $\Gamma$, such that $\gamma_1$ maps to a generator of $C$ and $\gamma_i\in \ker (G\to C)$ for any $i\geq 2$. We pick one $g_i \in c_{\gamma_i}(G, \pi)$ for each $i$, and we will show $[g_i, g_j]=1$ for any $i \neq j$. Then it follows immediately that the subgroup of $G$ generated by $g_1, \ldots, g_{d}$ is an abelian group isomorphic to $\Gamma$, so it gives a splitting for $\pi$ and then $G$ is isomorphic to $H \rtimes \Gamma$.
		
		From the argument above, for each $\gamma \in \Gamma$, the equality in \eqref{eq:d-lb} holds if and only if $[\overline{g},s]\in \im \alpha_{s, \gamma}$ for every $s \in  \overline{G}$ (here $\overline{g}$ is the image in $\overline{G}$ of an arbitrary element in $c_{\gamma}(G, \pi)$). Consider $[g_i, g_j]$ with $i<j$. 
		Denote the images of $g_i$ and $g_j$ by $\overline{g}_i$ and $\overline{g}_j$ in $\overline{G}$ respectively. Then the assumption that the equality \eqref{eq:d-lb} holds implies that $[\overline{g}_j, \overline{g}_i] \in \im \alpha_{\overline{g}_i, \gamma_i}$.
		Since $j>1$, $g_j$ acts trivially on $e\Z_p[\Gamma]$, and hence acts trivially on $H$. So $\im \alpha_{\overline{g}_i, \gamma_j}=1$ and $\frakB_{\gamma_j}^-(H)=1$, and then $[\overline{g}_j, \overline{g}_i]\in \im \alpha_{\overline{g}_i, \gamma_j}$ implies $[g_i, g_j]=[g_j, g_i]^{-1}=1$. The proof is completed.
	\end{proof}

\section{Proof of the function field moment theorem}\label{sect:ff-moment-proof}

\subsection{Hurwitz spaces}\label{ss:Hurwitz}
\hfill

	Given a finite group $G$ and a subset $c$ of $G$ closed under conjugation by elements of $G$ and closed under taking invertible powering, there is a Hurwitz scheme $\Hur_{G,c}^n$ defined over $\Z[|G|^{-1}]$, such that an object of $\Hur_{G,c}^n$ in the fiber $\Hur_{G,c}^n(S)$ over a scheme $S \to \Spec \Z[|G|^{-1}]$ is a triple $(f, \iota; P)$, where
	\begin{itemize}
		\item $f:X \to \PP^1_S$ is a tame Galois cover with $n$ branch points, such that $\infty \in \PP^1(S)$ is unramified and all the inertia groups are generated by elements in $c$,
		\item $\iota: \Aut f \to G$ is an isomorphism, and
		\item $P \in X(S)$ is a point lying over $\infty$.
	\end{itemize}
	See \cite[\S11]{LWZB} for more details about this Hurwitz scheme. When we fix a separable closure $\overline{\F_q(t)}$ of $\F_q(t)$ and a prime $\overline{\infty}$ of $\overline{\F_q(t)}$ lying over $\infty$ for $q\nmid |G|$, given $\F_q(t) \subset L \subset \overline{\F_q(t)}$, there is a unique prime of $L$ lying below $\overline{\infty}$.
	Then one see that there is a one-to-one correspondence between the points of $\Hur_{G,c}^n(\F_q)$ and the tuples $(L/\F_q(t), \iota)$, where 
	\begin{itemize}
		\item $L/\F_q(t)$ is a Galois subextension of $\overline{\F_q(t)}/\F_q(t)$ such that all the inertia subgroups are generated by elements in $c$ and $L/\F_q(t)$ is split completely at the prime $\infty$ of $\F_q(t)$, and 
		\item $\iota$ is an identification $\Gal(L/\F_q(t))\simeq G$.
	\end{itemize}
	We will first show in Lemma~\ref{lem:Cl-Hur} that the $\F_q$-points of $\Hur_{G,c}^n$, with appropriate choice of $G$ and $c$, are the objects of our interest.

	For a $\Z_p[\Gamma]$-module $H$, we say that $(G, \iota, \pi)$ is \emph{an extension of $\Gamma$ with kernel $H$} if $\pi: G \to \Gamma$ is a surjection and $\iota: \ker \pi \to H$ is a $\Gamma$-equivariant isomorphism, where the $\Gamma$-action on $\ker \pi$ is defined by the conjugation-by-$G$ action on $\ker \pi$ (note that since $H$ is abelian, this conjugation action factors through $\Gamma$). Two extensions $(G_1, \iota_1, \pi_1)$ and $(G_2, \iota_2, \pi_2)$ are isomorphic if there exists an isomorphism $\phi: G_1 \to G_2$ such that $\pi_1=\pi_2 \circ \phi$ and $\iota_1 \circ \iota_2^{-1}$ is the identity map on $H$. We define $\Ext_{\Gamma}(H)$ to be the set of isomorphism classes of extensions of $\Gamma$ with kernel $H$. 
	
	For $(G, \iota, \pi) \in \Ext_{\Gamma}(H)$, we let $\Aut(G, \iota, \pi)$ denote the set of isomorphisms of the extension $(G, \iota, \pi)$ to itself, and we define $c_{\pi}$ to be the set of elements of $G$ that have the same order as their image under $\pi$. 
	Let $\calA^+_{\Gamma}(q^n, \F_q(t))$ denote the set of isomorphism classes of $\Gamma$-extensions of $\F_q(t)$ such that $\rDisc K =q^n$ and $K/\F_q(t)$ is split completely at $\infty$.

\begin{lemma}\label{lem:Cl-Hur}
	Assume $H$ is a $\Z_p[\Gamma]$-module and $\Char(\F_q)$ is relatively prime to $p|\Gamma|$. Then 
	\begin{equation}\label{eq:Cl-Hur}
		\sum_{K \in \calA^+_{\Gamma}(q^n, \F_q(t))} \#\Sur_{\Gamma}(\Cl(K), H) = \sum_{(G, \iota, \pi) \in \Ext_{\Gamma}(H)} \frac{\#\Hur_{G, c_{\pi}}^n(\F_q)}{\# \Aut(G, \iota, \pi)}.
	\end{equation}
\end{lemma}

\begin{proof}	
	Regarding the right-hand side of \eqref{eq:Cl-Hur}, for any $(G, \iota, \pi) \in \Ext_{\Gamma}(H)$, a point of $\Hur_{G, c_{\pi}}^n(\F_q)$ is a split-completely-at-$\infty$ Galois extension $L/\F_q(t)$ together with a prime of $L$ lying above $\infty$ and an isomorphism $\varphi: \Gal(L/\F_q(t)) \overset{\sim}{\to} G$ such that every inertia subgroup of $L/\F_q(t)$ is generated by an element in $c_{\pi}$. Let $K$ denote the subfield of $L$ fixed by $\varphi^{-1}(\ker \pi)$, and then $\varphi$ induces an isomorphism $\phi:\Gal(K/\F_q(t)) \overset{\sim}{\to} \Gamma$, so $(K, \phi)$ is an element of $\calA_{\Gamma}^+(q^n, \F_q(t))$. Since $c_{\pi} \cap \ker \pi =1$, $L$ is an unramified extension of $K$. Also, the conjugation action by $\Gal(L/\F_q(t))$ on $\Gal(L/K)$ defines a $\Gal(K/\F_q(t))$-action on $\Gal(L/K)$, so $\phi$ and $\varphi$ gives a $\Gamma$-equivariant isomorphism $\rho: \Gal(L/K) \overset{\sim}{\to} H$. Then we obtain a map of sets.
		\begin{equation}\label{eq:Hur-map}
			\bigsqcup_{(G, \iota, \pi) \in \Ext_{\Gamma}(H)} \Hur_{G,c_\pi}^n(\F_q) \longrightarrow \left\{ (K, \phi, L, \rho) \; \Bigg| \;  \begin{aligned} 
		& (K, \phi) \in \calA_{\Gamma}^+(q^n, \F_q(t)) \\
		& L/K \text{ is unramfied and $L/\F_q(t)$ is Galois} \\ & \rho: \Gal(L/K) \overset{\sim}{\to} H \text{ is $\Gamma$-equivariant}
		\end{aligned}
			\right\}.
		\end{equation}
		This map is surjective because: given a tuple $(K, \phi, L, \rho)$ from the right-hand side, $G:=\Gal(L/\F_q(t))$, $\pi: \Gal(L/\F_q(t)) \to \Gal(K/\F_q(t)) \overset{\phi}{\to} \Gamma$ and $\iota:  \Gal(L/K)=\ker \pi  \overset{\rho}{\to} H$ give an element $(G, \iota, \pi)$ of $\Ext_{\Gamma}(H)$; and $(L, \rho \rtimes \phi)$ is an $\F_q$-point of $\Hur_{G,c_{\pi}}^n$.
		
		Suppose that two elements $(L_1, \varphi_1)$ and $(L_2, \varphi_1)$ on the left-hand side of \eqref{eq:Hur-map} give the same image $(K, \phi, L, \rho)$. Let $(G_i, \iota_i, \pi_i) \in \Ext_{\Gamma}(H)$ denote the extension defined by $(L_i, \varphi_i)$ for each $i=1,2$, (i.e., $(L_i, \varphi_i) \in \Hur_{G_i, c_{\pi_i}}^n(\F_q)$). Then $L_1=L_2=L$ and the following diagram is commutative for each $i=1,2$.
		\[\begin{tikzcd}
			1 \arrow{r} &\Gal(L/K) \arrow{r} \arrow["\rho"', "\sim"]{d} & \Gal(L/\F_q(t)) \arrow{r} \arrow["\varphi_i"']{d} & \Gal(K/\F_q(t)) \arrow{r} \arrow["\phi"', "\sim"]{d} & 1 \\
			1 \arrow{r} &H \arrow["\iota_i^{-1}"]{r} & G_i \arrow["\pi_i"]{r} & \Gamma \arrow{r} & 1
		\end{tikzcd}\]	
		It follows that $\varphi_2\circ \varphi_1^{-1}:G_1 \to G_2$ defines an isomorphism from $(G_1, \iota_1, \pi_1)$ to $(G_2, \iota_2, \pi_2)$. On the other hand, if $(L, \varphi)\in \Hur_{G, c_\pi}^n(\F_q)$ for $(G, \iota, \pi)\in \Ext_{\Gamma}(H)$ and $\alpha \in \Aut(G,\iota , \pi)$, then $(L, \alpha\circ\varphi)$ is also contained in $\Hur_{G, c_\pi}^n(\F_q)$ and has the same image as the image of $(L , \varphi)$ under \eqref{eq:Hur-map}. So, there is a bijection between the right-hand side of \eqref{eq:Hur-map} and 
		\[
			\bigsqcup_{(G,\iota, \pi) \in \Ext_{\Gamma}(H)} \faktor{\Hur_{G,c_\pi}^n(\F_q)}{\Aut(G, \iota, \pi)}.
		\]
		
		For $K \in \calA_{\Gamma}^+(q^n, \F_q(t))$, there is a bijective correspondence between $\Sur_{\Gamma}(\Cl(K),H)$ and the set of pairs $(L, \rho)$, where $L$ is an unramified extension of $K$ such that $L/\F_q(t)$ is Galois and split completely at $\infty$, and $\rho$ is a $\Gamma$-equivariant isomorphism $\Gal(L/K) \overset{\sim}{\to} H$. So there is a bijection between the right-hand side of \eqref{eq:Hur-map} and the set $\sqcup_{(K, \phi) \in \calA_{\Gamma}^+(q^n, \F_q(t))} \Sur_{\Gamma}(\Cl(K), H)$. Then the formula \eqref{eq:Cl-Hur} follows.
\end{proof}

	To prove the function field moment Theorem~\ref{thm:main-B}\eqref{item:main-B-2}, we need to estimate the number of points of $\Hur_{G,c}^n(\F_q)$, using the methods builded upon \cite{LWZB}. Briefly, applying the Grothendieck--Lefschetz trace formula, the first main term of $\#\Hur_{G,c}^n(\F_q)$ is given by $\pi_{G,c}(q,n) q^n$, where $\pi_{G,c}(q,n)$ is the number of Frobenius-fixed components of $(\Hur_{G,c}^n)_{\overline{\F}_q}$. To compute $\pi_{G,c}(q,n)$, one can analyze the braid group monodromy action on the Hurwitz space (see \cite{Wood-lifting} and \cite[\S12]{LWZB}); in particular by \cite[Proposition~12.7]{LWZB}
	\begin{equation}\label{eq:pi-b}
		\pi_{G,c}(q, n)=b(G,c,q,n)+O_G(n^{d_{G,c}(q)-2}),
	\end{equation}
	where $d_{G,c}(q)$ is the number of orbits of $q$th powering on the conjugacy classes in $c$ (under conjugation in $G$) and $b(G,c,q,n)$ is the number of some lattice points defined as below.

	A Schur covering $\phi: S \to G$ of $G$ is a stem extension such that the universal coefficient theorem map $H^2(G, \ker \phi) \to \Hom(H_2(G,\Z),\ker \phi)$ maps the class in $H^2(G, \ker\phi)$ representing $\phi$ to an isomorphism $H_2(G,\Z) \overset{\sim}{\to} \ker \phi$. Given $G$, $c$ and a Schur covering $\phi$ of $G$, the reduced Schur covering for $G, c$ and $\phi$, denoted by $\phi_c: S_c \to G$, is the quotient of $\phi: S \to G$ (i.e., $S_c$ and $\phi_c$ are obtained by taking quotient of $S$) by the normal subgroup generated by the set of commutators 
	\[
		\{[\hat{x},\hat{y} \mid \hat{x}, \hat{y}\in S, \phi(\hat{x})\in c, \text{ and } [\phi(\hat{x}), \phi(\hat{y})]=1\}.
	\]
	The kernel of $\phi_c$ is naturally a quotient of $H_2(G,\Z)$, which we denote by $H_2(G,c)$.
	Let $c/G$ denote the set of conjugacy classes of elements in $c$ and let $\Z^{c/G}$ denote the free abelian group generated by elements of the set $c/G$. Then the map $\Z^{c/G} \to G^{\ab}$ sending the generator for the class of $g\in c$ to the image of $g$ under $G \to G^{\ab}$ is a group homomorphism. 
	For each conjugacy class $\gamma \in c/G$, we pick an element $x_{\gamma}$ in $\gamma$ and a lift $\widehat{x_{\gamma}}$ of $x_{\gamma}$ in $S_c$. Then, if $q$ is prime to $|G|$, we define a group homomorphism $W_{q^{-1}}: \Z^{c/G} \to \ker \phi_c$ by sending the generator corresponding to $\gamma$ to $\widehat{x_{\gamma}}^{-1/q} \widehat{x_{\gamma}^{1/q}} \in \ker \phi_c$.
	 Write $\Z_{\equiv q, n, \geq M}^{c/G}$ for the sublattice of $\Z^{c/G}$ consisting of elements satisfying: 1) each coordinates is positive; 2) all the coordinates sum up to $n$; and 3) if $\gamma_1,  \gamma_2 \in c/G$ such that elements in $\gamma_2$ are the $q$th power of elements in $\gamma_1$, then the coordinates corresponding to $\gamma_1$ and $\gamma_2$ are equal. For an element $a\in \ker \phi_c$, we write $\nr_{q-1}(a)$ for the number of $x\in \ker\phi_c$ such that $x^{q-1}=a$. Then define
	\begin{equation}\label{eq:defofb}
		b(G,c,q,n):=\sum_{\underline{m} \in \ker \left(\Z_{\equiv q, n, \geq 0}^{c/G} \to G^{\ab} \right)} \nr_{q-1}(W_{q^{-1}}(\underline{m})).
	\end{equation}
	Here, $H_2(G,c)$ and $b(G,c,q,n)$ do not depend on the choice of the Schur covering $\phi$ we start with.

\subsection{Proof of Theorem~\ref{thm:main-B}\eqref{item:main-B-2}}
\hfill

		\begin{lemma}\label{lem:b}
		Let $\Gamma$ be a finite abelian group and $e$ a primitive idempotent of $\Q_p[\Gamma]$. Let $H_i$ be a finite $e\Z_p[\Gamma]$-module for each $i=1,2$, such that there is a surjective homomorphism $\rho^\circ:H_1 \to H_2$. Let $G_i$ be $H_i\rtimes \Gamma$, $\pi_i$ the natural surjection $G_i \to \Gamma$ with kernel $H_i$, and $\rho$ the surjection $G_1 \to G_2$ defined by $\rho|_{H_1}=\rho^{\circ}$ and $\rho|_{\Gamma}: \Gamma \overset{\sim}{\to} \Gamma$. Let $c_i:=\cup_{\gamma\in \Gamma} c_{\gamma}(G_i, \pi_i)$. Assume that the elements in $c_i$ generate $G_i$. Then the following statements hold for any prime power $q$ such that $p\nmid (q-1)q$.
		\begin{enumerate}
			\item \label{item:b-1} $d_{G_1,c_1}(q) \geq d_{G_2,c_2}(q)$, and the equality holds if and only if $\ker \rho$ is contained in $\frakB^-_{\gamma}(H_1)\frakB^+_{\gamma}(H_1)$ for each $\gamma \in \Gamma$.
			\item \label{item:b-2} If $d_{G_1, c_1}(q)=d_{G_2, c_2}(q)$, then $b(G_1, c_1, q, n) = b(G_2, c_2, q, n)$.
		\end{enumerate}
	\end{lemma}
	
	\begin{proof}
		Suppose $g_1, g_2\in c_{\gamma}(G_i, \pi_i)$ such that $g_1^{q^n}=g_2$ for some integer $n$. Since $\gamma=\pi_i(g_1)=\pi_i(g_2)$, $g_1^{q^n}=g_2$ implies $\gamma^{q^n}=\gamma$. Then because $|g_1|=|\gamma|$, we have $|g_1|\mid q^n-1$, so $g_1=g_1^{q^n}=g_2$. Thus, we showed that elements in $c_{\gamma}(G_i, \pi_i)$ lie in pairwisely distinct $q$-th powering orbits in $c_i$, and hence
		\begin{equation}\label{eq:formula-d}
			d_{G_i,c_i}(q)= \sum_{\gamma} d_{\gamma}(G_i, \pi_i),
		\end{equation}
		where the sum runs over a set of representatives of the $q$-th powering orbits of $\Gamma$.
	
		By Corollary~\ref{cor:d=isom-module}, $d_{\gamma}(G_i, \pi_i)$ is the number of the $\Gamma$-orbits of $H_i/\frakB^-_{\gamma}(H_i)\frakB^+_{\gamma}(H_i)$; and by Lemma~\ref{lem:basic-AB}\eqref{item:basic-AB-3}, $\rho$ induces a $\Gamma$-equivariant surjection from $H_1/\frakB^-_{\gamma}(H_1)\frakB^+_{\gamma}(H_1)$ to $H_2/\frakB^-_{\gamma}(H_2)\frakB^+_{\gamma}(H_2)$. So we have $d_{\gamma}(G_1, \pi_1)\geq d_{\gamma}(G_2, \pi_2)$, and the equality holds if and only if 
		\begin{equation*}
			H_1/\frakB^-_{\gamma}(H_1)\frakB^+_{\gamma}(H_1) \simeq H_2/\frakB^-_{\gamma}(H_2) \frakB^+_{\gamma}(H_2),
		\end{equation*}
		i.e., if and only if $\ker \rho \subset \frakB^-_{\gamma}(H_1)\frakB^+_{\gamma}(H_1)$. Then the statement \eqref{item:b-1} follows by the formula \eqref{eq:formula-d}. 
		
		For the rest of the proof, we assume $d_{G_1,c_1}(q)=d_{G_2,c_2}(q)$. 
		
		\emph{Claim 1: $\rho$ induces an isomorphism $G_1^{\ab} \overset{\sim}{\longrightarrow}G_2^{\ab}$.} 
		
		The assumption that $c_1$ generates $G_1$ implies that $G_1^{\ab}$ is generated by the images of elements of $c_1$. Since $\Gamma$ is abelian, it is a quotient of $G_1^{\ab}$. 
		For any $\gamma \in \Gamma$,  elements of $c_{\gamma}(G_1, \pi)$ has order $|\gamma|$,
		so the exponent of $G_1^{\ab}$ equals the exponent of $\Gamma$. Let $\gamma$ be an element of $\Gamma$ such that $|\gamma|$ equals the exponent of $\Gamma$, and let $M$ denote $\ker (G_1^{\ab} \to \Gamma)$, which can be viewed as a $\Z_p[\Gamma]$-module with the trivial $\Gamma$-action. Then $\gamma$ acts trivially on $M$ and $M$ has exponent dividing $|\gamma|$, so one can check $\frakB^-_{\gamma}(M)=\frakB^+_{\gamma}(M)=0$. Then by the assumption that $d_{G_1,c_1}(q)=d_{G_2,c_2}(q)$ and applying Lemma~\ref{lem:basic-AB}\eqref{item:basic-AB-3} to the quotient map $H_1 \to M$, we have 
		\[
			\ker \rho \subset \frakB^-_{\gamma}(H_1)\frakB^+_{\gamma}(H_1) \subset [G_1, G_1],
		\] 
		so we proved Claim~1.

		\emph{Claim 2: $\ker(\Z^{c_i/G_i}_{\equiv q, n, \geq 0} \to G_i^{\ab})=\ker(\Z^{c_i/G_i}_{\equiv q, n, \geq 0} \to \Gamma)$ for each $i=1,2$, where the map $\Z^{c_i/G_i}_{\equiv q, n, \geq 0} \to \Gamma$ is the composition of $\Z^{c_i/G_i}_{\equiv q, n, \geq 0} \to G_i^{\ab}$ and $G_i^{\ab} \to \Gamma$. }

		Because $G_i$ is the semidirect product $H_i \rtimes \Gamma$, its abelianization $G_i^{\ab}$ is the direct product $(H_i)_{\Gamma} \times \Gamma$, where $(H_i)_{\Gamma}$ is the $\Gamma$-coinvariant of $H_i$.
		Let $g\in c_i$ and $m$ be the smallest positive integer such that $g^{q^m}$ is conjugate to $g$. Then $g, g^q, g^{q^2}, \ldots, g^{q^{m-1}}$ lie in distinct conjugacy classes and they are all the conjugacy classes in the $q$-th powering orbits of $g$, so their corresponding coordinates in an element of $\Z^{c_i/G_i}_{\equiv q, n, \geq 0}$ are equal to each other. We let $\underline{e_g}$ denote the element in $\Z^{c_i/G_i}$ such that the coordinates corresponding to $g, g^q, \ldots, g^{q^m-1}$ are 1 and the other coordinates are 0. So each element of $\Z^{c_i/G_i}_{\equiv q, n, \geq 0}$ can be written as $\sum_{g} a_g \underline{e_g}$, where $a_g \in \Z$ and the sum runs over a set of representatives of the $q$-th powering orbits of $c_i/G_i$. 
		
		Let $(x, \gamma)$ denote the image of $g$ in $G_i^{\ab}$, where $x\in (H_i)_{\Gamma}$ and $\gamma=\pi_i(g)\in \Gamma$. By definition of $c_i$, we have $|g|=|\gamma|$, so $|x|$ divides $|\gamma|$.
		Because $\Gamma$ is abelian, $\pi(g^{q^m}) \sim \pi(g)$ implies $\gamma^{q^m}=\gamma$, so $|\gamma|\mid q^m-1$. So the image of $\underline{e_g}$ in $G_i^{\ab}$ is 
		\begin{equation}\label{eq:g-power}
			(x, \gamma)^{1+q+\cdots q^{m-1}}=(x^{1+q+\cdots q^{m-1}}, \gamma^{1+q+\cdots q^{m-1}}).
		\end{equation}
		Since $(H_i)_{\Gamma}$ is an abelian $p$-group and $p\nmid q-1$, we have $m>1$, so it follows by $|x|\mid q^m-1$ that $1+q+\cdots +q^{m-1}=\frac{q^m-1}{q-1}$ is a multiple of $|x|$. Thus, the first coordinate in \eqref{eq:g-power} is zero, and this is true for any $g \in c_i$. So Claim 2 follows.
		
		Recall that we proved that $d_{G_1,c_1}(q)=d_{G_2,c_2}(q)$ implies $d_{\gamma}(G_1, \pi_1)=d_{\gamma}(G_2, \pi_2)$ for each $\gamma\in \Gamma$. So for each $\gamma$, the quotient map $G_1 \to G_2$ defines a bijection between the conjugacy classes of elements in $c_{\gamma}(G_1, \pi_1)$ and the conjugacy classes of elements in $c_{\gamma}(G_2, \pi_2)$, and then we have a bijection $\phi: \Z^{c_1/G_1}_{\equiv q, n, \geq 0} \to \Z^{c_2/G_2}_{\equiv q, n, \geq 0}$. One can check that $\phi$ is compatible with the maps $\Z^{c_i/G_i}_{\equiv q, n, \geq 0} \to \Gamma$, i.e., the following diagram commutes.
		\[\begin{tikzcd}
			\Z^{c_1/G_1}_{\equiv q, n, \geq 0} \arrow{rr} \arrow["\phi", swap]{d} && \Gamma \\
			\Z^{c_2/G_2}_{\equiv q, n, \geq 0} \arrow{rru} &&
		\end{tikzcd}\]
		Then by Claim 2, $\phi$ defines a bijection between 
		\begin{equation}\label{eq:phi}
		\begin{tikzcd}
			\phi: \ker(\Z^{c_1/G_1}_{\equiv q, n, \geq 0} \to G_1^{\ab}) \arrow["1-1"]{r} &\ker(\Z^{c_2/G_2}_{\equiv q, n, \geq 0} \to G_2^{\ab}).
		\end{tikzcd}
		\end{equation}
		
		Let $\Gamma'$ denote the quotient of $\Gamma$ modulo the Sylow $p$-subgroup of $\Gamma$, and $c_{\gamma'}$ denote the set of elements of $\Gamma'$.  By \cite[Lemma~12.10]{LWZB} \footnote{In the statement of \cite[Lemma~12.10]{LWZB}, $H$ is required to be an admissible $\Gamma$-group, but this condition can be removed because it is not used in the proof. Here, we apply this lemma to $G_i=H_i \rtimes \Gamma=(H_i \times \Gamma_p)\rtimes \Gamma'$, where $H_i\times \Gamma_p$ has order coprime to $|\Gamma'|$ but is not necessarily an admissible $\Gamma'$-group (for example, when $\Gamma_p$ is nontrivial, $H_i \times \Gamma_p$ is not admissible; see the definition of \emph{admissible $\Gamma$-groups} in \cite[\S2]{LWZB}).}, there are Schur coverings $S_i \to G_i$ for $i=1,2$ and $S_{\Gamma'}\to \Gamma'$ satisfying the following diagram for each $i$
		\begin{equation}\label{eq:com-schurcovering}
		\begin{tikzcd}
			S_i \arrow{r} \arrow["f_i"']{d} & G_i \arrow{d} \\
			S_{\Gamma'} \arrow{r} & \Gamma', 
		\end{tikzcd}
		\end{equation}
		and moreover, the order of the kernel of $f|_{\ker(S_i \to G_i)}$, a map from $\ker(S_i \to G_i)$ to $\ker(S_{\Gamma'} \to \Gamma')$, is a power of $p$. 
		Let $Q_i$ (resp. $Q_{\Gamma'}$) denote the subgroup of $S_i$ (resp. $S_{\Gamma'}$) generated by all commutators $[\hat{x}, \hat{y}]$, where $\hat{x}, \hat{y}$ are elements of $S_i$ (resp. $S_{\Gamma'}$) and their images in $G_i$  (resp. $\Gamma'$), denoted by $x$ and $y$ respectively, commutes and $x \in c_i$ (resp. $x\in c_{\Gamma'}$). (Note that the commutator $[\hat{x}, \hat{y}]$ does not depend on the choice of lifts $\hat{x}$ and $\hat{y}$ since the Schur coverings are central extensions.) Then $Q_i \subseteq \ker(S_i \to G_i)$ and $Q_{\Gamma'} \subseteq \ker(S_{\Gamma'} \to \Gamma')$, and one can check that the image of $Q_i$ under the map $f_i:S_i \to S_{\Gamma}'$ is contained in $Q_{\Gamma'}$. On the other hand, suppose $x \in c_{\Gamma'}$ and $y \in \Gamma'$; since $G_i=(H_i \times \Gamma_p) \rtimes \Gamma'$, the natural splitting $\Gamma' \hookrightarrow G_i$ maps $x,y$ to $\tilde{x}, \tilde{y} \in G_i$, so $\tilde{x} \in c_i$ and $\tilde{x}$ commutes with $\tilde{y}$. Then picking lifts $\hat{\tilde{x}},\hat{\tilde{y}} \in S_{i}$ of $\tilde{x}, \tilde{y}$ respectively, the image of $[\hat{\tilde{x}}, \hat{\tilde{y}}]$ in $S_{\Gamma'}$ is the element of $Q_i$ defined by $x,y$. So we see 
		\begin{equation}\label{eq:image-Q}
			f_i(Q_i)=Q_{\Gamma'}.
		\end{equation}

		Let $\overline{S_i}:=S_i/Q_i$ and $\overline{S_{\Gamma'}}:=S_{\Gamma'}/Q_{\Gamma'}$ be the reduced Schur covering of $G_i$ and $\Gamma'$. Then the diagram \eqref{eq:com-schurcovering} defines a map 
		\[
			\kappa_i: \ker(\overline{S_i} \to G_i) \to \ker(\overline{S_{\Gamma'}} \to \Gamma')
		\]
		for each $i=1,2$. By \eqref{eq:image-Q}, we have the following commutative diagram in which rows are exact.
		\[\begin{tikzcd}
			1 \arrow{r} & Q_i \arrow{r} \arrow["f|_{Q_i}", two heads]{d} &\ker(S_i \to G_i) \arrow{r} \arrow["f|_{\ker(S_i \to G_i)}"]{d} & \ker(\overline{S_i}\to G_i) \arrow{r} \arrow["\kappa_i"]{d} & 1 \\
			1 \arrow{r} & Q_{\Gamma'} \arrow{r} &\ker(S_{\Gamma'}\to \Gamma') \arrow{r} &\ker(\overline{S_{\Gamma'}}  \to \Gamma')\arrow{r} &1 .
		\end{tikzcd}\]
		By the snake lemma, $\ker \kappa_i$ is a quotient group of $\ker f |_{\ker(S_i \to G_i)}$, so the order of $\ker \kappa_i$ is a power of $p$. Since $p\nmid q-1$,we have 
		\begin{equation}\label{eq:nr}
			\nr_{q-1}(x)=\nr_{q-1}(\kappa_i(x)) \text{  for any $x\in \ker(\overline{S_i} \to G_i)$}. 
		\end{equation}
		We define $\overline{W}_{q^{-1}}^i: \Z^{c_i/G_i} \to \ker(\overline{S_{\Gamma'}} \to \Gamma')$ to be the composition of $W_{q^{-1}}^i$ and $\kappa_i$. Then, one can check that $\phi$ fits into the following commutative diagram
		\begin{equation}\label{eq:com-Z}
		\begin{tikzcd}
			\Z_{\equiv q, n, \geq 0}^{c_1/G_1} \arrow["\overline{W}_{q^{-1}}^1"]{rr} \arrow["\phi", swap]{d} & &\ker(\overline{S_{\Gamma'}} \to \Gamma') \\
			\Z_{\equiv q, n, \geq 0}^{c_2/G_2} \arrow["\overline{W}_{q^{-1}}^2", swap]{rru} &&
		\end{tikzcd}
		\end{equation}
		and therefore
		 \begin{eqnarray*}
		 	b(G_1, c_1, q, n) &=& \sum_{\underline{m} \in \ker(\Z_{\equiv q, n, \geq 0}^{c_1/G_1} \to G_1^{\ab})} \nr_{q-1} (W^1_{q^{-1}}(\underline{m}))\\
			&=& \sum_{\underline{m} \in \ker(\Z_{\equiv q, n, \geq 0}^{c_1/G_1} \to G_1^{\ab})} \nr_{q-1} (\overline{W}^1_{q^{-1}}(\underline{m})) \\
			&=& \sum_{\underline{m} \in \ker(\Z_{\equiv q, n, \geq 0}^{c_1/G_1} \to G_1^{\ab})} \nr_{q-1} (\overline{W}^2_{q^{-1}}(\phi(\underline{m}))) \\
			&=& \sum_{\underline{m} \in \ker(\Z_{\equiv q, n, \geq 0}^{c_2/G_2} \to G_2^{\ab})} \nr_{q-1} (\overline{W}^2_{q^{-1}}(\underline{m})) \\
			&=& b(G_2, c_2, q, n).
		 \end{eqnarray*}
		 Here the first equality follows by the definition of $b(G_1, c_1, q, n)$, the second equality uses \eqref{eq:nr}, the third uses the commutative diagram \eqref{eq:com-Z}, the fourth uses \eqref{eq:phi}, and the last equality follows by definition and \eqref{eq:nr}.
	\end{proof}

	\begin{proposition}\label{prop:final-prop}
		Let $\Gamma$ be a finite abelian group and $e$ a nontrivial primitive idempotent of $\Q_p[\Gamma]$. Let $H_1$ and $H_2$ be finite $e\Z_p[\Gamma]$-modules with a surjective homomorphism $\rho^{\circ}: H_1 \to H_2$. Let $I_e$ be the ideal of $e\Z_p[\Gamma]$ defined in Definition~\ref{def:thresfold-ideal}. Then
		\begin{equation}\label{eq:final-prop}
			\lim_{N \to \infty}  \lim_{\substack{q \to \infty \\ p\nmid q(q-1)\\ \gcd(q,|\Gamma|)=1}} \frac{\sum\limits_{0\leq n \leq N} \sum\limits_{K\in \calA^+_{\Gamma}(q^n, \F_q(t))} \#\Sur_{\Gamma} (\Cl(K), H_1) }{\sum\limits_{0\leq n \leq N} \sum\limits_{K\in \calA^+_{\Gamma}(q^n, \F_q(t))} \#\Sur_{\Gamma} (\Cl(K), H_2)}=\begin{cases}
				\dfrac{|H_2|}{|H_1|} & \text{ if }\ker \rho^{\circ} \subseteq I_eH_1 \\
				\infty &\text{otherwise.}
			\end{cases}
		\end{equation}
	\end{proposition}
	
	\begin{proof}
		For each $i=1,2$, let $G_i$ denote $H_i \rtimes \Gamma$, $\pi_i$ the natural surjection $G_i \to \Gamma$, $\iota_i$ the natural embedding $H_i \hookrightarrow G_i$, $c_i$ denote the elements of $G_i$ that have the same order as their image under $\pi_i$.
		Since the idempotent $e$ is assumed to be nontrivial, there exists $\gamma \in \Gamma$ such that $\gamma$ acts nontrivially on $e \Z_p[\Gamma]$, so $\frakB^-_{\gamma}(e\Z_p[\Gamma]) \neq 0$. Then by Lemma~\ref{lem:gamma-ann}, $\frakB^+_{\gamma}(e\Z_p[\Gamma])=0$, so $\frakB^+_{\gamma}(H_i)=0$ and $\frakA^-_{\gamma}(H_i)=H_i$ for each $i =1 ,2$. Note that for each $h \in \frakA_{\gamma}^-(H_i)$, the element $(h, \gamma) \in c_i$, so it follows by $(h,1)=(h, \gamma) (1, \gamma^{-1})$ that the subgroup of $G_i$ generated by elements in $c_i$ contains $H_i$. Also, choosing a splitting $\Gamma \hookrightarrow H_i \rtimes \Gamma$, all the elements of $\Gamma$ is contained in $c_i$. So $c_i$ generates the group $G_i$, and hence $\Hur_{G_i, c_i}^n$ is not empty.
		
		We let $\pi_{G,c}(q, n)$ denote the number of $\Frob_{(\Hur_{G,c}^n)_{\F_q}}$-fixed components of $(\Hur_{G,c}^n)_{\F_q}$. By \cite[Corollary~12.9]{LWZB}, for each $i=1,2$ and $(G, \iota, \pi) \in \Ext_{\Gamma}(H_i)$, $\pi_{G,c_{\pi}}(q,n)$ is either 0 or $O_G(n^{d_{G,c_{\pi}}(q)-1})$.
		Note that 
		\[
			d_{G, c_{\pi}}(q) = \sum_{\gamma} d_{\gamma}(G, \pi)
		\]
		where the sum runs over a set of representatives of the $q$-th powering orbits of $\Gamma$. So for each $i=1,2$, by Lemma~\ref{prop:d-semi}, if there exists some $(G, \iota, \pi)\in \Ext_{\Gamma}(H_i)$ such that $\pi_{G, c_{\pi}}(q,n)>0$, then $\pi_{G_i, c_i}(q,n)>0$ and 
		\[
			\pi_{G_i, c_i}(q,n) = O_G(n^{d_{G_i, c_i}(q)-1}), \quad \text{and}\quad \pi_{G, c_{\pi}}(q, n) = O_G(n^{d_{G_i, c_i}(q)-2}) \text{ if $(G, \iota, \pi)\neq (G_i, \iota_i, \pi_i)$}.
		\]

		Let $N_i$ denote the largest integer such that $N_i\leq N$ and $\pi_{G_i, c_i}(q, N_i)>0$. Then by the arguments (about how to apply the trace formula and the Weil bounds) in the proof of \cite[Theorem~1.4]{LWZB} and by Lemma~\ref{lem:Cl-Hur}, we have
		\begin{eqnarray}
			&& \sum_{0 \leq n \leq N}\sum_{K \in \calA^+_{\Gamma}(q^n, \F_q(t))} \Sur_{\Gamma}(\Cl(K), H_i) \nonumber \\
			&=& \sum_{0 \leq n \leq N}\sum_{(G, \iota, \pi)\in \Ext_{\Gamma}(H_i)} \frac{\#\Hur_{G, c_{\pi}}^n(\F_q)}{\#\Aut(G, \iota, \pi)}  \nonumber\\
			&=& \frac{\pi_{G_i, c_i}(q, N_i) q^{N_i}}{\#\Aut(G_i, \iota_i, \pi_i)} + \sum_{\substack{(G, \iota, \pi)\in \Ext_{\Gamma}(H_i) \\ (G, \iota, \pi)\neq (G_i, \iota_i, \pi_i)}} \sum_{0 \leq n \leq N}\frac{\#\Hur_{G, c_{\pi}}^n(\F_q)}{\#\Aut(G, \iota, \pi)}  \nonumber\\
			&=& \frac{\pi_{G_i, c_i}(q, N_i) q^{N_i}}{\#\Aut(G_i, \iota_i, \pi_i)} + \sum_{\substack{(G, \iota, \pi)\in \Ext_{\Gamma}(H_i) \\ (G, \iota, \pi)\neq (G_i, \iota_i, \pi_i)}}  E_i(N, q, G, c_{\pi}) q^{N_i-\frac{1}{2}}
			\label{eq:main-1},
		\end{eqnarray}
		where $E_i(N, q, G, c_{\pi}) = O_{N, G}(1)$, and the sum is taken over a finite set since $\Ext_{\Gamma}(H_i)$ is finite.
		By \cite[Corollary~12.9]{LWZB}, there exists some positive integer $r_i$ such that 
		\[
			\pi_{G_i,c_i}(q,N_i)=r_i N_i^{d_{G_i, c_i}(q)-1} +O_{G_i}(N_i^{d_{G_i, c_i}(q)-2}).
		\]
		So by Lemma~\ref{lem:b}\eqref{item:b-1}, if $\ker \rho^{\circ} \not\subset I_eH_1=\frakB_{\gamma}^-(H_1)\frakB_{\gamma}^+(H_1)$, then $d_{G_1, c_1}(q)>d_{G_2, c_2}(q)$, and hence, in this case, the proposition follows by \cite[Corollary~12.9]{LWZB}.

		For the rest of the proof, assume $\ker \rho^{\circ} \subseteq I_eH_1$, then $d_{G_1, c_1}(q)=d_{G_2, c_2}(q)$, then $b(G_1, c_1, q, n)=b(G_2, c_2, q, n)$ for any $n$ by Lemma~\ref{lem:b}\eqref{item:b-2}. So it follows by \cite[Proposition~12.7]{LWZB} and the formula~\eqref{eq:main-1} that $N_1=N_2$ and the left-hand side of \eqref{eq:final-prop} equals
		\begin{equation}\label{eq:main-aut}
			\frac{\#\Aut(G_2, \iota_2, \pi_2)}{\#\Aut(G_1, \iota_1, \pi_1)}.
		\end{equation}	
		It suffices to show that $\#\Aut(G_1, \iota_1, \pi_1)/\#\Aut(G_2, \iota_2, \pi_2)=|H_1|/|H_2|$. By Lemma~\ref{lem:max-ideal}, the $\Gamma$-action on $e\Z_p[\Gamma]$ factors through a cyclic quotient, so we can choose a set of generators $\{\gamma_1, \cdots, \gamma_d\}$ of the abelian group $\Gamma$ such that $\Gamma=\prod_{i=1}^d \langle \gamma_i \rangle$ and $\gamma_j$ acts trivially on $e\Z_p[\Gamma]$ for each $j\geq 2$.
		
		We claim that 
		\begin{equation}\label{eq:Aut}
			\#\Aut(G_i, \iota_i, \pi_i)=\# \frakA^-_{\gamma_1}(H_i)\prod_{j=2}^d \#\frakA^+_{\gamma_1}(H_i)[|\gamma_j|],
		\end{equation}
		where $\frakA^+_{\gamma_1}(H_i)[|\gamma_j|]$ denotes the $|\gamma_j|$-torsion elements of $\frakA^+_{\gamma_1}(H_i)$. We will prove the formula~\eqref{eq:Aut} for $i=1$, and then the case when $i=2$ similarly follows.
		Since $G_1=H_1 \rtimes \Gamma$, one can check that $\#\Aut(G_1, \iota_1, \pi_1)$ equals the number of homomorphic splitting $\Gamma \hookrightarrow G_1$ of $\pi_1$. Since $\Gamma$ is abelian, $\{\gamma_j \mapsto g_j\}_{j=1}^d$ defines a homomorphic splitting if and only if $g_j \in \pi_1^{-1}(\gamma_j)$ such that $|g_j|=|\gamma_j|$ and $g_j g_k=g_k g_i$ for any $1\leq j, k \leq d$. By the multiplication rule of semidirect products, $g_1\in \pi_1^{-1}(\gamma_1)$ satisfies $|g_1|=|\gamma_1|$ if and only if $g_1$ is written as $(h_1, \gamma_1)\in H_1 \rtimes \Gamma$ such that $h_1\in \frakA^-_{\gamma_1}(H_1)$. For any $j\geq 2$, since $\gamma_j$ acts trivially on $H_1$, $g_j\in \pi_1^{-1}(\gamma_j)$ satisfies $|g_j|=|\gamma_j|$ if and only if $g_j=(h_j , \gamma_j)$ for $h_j \in H_1[|\gamma_j|]$. Moreover, we compute for any $j, k\geq 2$ and any $a, b \in H_1$
		\begin{eqnarray*}
			(a, \gamma_1) (b, \gamma_j) &=& (a\gamma_1(b), \gamma_1 \gamma_j) \\
			(b, \gamma_j) (a, \gamma_1) &=& (ab, \gamma_1\gamma_j) \\
			(a, \gamma_j) (b, \gamma_k) &=& (ab, \gamma_j\gamma_k),
		\end{eqnarray*}
		from which we conclude that $\{\gamma_j \mapsto g_j\}_{j=1}^d$ defines a homomorphic splitting if and only if $g_j=(h_j, \gamma_j)$ such that 
		\[
			h_1 \in \frakA^-_{\gamma_1}(H_1)\quad \text{ and }\quad h_j \in \frakA^+_{\gamma_1}(H_1)[|\gamma_j|], \, \forall j>1.
		\]
		Thus, the formula~\eqref{eq:Aut} immediately follows.
		
		Since the idempotent $e$ is assumed to be nontrivial, the $\Gamma$ acts nontrivially on $e\Z_p[\Gamma]$, so $\gamma_1$ acts nontrivially on $e\Z_p[\Gamma]$ and hence $\frakB_{\gamma_1}^-(e\Z_p[\Gamma])\neq 0$. Then, by Lemma~\ref{lem:gamma-ann}, $\frakB^+_{\gamma_1}(e\Z_p[\Gamma])=0$. So, every element in $H_1$ and $H_2$ are annihilated by $\sum_{m=1}^{|\gamma_1|} \gamma_1^m$, so $\frakA^-_{\gamma_1}(H_i)=H_i$ and $\frakA^{+}_{\gamma_1}(H_i)=\frakA^0_{\gamma_1}(H_i)$ for each $i=1,2$. The assumption that $\ker\rho^{\circ} \subset I_eH_1$ implies that $\ker \rho^{\circ} \subset \frakB_{\gamma_1}^-(H_1)$. Then by Lemma~\ref{lem:basic-AB}\eqref{item:basic-AB-3}
			and Lemma~\ref{lem:isom-module}, we have 
			\[
				\frakA^0_{\gamma_1}(H_1)\simeq H_1/\frakB_{\gamma_1}^-(H_1) \simeq H_2/\frakB_{\gamma_1}^-(H_2) \simeq \frakA^0_{\gamma_1}(H_2).
			\]
			So it follows by \eqref{eq:Aut} that the formula in \eqref{eq:main-aut} is equal to $|H_2|/|H_1|$.
	\end{proof}

	Now we have all the ingredients to prove Theorem~\ref{thm:main-B}\eqref{item:main-B-2}.
	
	\begin{proof}[Proof of Theorem~\ref{thm:main-B}\eqref{item:main-B-2}]
		Let $H_1$ denote the $I_e$-closure of $M$ and $H_2:=(e\Z_p[\Gamma]/I_e)^r$. Then $I_eH_1=M$, $H_1[I_e]=(H_1)_{/I_e}\simeq H_2$ and there is a natural surjection $\rho^{\circ}: H_1 \to H_2$ whose kernel is $M=I_eH_1$. 
		By $e\Z_p[\Gamma]$ is a discrete valuation ring, there exists $\gamma\in \Gamma$ such that $I_e=\rho_{\Z_p[\Gamma],e}((1-\gamma, \sum_{i=1}^{|\gamma|} \gamma^i ))$. 
		Then for this $\gamma$,  $\frakA^0_{\gamma}(H_2)=H_2$ because $I_eH_2=0$. For any proper submodule $H^\circ$ of $H_2$, Lemma~\ref{lem:basic-AB}\eqref{item:basic-AB-2} shows that $\frakA^0_{\gamma}(H^{\circ})=H^{\circ}$, so the number of $\Gamma$-orbits of $\frakA^0_{\gamma}(H^{\circ})$ is strictly less that the number of $\Gamma$-orbits of $\frakA^0_{\gamma}(H_2)$. 
		Writing $\pi^{\circ}: H^{\circ} \rtimes \Gamma \to \Gamma$ and $\pi_2:H_2 \rtimes \Gamma \to \Gamma$, by Corollary~\ref{cor:d=isom-module} and \eqref{eq:formula-d}, we have $d_{H_2 \rtimes \Gamma, \pi_2}(q)>d_{H^{\circ} \rtimes \Gamma, \pi^{\circ}}(q)$ for any $q$.
		Then we apply the same argument as in the proof of Proposition~\ref{prop:final-prop} and obtain
		\[
			\lim_{N \to \infty} \lim_{\substack{q \to \infty \\ p\nmid q(q-1)\\ \gcd(q,|\Gamma|)=1}} \frac{\sum\limits_{0\leq n \leq N} \sum\limits_{K\in \calA^+_{\Gamma}(q^n, \F_q(t))} \#\Sur_{\Gamma} (\Cl(K), H^{\circ}) }{\sum\limits_{0\leq n \leq N} \sum\limits_{K\in \calA^+_{\Gamma}(q^n, \F_q(t))} \#\Sur_{\Gamma} (\Cl(K), H_2)}=0.
		\]
		Then by an inclusion-exclusion argument, we have 
		\[
			\lim_{N \to \infty} \lim_{\substack{q \to \infty \\ p\nmid q(q-1)\\ \gcd(q,|\Gamma|)=1}} \frac{\sum\limits_{0\leq n \leq N} \sum\limits_{K\in \calA^+_{\Gamma}(q^n, \F_q(t))} \#\Hom_{\Gamma} (\Cl(K), H_2) }{\sum\limits_{0\leq n \leq N} \sum\limits_{K\in \calA^+_{\Gamma}(q^n, \F_q(t))} \#\Sur_{\Gamma} (\Cl(K), H_2)}=1.
		\]
		Thus, the theorem follows by Proposition~\ref{prop:final-prop} and Proposition~\ref{prop:weight-formula}.
	\end{proof}

	Theorem~\ref{thm:main-B}\eqref{item:main-B-2} is only for nontrivial primitive idempotents. For the trivial primitive idempotent $e_0=(\sum_{\gamma \in \Gamma} \gamma)/|\Gamma|$, we prove the following proposition.

	\begin{proposition}\label{prop:trivial-idem}
		 Let $\Gamma$ be a finite abelian group, and $K$ is a $\Gamma$-extension of $\Q$ or a $\Gamma$-extension of $\F_q(t)$ that is completely split at $\infty$. Assume $e$ is the trivial primitive idempotent, that is $e=(\sum_{\gamma \in \Gamma} \gamma)/|\Gamma|$. Let $p^n$ denote the exponent of $\Gamma_p$. Then 
		 \[
		 	I_e(e\Cl(K)) = p^n \Cl(K)_{\Gamma} \quad \text{and} \quad |I_e(e\Cl(K))| \leq |\wedge^2 \Gamma_p|.
		\]
	 \end{proposition}

	 \begin{remark}
	 	If $\Gamma_p$ is cyclic, this proposition implies that $I_e(e\Cl(K))$ is trivial, which also follows from the fact that the norm map $\sum_{\gamma \in \Gamma} \gamma$ annihilates the class group $\Cl(K)$.
	 \end{remark}
	 
	 \begin{proof}
	 	Let $Q$ denote $\Q$ or $\F_q(t)$. The first claim immediately follows by the definition of $e\Cl(K)$ and $I_e$. By class field theory, $\Cl(K)$ is isomorphic to the Galois group of the maximal unramified extension (in the number field case) or the maximal unramified and completely-split-at-$\infty$ extension (in the function field case) of $K$. Since $e\Cl(K)$ is a quotient of $\Cl(K)$, $e\Cl(K)$ corresponds to an extension of $K$, and we denote it by $L/K$. 		Consider the abelianization $\Gal(L/Q)^{\ab}$ of $\Gal(L/Q)$, and let $L^{\ab}/Q$ denote the subextension of $L/Q$ that corresponds to $\Gal(L/Q)^{\ab}$.
		Because $\Gamma$ is abelian and is a quotient of $\Gal(L/Q)$, $K$ is contained in $L^{\ab}$. Then as $L/K$ is unramified, $\calT_{\frakp}(L/Q)\simeq \calT_{\frakp}(L^{\ab}/Q)\simeq \calT_{\frakp}(K/Q)$ for every prime $\frakp$ of $Q$. Note that $\Q$ (resp. $\F_q(t)$) does not have any nontrivial unramified (resp. unramified and completely-split-at-$\infty$) extension. So $\Gal(L/Q)$ equals the normal subgroup generated by all $\calT_{\frakp}(L/Q)$ for prime $\frakp$ of $Q$, and $\Gal(L^{\ab}/Q)$ is generated by $\calT_{\frakp}(L^{\ab}/Q)$ for all $\frakp$. Then as $\calT_{\frakp}(L^{\ab}/Q)\simeq \calT_{\frakp}(K/Q) \subset \Gamma$, we see that the exponent of $\Gal(L^{\ab}/Q)$ equals $p^n$.
	 	Since $\Gamma$ acts trivially on $e\Z_p[\Gamma]$, 
		\[
			I_e=\bigcap_{\gamma\in \Gamma} \rho_{\Z_p[\Gamma],e}(( \sum_{i=1}^{|\gamma|} \gamma^i))=p^n \cdot e\Z_p[\Gamma],
		\]
		so $I_e (e\Cl(K)) \subseteq \Gal(L/L^{\ab})$. Then it is enough to show $|\Gal(L/L^{\ab})| \leq |\wedge^2 \Gamma_p|$.

		For every $x\in \Gal(L^{\ab}/Q)$, pick a lift $\hat{x} \in \Gal(L/Q)$. Then
		since $\Gamma$ acts trivially on $\Gal(L/L^{\ab}) \subseteq \Gal(L/K)$, the following is a central extension
		\[
			1 \longrightarrow \Gal(L/L^{\ab}) \longrightarrow \Gal(L/Q) \longrightarrow \Gal(L^{\ab}/Q) \longrightarrow 1.
		\]
		so $\Gal(L/L^{\ab})$ is the subgroup of $\Gal(L/Q)$ generated by $\{[\hat{x}, \hat{y}] \,|\, x,y \in \Gal(L^{\ab}/Q)\}$. If $x\in \ker(\Gal(L^{\ab}/Q) \to \Gal(K/Q))$, then $\hat{x} \in \Gal(L/K)=e\Cl(K)$ is in the center of $\Gal(L/Q)$, so for any $y,z \in \Gal(L^{\ab}/Q)$,
		\[
			[\widehat{xz},\hat{y}]=[\hat{x}\hat{z}, \hat{y}]=[\hat{x},\hat{y}]^{\hat{z}} [\hat{z}, \hat{y}] = [\hat{z}, \hat{y}].
		\]
		So picking a lift $\tilde{\gamma} \in \Gal(L/Q)$ for each $\gamma \in \Gal(K/Q)\simeq \Gamma$, the commutator subgroup $\Gal(L/L^{\ab})$ is generated by $\{[\widetilde{\gamma_1}, \widetilde{\gamma_2}] \,|\, \gamma_1, \gamma_2 \in \Gamma\}$. 
		Finally, since $[\widetilde{\gamma_1}, \widetilde{\gamma_2}]=[\widetilde{\gamma_2}, \widetilde{\gamma_1}]^{-1}$, $[\widetilde{\gamma_1}, \widetilde{\gamma_1}]=1$, and $[\widetilde{\gamma_1}\widetilde{\gamma_2}, \widetilde{\gamma_3}]=[\widetilde{\gamma_1},\widetilde{\gamma_3}]^{\widetilde{\gamma_2}} [\widetilde{\gamma_2},\widetilde{\gamma_3}]=[\widetilde{\gamma_1},\widetilde{\gamma_3}] [\widetilde{\gamma_2},\widetilde{\gamma_3}]$ for any $\gamma_1, \gamma_2, \gamma_3 \in \Gamma$, so there is a quotient map
		\begin{eqnarray*}
			\wedge^2 \Gamma &\longrightarrow& \Gal(L/L^{\ab}) \\
			\gamma_1 \wedge \gamma_2 &\longmapsto& [\widetilde{\gamma_1}, \widetilde{\gamma_2}].
		\end{eqnarray*}
		Because $\Gal(L/L^{\ab})$ is a $p$-group, the above quotient map factors through $\wedge^2 \Gamma_p$, then the proof is completed.
	 \end{proof}
	 
\section{Proof of Theorem~\ref{thm:Z/2Z}}\label{sect:Z/2Z}

	When $\Gamma:=\Z/2\Z$, there is a unique nontrivial primitive idempotent of $\Q_2[\Z/2\Z]$, which is $e:=\frac{1-\sigma}{2}$, where $\sigma$ is the nontrivial element of $\Z/2\Z$. Throughout this section, let $q$ be a power of an odd prime.

\subsection{Properties of $\im W_{q^{-1}}$}
\hfill

	Let $H$ be a finite $e\Z_2[\Z/2\Z]$-module, let $G$ denote $H \rtimes \Z/2\Z$, and let $c$ denote the set of all elements of $G$ that has order 2 and is not contained in $H$. Since $\sigma$ acts on $H$ as taking inverse, we have
	\[
		G^{\ab} = H/2H \times \Z/2\Z;
	\]
	we define $c^{\ab}$ to be all the elements of $G^{\ab}$ whose image under the quotient map $G^{\ab} \to \Z/2\Z$ is nontrivial.
	
	\begin{lemma}\label{lem:c/G-ab}
		The quotient map $G \to G^{\ab}$ induces a bijection between $c/G$ and $c^{\ab}/G^{\ab}$; moreover, it induces a bijection 
		\[
			\ker(\Z^{c/G}_{\equiv q, n, \geq 0} \to G^{\ab} )\overset{\sim}{\longrightarrow} \ker(\Z^{c^{\ab}/G^{\ab}}_{\equiv q, n, \geq 0} \to G^{\ab} )
		\]
	\end{lemma}
	
	\begin{proof}
		We write elements of $G$ as $(a,g)$ for $a \in H$ and $g \in \Z/2\Z$. Then the set $c$ is $\{(a,\sigma) \mid a \in H\}$. For any element $b \in H$, the conjugation of $(a, \sigma)$ by $b$ is 
	$(b^{-1}, 1)(a, \sigma) (b,1) = (a b^{-2}, \sigma)$. So for any $h \in H/2H$, all elements of $G$ whose image in $G^{\ab}$ is $(h,\sigma)$ are all conjugate to each other, which implies the first bijection in the lemma.
	
		Since every element in $c$ and $c^{\ab}$ has order 2, the requirement ``$\equiv q$'' in $\Z_{\equiv q, n , \geq 0}^{c/G}$ and $\Z_{\equiv q, n, \geq 0}^{c^{\ab}/G^{\ab}}$ can be removed without changing the sets. Then the second bijection follows from the first bijection.
	\end{proof}
	Recall the definition of $b(G,c,q,n)$ in \eqref{eq:defofb}, and we need to compare $b(G, c, q, n)$ and $b(G^{\ab}, c^{\ab}, q, n)$ in the proof of Theorem~\ref{thm:Z/2Z}. First, we describe the Schur multiplier $H_2(G, c)$ and a reduced Schur covering map of $G$ and $c$.
	
	\begin{lemma}\label{lem:SchurCover-C2}
		Retain the notation of $G$ and $c$ from above. Write the group $H$ as $\prod_{i=1}^r \Z/2^{d_i}\Z$ with $d_1 \geq d_2 \geq \ldots d_r>0$, and let $x_1,\ldots, x_r$ be a standard basis of $H$ such that $|x_i|=2^{d_i}$. Then the reduced Schur multiplier of $G$ and $c$ is 
		\[
			H_2(G,c) \simeq \prod_{1\leq i < j \leq r} \Z/2^{d_j-1}\Z \simeq \wedge^2 2H.
		\]
		Let $\widetilde{H}$ denote the nilpotency class-2 2-group generated by $\widetilde{x}_1, \ldots, \widetilde{x}_r$ such that there is a surjection 
		\begin{eqnarray*}
			\rho:  \widetilde{H} &\longrightarrow& H \\
			\widetilde{x}_i & \longmapsto& x_i, \quad \forall i,
		\end{eqnarray*}
		with $|\widetilde{x}_i|=|x_i|$ and $\ker \rho \simeq H_2(G,c)$ is generated by $[\widetilde{x}_i, \widetilde{x}_j]$ for all $1 \leq i < j \leq r$. There is a unique $\sigma$-action on $\widetilde{H}$ such that $\sigma(\widetilde{x}_i)=\widetilde{x}_i ^{-1}$. Using this $\sigma$-action, we obtain a semidirect product $\widetilde{H} \rtimes \Z/2\Z$; then 
		\[
			1 \longrightarrow H_2(G,c) \longrightarrow \widetilde{H} \rtimes \Z/2\Z \longrightarrow H \rtimes \Z/2\Z \longrightarrow 1
		\]
		is a reduced Schur covering of $G$ and $c$.
	\end{lemma}
	
	\begin{proof}
		By \cite[\S2 (1)]{Evens}, the Schur multiplier of $G$ has size $\#H_2(G, \Z)=\#H_2(H,\Z)_{\Z/2\Z} \cdot \# H[2]$. Since $\Z/2\Z$ acts trivially on $H_2(H,\Z)$, we have $\#H_2(G,\Z)=2^r \prod_{1 \leq i< j \leq r} 2^{d_j}$. Now, let's describe a Schur covering group of $G$. Let $\widehat{H}$ be the nilpotentcy class-2 2-group generated by $\widehat{x}_1, \ldots, \widehat{x}_r$ such that (1) $|\widehat{x}_i|=x_i^{d_i+1}$ for all $i$, (2) its abelianization is $\widehat{H}^{\ab} \simeq \prod_{i=1}^r \Z/2^{d_i+1}\Z$, and (3) $|[\widehat{x}_i, \widehat{x}_j]|=2^{d_j}$ for all $1\leq i < j \leq r$. There is a unique $\sigma$-action on $\widehat{H}$ such that $\sigma(\widehat{x}_i)=\widehat{x}_i^{-1}$ (note that $\sigma(\widehat{x}_i)=\widehat{x}_i^{-1}$ induces the trivial $\sigma$-action on $[\widehat{H},\widehat{H}]$). Then one can check that $\widehat{H} \rightarrow H, \widehat{x}_i \mapsto x_i$ defines a $\Z/2\Z$-equivariant surjection, and it induces a stem extension $\varrho: \widehat{H} \rtimes \Z/2\Z \to H \rtimes \Z/2\Z$. Since the size of $\ker(\widehat{H}\to H)$ equals the size of $H_2(G,\Z)$, we see that $\varrho$ is a Schur covering for $G$. 
		
		By definition of the reduced Schur covering map, to obtain a reduced Schur covering for $G, c$ from $\varrho$, we need take the quotient of $\widehat{H} \rtimes \Z/2\Z$ by the elements $[\hat{x}, \hat{y}]$ for all $\hat{x}, \hat{y} \in \widehat{H} \rtimes \Z/2\Z$ such that $\varrho(\hat{x}) \in c$ and $[\varrho(\hat{x}), \varrho(\hat{y})]=1$. One can compute this type of commutators and verify the statements in the lemma.
	\end{proof}
	
	In the rest of this section, we will use the notation in the above lemma. By a slight abuse of notation, we denote 
	\begin{equation}\label{def:W}
		W_{q^{-1}}: \ker(\Z^{c/G}_{\equiv q, n, \geq 0} \to G^{\ab}) \longrightarrow  \ker \rho,
	\end{equation}
	i.e., it is the restriction of the homomorphism $W_{q^{-1}}$ defined in \S\ref{ss:Hurwitz} to the subset $\ker(\Z^{c/G}_{\equiv q, n, \geq 0} \to G^{\ab})$ of $\Z^{c/G}$. 
	
	\begin{lemma}\label{lem:W}
		Let $q$ be a power of an odd prime, and $W_{q^{-1}}$ be the map \eqref{def:W}. Then, the map $W_{q^{-1}}$ depends only on $\val_2(q-1)$, i.e., then $\Z_{\equiv q_1, n, \geq 0}^{c/G}=\Z_{\equiv q_2, n, \geq 0}^{c/G}$ and $W_{q_1^{-1}}=W_{q_2^{-1}}$, for $\val_2(q_1-1)=\val_2(q_2-1)$.
		When $n$ is even, the following statements about $\im W_{q^{-1}}$ hold.
		\begin{enumerate}
			\item \label{item:W-1} If $n$ is sufficiently large (for example, when $n\geq 2^r$)$, then \im W_{q^{-1}} = 2^{\val_2(q-1)-1} \ker \rho$.
			\item \label{item:W-2} Let $\varpi_n$ denote the composition of the following surjections
			\[
				\ker(\Z^{c/G}_{\equiv q, n, \geq 0} \to G^{\ab}) \xtwoheadrightarrow{W_{q^{-1}}} \im W_{q^{-1}} \xtwoheadrightarrow{} \im W_{q^{-1}} / 2 \im W_{q^{-1}}.
			\]
			For any $\lambda \in \im W_{q^{-1}}/2\im W_{q^{-1}}$, we have 
			\[
				\lim_{\substack{n \to \infty \\ \text{$n$ is even}}}\frac{\# \varpi_n^{-1}(\lambda)}{\#\ker \varpi}=1.
			\]
		\end{enumerate}
		When $n$ is odd, we have $\ker(\Z_{\equiv q, n, \geq 0}^{c/G} \to G^{\ab})=0$.
	\end{lemma}
	
	\begin{proof}
		Recall the definition of $W_{q^{-1}}$: for each conjugacy class $\gamma \in c/G$, let $x_{\gamma}$ be an element in $\gamma$ and $\widehat{x_{\gamma}}$ be a lift of $x_{\gamma}$ in $\widetilde{H}\rtimes \Z/2\Z$; since all elements in $c$ have order 2, $x_{\gamma}^{1/q}=x_{\gamma}$, so $W_{q^{-1}}$ sends the generator of $\Z^{c/G}$ corresponding to $\gamma$ to $\widehat{x_{\gamma}}^{-1/q} \widehat{x_{\gamma}^{1/q}}=\widehat{x_{\gamma}}^{\frac{q-1}{q}} \in \ker \rho$. Since $\#\ker \rho$ is a power of 2, $q$ is odd, and $\Z_{\equiv q, n, \geq 0}^{c/G}$ are equal for all $q$, we have that $W_{q^{-1}}$ depends only on $\val_2(q-1)$. Moreover, we see that if the statements in the lemma hold for some $q$ with $\val_2(q-1)=v$, then they hold for any $q$ with $\val_2(q-1)\geq v$. So it suffices to prove for the case that $\val_2(q-1)=1$. We assume $\val_2(q-1)=1$ for the rest of the proof.
		
		We use the basis $x_1, \ldots, x_r$ of $H$ and the basis $\widetilde{x}_1, \ldots, \widetilde{x}_r$ of $\widetilde{H}$ defined in Lemma~\ref{lem:SchurCover-C2}. Then by Lemma~\ref{lem:c/G-ab}, for any $\gamma \in c$, we can pick the unique representative of $\gamma$ in the following form:
		\[
			x_{\gamma}= (x_1^{a_1^{\gamma}}x_2^{a_2^{\gamma}}\cdots x_r^{a_r^{\gamma}},\sigma)\in H \rtimes \Z/2\Z, \quad a_i^{\gamma} \in \{0,1\};
		\]
		and we pick the lift as
		$\widehat{x_{\gamma}}= (\widetilde{x}_1^{a_1^{\gamma}}\widetilde{x}_2^{a_2^{\gamma}}\cdots \widetilde{x}_r^{a_r^{\gamma}},\sigma)\in \widetilde{H} \rtimes \Z/2\Z$.
		Then 
		\[
			\widehat{x_{\gamma}}^2=\widetilde{x}_1^{a_1^{\gamma}}\cdots \widetilde{x}_r^{a_r^{\gamma}} \sigma(\widetilde{x}_1^{a_1^{\gamma}}\cdots \widetilde{x}_r^{a_r^{\gamma}}) =\widetilde{x}_1^{a_1^{\gamma}}\cdots \widetilde{x}_r^{a_r^{\gamma}}\widetilde{x}_1^{-a_1^{\gamma}}\cdots \widetilde{x}_r^{-a_r^{\gamma}} \equiv \prod_{1 \leq i<j\leq r} [\widetilde{x}_i, \widetilde{x}_j]^{a_i^{\gamma} a_j^{\gamma}} \mod\, 2\ker \rho,
		\]
		and $W_{q^{-1}}(\underline{m}) \in 2 \ker \rho$ for any $\underline{m} \in 2\Z^{c/G} \subset \Z^{c/G}$.
		So to study $\im W_{q^{-1}}$ modulo $2\ker \rho$, we just need to consider the elements $\underline{m} \in \Z^{c/G}$ such that 1) the coordinate corresponding to each conjugacy class in $c/G$ is 0 or 1, and 2) the sum of all coordinates has the same parity as $n$ and is not greater than $n$.
		
		Note that every element in $c$ has nontrivial image under the projection map $H \rtimes \Z/2\Z \to \Z/2\Z=\langle{\sigma \rangle}$. When $n$ is odd, for any $\underline{m} \in \Z^{c/G}_{\equiv q, n, \geq 0}$, the image of $\underline{m}$ under the composite map $\Z^{c/G}_{\equiv q, n, \geq 0} \to G^{\ab}\to \langle{\sigma}\rangle$ is nontrivial. So $\ker(\Z^{c/G}_{\equiv q, n, \geq 0} \to G^{\ab})=0$.
		
		Next, we prove $\im W_{q^{-1}}=\ker \rho$ for even $n$. When $r=1$, we have $\ker \rho=0$ by Lemma~\ref{lem:SchurCover-C2}, so the lemma obviously holds. Then we assume $r>1$, and we will prove $[\widetilde{x}_i, \widetilde{x_j}] \in \im W_{q^{-1}} \bmod \, 2\ker\rho$ for all $1 \leq i < j \leq r$, and then the statement in \eqref{item:W-1} naturally follows. When $n\geq 4$ is even, considering the vector $\underline{m_{even}}$ such that the coordinates corresponding to $(1, \sigma)$, $(x_i,\sigma)$, $(x_j, \sigma)$, $(x_ix_j, \sigma)$ are 1 and all other coordinates are 0, we see that $W_{q^{-1}}(\underline{m_{even}})\equiv [\widetilde{x}_i, \widetilde{x}_j] \bmod\,  2 \ker \rho$. So the proof of \eqref{item:W-1} is completed. 
		
		When $n$ is even and sufficiently large, there is a surjection 
		\[
			\alpha_n: \ker(\Z_{\equiv q, n , \geq 0}^{c/G} \to G^{\ab})  \longrightarrow T:=\left\{V \subset \F_2^{\oplus r} \times \{1\} \subset \F_2^{\oplus r+1} \, \bigg{|}\,  \sum_{\vec{v} \in V} \vec{v} =0\right\}
		\]
		defined by sending the coordinate corresponding to $\gamma$ to $(a_1^{\gamma}, a_2^{\gamma}, \ldots, a_r^{\gamma},1)$. For every $V \in T$, the size of $\alpha_n^{-1}(V)$ equals the number of $\underline{m}\in \Z^{c/G}$ such that each coordinate is non-negative even and the sum of all coordinates is $n-\#V$. By \cite[Lemma~12.8]{LWZB}, 
		\begin{equation}\label{eq:size-alpha-1}
			\#\alpha_n^{-1}(V) = R(n-\#V)^{2^r-1} + O_r((n-\#V)^{2^r-2})
		\end{equation}
		for some constant $R$ depending on $r$. 
		There is a surjection 
		\begin{eqnarray*}
			\beta: T & \longrightarrow & \ker \rho /2 \ker \rho \\
			V &\longmapsto & \sum_{(a_1^{\gamma}, \ldots, a_r^{\gamma}, 1) \in V} \prod_{1\leq i < j \leq r} [\widetilde{x}_i, \widetilde{x}_j]^{a_i^{\gamma} a_j^{\gamma}} \bmod\, 2\ker \rho.
		\end{eqnarray*}
		Then one can check that 
		\begin{equation}\label{eq:varpi-alpha}
			\varpi_n= \beta \circ \alpha_n.
		\end{equation}
		
		{\bf Claim:} $\#\beta^{-1}(\lambda) = \# \ker \beta$ for every $\lambda \in \ker \rho / 2\ker \rho$.
		
		For any $V_1, V_2 \in T$, we define $V_1+V_2$ to be the union of $V_1\backslash V_2$ and $V_2\backslash V_1$, and one can check that $V_1+V_2 \in T$. So $T$, with this addition ``$+$'' and identity element $\O$, is an elementary abelian $2$-group. We compute
		\begin{eqnarray*}
			\beta(V_1)+\beta(V_2) &=& \left(\sum_{(a_1^{\gamma}, \ldots, a_r^{\gamma}, 1) \in V_1} \prod_{1\leq i < j \leq r} [\widetilde{x}_i, \widetilde{x}_j]^{a_i^{\gamma} a_j^{\gamma}} +\sum_{(a_1^{\gamma}, \ldots, a_r^{\gamma}, 1) \in V_2} \prod_{1\leq i < j \leq r} [\widetilde{x}_i, \widetilde{x}_j]^{a_i^{\gamma} a_j^{\gamma}}\right)\bmod\, 2\ker \rho \\
			&=& \beta(V_1+V_2),
		\end{eqnarray*}
		which implies that $\beta$ is a group homomorphism. So we proved the claim.
		
		Finally, the statement in \eqref{item:W-2} follows because 
		\[
			\lim_{\substack{n \to \infty \\ n \text{ is even}}} \frac{\# \varpi_n^{-1}(\lambda)}{\#\ker \varpi}= \lim_{\substack{n \to \infty \\ n \text{ is even}}} \frac{ \sum\limits_{V \in \beta^{-1}(\lambda)} \# \alpha_n^{-1}(V)}{\sum\limits_{V \in \ker(\beta)} \# \alpha_n^{-1}(V)}=1,
		\]
		where the first equality uses \eqref{eq:varpi-alpha} and the second one uses \eqref{eq:size-alpha-1} and the claim.
	\end{proof}

\subsection{Proof of Theorem~\ref{thm:Z/2Z}}
\hfill

	Let $H$ be an $e\Z_2[\Z/2/\Z]$-module and $q$ is a power of an odd prime. 
	Let $G_1:=H \rtimes \Z/2\Z$ and $\pi_1: G_1 \to \Z/2\Z$ be the quotient map modulo $H$, and let $c_1$ be the set of elements of $G_1$ that have the same order as their image under $\pi_1$. 
	Let $\iota_1: \ker \pi \to H$ be the identity map. Then $\Aut(G_1, \iota_1, \pi_1)$ is one-one corresponding to the splitting of $\pi_1$. So $\#\Aut(G_1, \iota_1, \pi_1)=|H|$, as $\sigma \mapsto (h, \sigma)$ defines a splitting for every $h \in H$. 
	For any positive integer $n$, by Lemma~\ref{lem:Cl-Hur}, we have
	\[
		\sum_{K\in \calA_{\Z/2\Z}^+(q^n, \F_q(t))} \# \Sur(\Cl(K), H) = \frac{\#\Hur^n_{G_1,c_1}(\F_q)}{|H|}.
	\]
	
	Similarly, define $G_2:=G_1^{\ab}$ and $\pi_2:G_2 \to \Z/2\Z$, and then define $\iota_2$, $c_2$ accordingly; we have 
	\[
		\sum_{K\in \calA_{\Z/2\Z}^+(q^n, \F_q(t))} \# \Sur(\Cl(K), H/2H) = \frac{\#\Hur^n_{G_2,c_2}(\F_q)}{|H/2H|}.
	\]
	Applying the Hurwitz-point counting method in \eqref{eq:main-1} (the method established in \cite{LWZB}), we have
	\[
		 \lim_{N\to \infty} \lim_{\substack{q \to\infty \\ \val_2(q-1)=v}}\frac{\sum\limits_{0 \leq n \leq N}\#\Hur^n_{G_1,c_1}(\F_q)}{\sum\limits_{0 \leq n \leq N}\#\Hur^n_{G_2,c_2}(\F_q)} = \lim_{N\to \infty} \lim_{\substack{q \to\infty \\ \val_2(q-1)=v}} \frac{b(G_1,c_1,q, 2\lfloor \frac{N}{2} \rfloor) }{b(G_2,c_2,q, 2\lfloor \frac{N}{2} \rfloor)},
	\]
	here $2\lfloor \frac{N}{2} \rfloor$ is the largest even number $\leq N$, which is the largest integer $n\leq N$ such that $b(G_i, c_i, q, n)>0$ by Lemma~\ref{lem:W}. Also, by Lemma~\ref{lem:W} and the definition \eqref{eq:defofb}, letting $\rho_i$ be the map $\rho$ there for $G:=G_i$ and $c:=c_i$, we have 
	\begin{eqnarray*}
		b(G_i,c_i, q, 2\lfloor \frac{N}{2} \rfloor) &=& \#\ker\rho_i[2^{\val_2(q-1)}] \cdot \frac{\# 2^{\val_2(q-1)} \ker \rho_i}{\# 2^{\val_2(q-1)-1} \ker \rho_i} \cdot \#\ker(\Z^{c_i/G_i}_{\equiv q, 2\lfloor N/2 \rfloor, \geq 0} \to G_i^{\ab}) \\
		&=&\#\ker\rho_i[2^{v}] \cdot \frac{\#\ker\rho_i[2^{v-1}]}{\#\ker\rho_i[2^{v}]} \cdot \#\ker(\Z^{c_i/G_i}_{\equiv q, 2\lfloor N/2 \rfloor, \geq 0} \to G_i^{\ab}) \\
		&=& \#\ker\rho_i[2^{v-1}] \cdot \#\ker(\Z^{c_i/G_i}_{\equiv q, 2\lfloor N/2 \rfloor, \geq 0} \to G_i^{\ab}).
	\end{eqnarray*}
	By Lemma~\ref{lem:SchurCover-C2}, $\#\ker \rho_1 [2^{v-1}] = \# (\wedge^2 2H)[2^{v-1}]$ and $\#\ker \rho_2 [2^{v-1}]=1$.
	Then Theorem~\ref{thm:Z/2Z} follows by Proposition~\ref{prop:weight-formula} and Lemma~\ref{lem:c/G-ab}.

\section{Conjectures for Moment and Probability}\label{sect:conj}

	Let $e$ be a primitive central idempotent of $\Q_p[\Gamma]$ and $\calP_{e\Z_p[\Gamma]}$ denote the set of isomorphism classes of finite $e\Z_p[\Gamma]$-modules. Define a topology on $\calP_{e\Z_p[\Gamma]}$ in which the basic opens are the sets 
	\[
		U_{M,I}:=\{X \in \calP_{e\Z_p[\Gamma]} \mid X \otimes_{e\Z_p[\Gamma]} e\Z_p[\Gamma]/I \simeq M\}
	\]
	for each $M \in \calP_{e\Z_p[\Gamma]}$ and $I$ an nonzero ideal of $e\Z_p[\Gamma]$. Applying the result of Sawin and Wood \cite{Sawin-Wood-category}, we show in Proposition~\ref{prop:prob} that there exists a unique probability measure on $\calP_{e\Z_p[\Gamma]}$ such that the $M$-moment, i.e., the average size of $\Sur_{e\Z_p[\Gamma]}(-,M)$, is $1/|M|$ for every finite $e\Z_p[\Gamma]$-module $M$. Finally, in Conjecture~\ref{conj:GG}, we give the conjecture for the distribution of $I_e \cdot e\Cl(K)$ as $K$ varies over totally real $\Gamma$-extensions of $Q=\Q$ or $\F_q(t)$ ordered by $\rDisc$.
	
	\begin{proposition}\label{prop:prob}
		There is a unique probability measure $\mu_{e\Z_p[\Gamma]}$ on $\calP_{e\Z_p[\Gamma]}$ such that 
		\[
			\int_{X\in \calP_{e\Z_p[\Gamma]}} \#\Sur_{e\Z_p[\Gamma]}(X, M) d\mu_{e\Z_p[\Gamma]} = \frac{1}{|M|}.
		\]
		Denote $A:=e\Z_p[\Gamma]/\frakm_e$ (recall that $\frakm_e$ is the maximal ideal of $e\Z_p[\Gamma]$). The formula of $\mu_{e\Z_p[\Gamma]}$ is given by
		\[
			\mu_{e\Z_p[\Gamma]} (M) =   \frac{1}{|\Aut_{e\Z_p[\Gamma]} (M)| |M|} \prod_{i=2}^{\infty} (1-|A|^{-i}).
		\]
	\end{proposition}
	
	\begin{proof}
		For every positive integer $n$, denote $U_{M,n}:=U_{M, \frakm_A^n}$.
		For a given $M \in \calP_{e\Z_p[\Gamma]}$, there is a maximal integer $m$ such that $\frakm_e^n M=0$ for every $n\geq m$. In particular, for every $n > m$, the basic open $U_{M,n}=\{M\}$.
		Then, by \cite[Theorem~1.2 and Lemma~6.3]{Sawin-Wood-category}, the proposition holds for $\mu_{e\Z_p[\Gamma]}(M)=v_{\calC_S,M}$ with $S:=e\Z_p[\Gamma]/\frakm_e^n$ and $n > m$.
		So, it suffices to show that the formula for $v_{\calC_S, M}$ given in \cite[Lemma~6.3]{Sawin-Wood-category} equals the one for $\mu_{e\Z_p[\Gamma]}$ in the proposition.
		
		Recall that $e\Z_p[\Gamma]$ is a discrete valuation ring and $A$ is the unique finite simple $e\Z_p[\Gamma]$-module. One can write
		\[
			 M\simeq \bigoplus_{j=1}^m (e\Z_p[\Gamma]/\frakm_e^j)^{\oplus d_j},
		\]
		for some $d_j \in \Z_{\geq 0}$. Then 
		\[
			\Ext^1_S(M, A)\simeq \bigoplus_{j=1}^m \Ext^1_S(e\Z_p[\Gamma]/\frakm_e^j, A)^{\oplus d_j}\simeq \bigoplus_{j=1}^m A^{\oplus d_j} \simeq \Hom_S(M,A).
		\]
		Also, by Lemma~\ref{lem:End=A}, $|\End_{e\Z_p[\Gamma]}(A)|=|A|$. So we see that 
		\[
			v_{\calC_S, M}=\frac{1}{|\Aut_{e\Z_p[\Gamma]} (M)| |M|} \prod_{i=1}^{\infty} (1-\frac{1}{|A|}|A|^{-i}),	
		\]
		where $v_{\calC_S, M}$ is defined in \cite[Lemma~6.3]{Sawin-Wood-category}. Then the proposition follows since $\mu_{e\Z_p[\Gamma]}(M)=v_{\calC_S, M}$.
	\end{proof}

	\begin{conjecture}\label{conj:GG}
		Let $\Gamma$ be a finite abelian group, $p$ a prime number, and $Q$ be either $\Q$ or $\F_q(t)$ for $q$ such that $\gcd(p|\Gamma|,q)=\gcd(p,q-1)=1$. Let $\calA^+_{\Gamma}(D, Q)$ be the set of isomorphism classes of totally real $\Gamma$-extensions of $Q$ with $\rDisc K=D$. 
		Let $e$ be a nontrivial primitive idempotent of $e\Q_p[\Gamma]$, $A:=e\Z_p[\Gamma]/\frakm_e$, and $M$ a finite $e\Z_p[\Gamma]$-module.
		Then 
		\[
			\lim_{B \to \infty}  \frac{\sum\limits_{D\leq B} \#\{K \in \calA^+_{\Gamma}(D, Q) \mid I_e \cdot e\Cl(K) \simeq M\}}{\sum\limits_{D \leq B} \#\calA^+_{\Gamma}(D,Q)}=\frac{1}{|\Aut_{\Gamma} (M)| |M|} \prod_{i=2}^{\infty} (1-|A|^{-i})
		\]
		and 
		\[
			\lim_{B\to \infty} \frac{\sum\limits_{D\leq B}\sum\limits_{K \in \calA^+_{\Gamma}(D,Q)} \# \Sur_{\Gamma} (I_e\cdot e\Cl(K), M)}{\sum\limits_{D \leq B} \#\calA^+_{\Gamma}(D,Q)} = \frac{1}{|M|}.
		\]
	\end{conjecture}

\vspace{.2 in}
\appendix
\section{The average number of primes satisfying given ramification types}\label{ss:appendix}

	\begin{center}
		by Peter Koymans
	\end{center}

Fix a number field $Q$, a finite abelian group $\Gamma$ and a nontrivial cyclic subgroup $\Gamma_0$ of $\Gamma$. Recall that a $\Gamma$-extension is by definition a surjective homomorphism $\varphi: G_Q \twoheadrightarrow \Gamma$. We define $Q(\varphi)$ to be the extension of $Q$ corresponding to $\varphi$ (so $\Gal(Q(\varphi)/Q) \cong \Gamma$), and we define $\mathrm{rDisc}(\varphi)$ to be the absolute norm of the radical of the discriminant ideal $\mathrm{Disc}(Q(\varphi)/Q)$. Define
\[
\omega(\varphi) = \{\mathfrak{p} \subset Q  \mid \varphi(\calT_{\frakp})=\varphi(\calG_{\frakp}) = \Gamma_0\}.
\]
Recall that the definition of $\calT_{\frakp}$ and $\calG_{\frakp}$ (defined in Section~\ref{ss:notation}) depends on an implicitly chosen embedding $\iota_\mathfrak{p}: \overline{Q} \rightarrow \overline{Q_\mathfrak{p}}$. We stress that $\varphi(\calT_{\frakp})$ and $\varphi(\calG_{\frakp})$ do not depend on the choice of embedding $\iota_\mathfrak{p}$, since a different embedding yields conjugate subgroups of $\calG_{\frakp}$ and $\calT_{\frakp}$ (in $G_Q$) and $\Gamma$ is abelian. However, it is possible that $\omega(\varphi) \neq \omega(\varphi')$ even when $Q(\varphi) = Q(\varphi')$.

Next, we fix a finite set $\mathcal{Z}$ of primes of $Q$ and for each $\mathfrak{p} \in \mathcal{Z}$ a continuous homomorphism $\varphi_\mathfrak{p}: G_{Q_\mathfrak{p}} \rightarrow \Gamma$. Then we define
$$
\mathcal{A}(X, (\varphi_\mathfrak{p})_{\mathfrak{p} \in \mathcal{Z}}) := \{\varphi: G_Q \twoheadrightarrow A : \mathrm{rDisc}(\varphi) \leq X, \varphi \circ \iota_\mathfrak{p}^\ast = \varphi_{\mathfrak{p}}\}
$$
Our goal is to show the following result.

\begin{theorem}
\label{tCond}
Let $Q$, $\Gamma$, $\Gamma_0$, $\mathcal{Z}$ and $(\varphi_\mathfrak{p})_{\mathfrak{p} \in \mathcal{Z}}$ be as above. Assume that $\mathcal{A}(X, (\varphi_\mathfrak{p})_{\mathfrak{p} \in \mathcal{Z}})$ is not empty for $X$ sufficiently large. Then we have
\begin{align}
\label{eMainResultCond}
\lim_{X \rightarrow \infty} \frac{\sum\limits_{\varphi \in \mathcal{A}(X, (\varphi_\mathfrak{p})_{\mathfrak{p} \in \mathcal{Z}})} \omega(\varphi)}{\sum\limits_{\varphi \in \mathcal{A}(X, (\varphi_\mathfrak{p})_{\mathfrak{p} \in \mathcal{Z}})} 1} = \infty.
\end{align}
\end{theorem}

We emphasize that Wood \cite{Wood2010} has already shown an asymptotic formula for the denominator, so it suffices to give a lower bound for the numerator in equation (\ref{eMainResultCond}). In fact it should be possible to obtain an asymptotic formula for the numerator. However, this would require a substantial amount of work, and would be besides the point of this appendix.

We will assume familiarity with the results of Wood \cite{Wood2010} throughout our proof.

\begin{proof}
Let $B > 0$ be a large number. Fix a finite collection of primes $T$ of $Q$ such that
\begin{itemize}
\item we have
\[
\sum_{\mathfrak{q} \in T} \frac{1}{N_{Q/\Q}(\mathfrak{q})} > B;
\]
\item we have $N_{Q/\Q}(\mathfrak{q}) \equiv 1 \bmod |\Gamma|$ for every $\mathfrak{q} \in T$;
\item we have that $T$ is disjoint from $\mathcal{Z}$ and all primes dividing $|\Gamma|$.
\end{itemize}
Such a collection $T$ exists. Indeed, this is a consequence of the Chebotarev density theorem applied to $Q(\zeta_{|\Gamma|})/Q$ and an application of partial summation. 

Our goal is to apply \cite[Theorem 2.1]{Wood2010}. We take $\mathrm{rDisc}(\varphi)$ as our counting function; for the definition of a counting function, see \cite[Section 2.1]{Wood2010}. This counting function is readily verified to be fair. For each prime $\mathfrak{q} \in T$, there exists a homomorphism $\varphi_\mathfrak{q}: G_{Q_\mathfrak{q}} \rightarrow \Gamma$ with the following two properties
\begin{itemize}
\item we have $\text{im}(\varphi_\mathfrak{q}) = \Gamma_0$;
\item we have that the fixed field of $\varphi_\mathfrak{q}$ is a totally tamely ramified extension of $Q_\mathfrak{q}$ (that is, the image of the inertia subgroup of $G_{Q_\frakq}$ under $\varphi_\frakq$ is also $\Gamma_0$).
\end{itemize}
Fix such a choice $\varphi_\mathfrak{q}: G_{Q_\mathfrak{q}} \rightarrow \Gamma$ for each $\mathfrak{q} \in T$. Each $\varphi_\mathfrak{q}$ corresponds to a $\Gamma$-structured $G_{Q_\mathfrak{q}}$-algebra by \cite[Lemma 2.5]{Wood2010}. Then we define a local specification $\Sigma_\mathfrak{q}$ by taking the $\Gamma$-structured $G_{Q_\mathfrak{q}}$-algebra corresponding to $\varphi_\mathfrak{q}$. For the definition of local specification, see \cite[p.~4]{Wood2010}. We similarly find local specifications $\Sigma_\mathfrak{p}$ for $\mathfrak{p} \in \mathcal{Z}$. We now apply \cite[Theorem 2.1]{Wood2010} for every $\mathfrak{q} \in T$ with $S = \{\mathfrak{q}\} \cup \mathcal{Z}$ and local specifications $\Sigma = (\Sigma_\mathfrak{p})_{\mathfrak{p} \in S}$. This yields
\[
\lim_{X \rightarrow \infty} \frac{\sum\limits_{\varphi \in \mathcal{A}(X, (\varphi_\mathfrak{p})_{\mathfrak{p} \in \mathcal{Z}})} \mathbf{1}_{\varphi \circ \iota_\mathfrak{q}^\ast = \varphi_\mathfrak{q}}}{\sum\limits_{\varphi \in \mathcal{A}(X, (\varphi_\mathfrak{p})_{\mathfrak{p} \in \mathcal{Z}})} 1} = \frac{1}{N_{Q/\Q}(\mathfrak{q})} \cdot \frac{1}{|\Gamma| + (|\Gamma|^2 - |\Gamma|)/N_{Q/\Q}(\mathfrak{q})} > \frac{1}{N_{Q/\Q}(\mathfrak{q})} \cdot \frac{1}{|\Gamma|^2}
\]
Here we have computed the local probabilities in \cite[Theorem 2.1]{Wood2010} as follows. Let $M$ be the maximal abelian extension of $Q_\mathfrak{q}$ such that $\Gal(M/Q_\mathfrak{q})$ is killed by $|\Gamma|$. Since $N_{Q/\Q}(\mathfrak{q}) \equiv 1 \bmod |\Gamma|$, we see that $M$ equals the compositum of the unramified extension of degree $|\Gamma|$ and some totally tamely ramified extension of degree $|\Gamma|$. In particular, we deduce that $\Gal(M/Q_\mathfrak{q}) \cong (\Z/|\Gamma|\Z)^2$. Therefore the set of homomorphisms $\varphi_\mathfrak{q}: G_{Q_\mathfrak{q}} \rightarrow \Gamma$ are in bijection with homomorphisms $(\Z/|\Gamma|\Z)^2 \rightarrow \Gamma$. There are $|\Gamma|^2$ such maps, of which $|\Gamma|$ are unramified. The unramified ones have radical discriminant $1$, while the remaining ones have radical discriminant $N_{Q/\Q}(\mathfrak{q})$. From this we compute the local probabilities in \cite[Theorem 2.1]{Wood2010}.

Since our set $T$ is finite, we deduce that
\[
\frac{\sum\limits_{\varphi \in \mathcal{A}(X, (\varphi_\mathfrak{p})_{\mathfrak{p} \in \mathcal{Z}})} \mathbf{1}_{\varphi \circ \iota_\mathfrak{q}^\ast = \varphi_\mathfrak{q}}}{\sum\limits_{\varphi \in \mathcal{A}(X, (\varphi_\mathfrak{p})_{\mathfrak{p} \in \mathcal{Z}})} 1} > \frac{1}{N_{Q/\Q}(\mathfrak{q})} \cdot \frac{1}{|\Gamma|^2}
\]
for every $\mathfrak{q} \in T$, provided that we take $X$ sufficiently large. Furthermore, because
\[
\omega(\varphi) \geq \sum_{\mathfrak{q} \in T} \mathbf{1}_{\varphi \circ \iota_\mathfrak{q}^\ast = \varphi_\mathfrak{q}},
\]
we get
\[
\frac{\sum\limits_{\varphi \in \mathcal{A}(X, (\varphi_\mathfrak{p})_{\mathfrak{p} \in \mathcal{Z}})} \omega(\varphi)}{\sum\limits_{\varphi \in \mathcal{A}(X, (\varphi_\mathfrak{p})_{\mathfrak{p} \in \mathcal{Z}})} 1} > \frac{B}{|\Gamma|^2}
\]
for $X$ sufficiently large. Since $B$ was arbitrary and $\Gamma$ is fixed, the theorem follows.
\end{proof}

Retain the notation above. We write $D(\varphi)$ for the absolute norm of the relative discriminant of $Q(\varphi)/Q$.

\begin{theorem}\label{thm:A.2}
Let $\ell$ be the smallest prime divisor of $|\Gamma|$. If $\Gamma_0\simeq \F_{\ell}$, then
\begin{align}
\label{eMainResultDisc}
\lim_{X \rightarrow \infty} \frac{\sum\limits_{\varphi: G_Q \twoheadrightarrow \Gamma, D(\varphi) \leq X} \omega(\varphi)}{\sum\limits_{\varphi: G_Q \twoheadrightarrow \Gamma, D(\varphi) \leq X} 1} = \infty.
\end{align}
\end{theorem}

A classical result of Wright \cite{Wright} gives an asymptotic formula for the denominator of equation (\ref{eMainResultDisc}), so we shall restrict our attention to the numerator. Note that the result of Wright \cite{Wright} does not allow for local conditions, so we have also omitted local conditions in our result.

\begin{proof}
There certainly exists a surjective homomorphism $\widetilde{\varphi}: G_Q \rightarrow \Gamma$. Fix such a choice $\widetilde{\varphi}$. Let $B > 0$ be a large number. We again fix a finite collection of primes $T$ of $Q$ such that
\begin{itemize}
\item we have
\[
\sum_{\mathfrak{q} \in T} \frac{1}{N_{Q/\Q}(\mathfrak{q})} > B;
\]
\item we have $N_{Q/\Q}(\mathfrak{q}) \equiv 1 \bmod |\Gamma|$ for every $\mathfrak{q} \in T$;
\item we have that $\mathfrak{q}$ splits in $Q(\widetilde{\varphi})$ for every $\mathfrak{q} \in T$.
\end{itemize}
We apply Wood's result \cite{Wood2010} to $\Gamma[\ell]$. We also take the same local specifications $\varphi_\mathfrak{q}$ for $\mathfrak{q} \in T$ as in the proof of Theorem~\ref{tCond}. Then we obtain
\begin{align}
\label{eLower}
\frac{\sum\limits_{\chi: G_Q \twoheadrightarrow \Gamma[\ell], \ \mathrm{rDisc}(\chi) \leq X} \mathbf{1}_{\chi \circ \iota_\mathfrak{q}^\ast = \varphi_\mathfrak{q}}}{\sum\limits_{\chi: G_Q \twoheadrightarrow \Gamma[\ell], \ \mathrm{rDisc}(\chi) \leq X} 1} > \frac{1}{N_{Q/\Q}(\mathfrak{q})} \cdot \frac{1}{|\Gamma[\ell]|^2}
\end{align}
for all $\mathfrak{q} \in T$ just like before. We have an asymptotic formula for the denominator by Wood \cite{Wood2010}. It then follows from the work of Wright \cite{Wright} that
\begin{align}
\label{eTwisting}
\sum_{\chi: G_Q \twoheadrightarrow \Gamma[\ell], \ \mathrm{rDisc}(\chi) \leq X} 1 \asymp \sum_{\varphi: G_Q \twoheadrightarrow \Gamma, \ D(\varphi) \leq X^{|\Gamma| \cdot \frac{\ell - 1}{\ell}}} 1.
\end{align}
We will now give a lower bound for
\[
\sum\limits_{\substack{\varphi: G_Q \twoheadrightarrow \Gamma \\ D(\varphi) \leq X^{|\Gamma| \cdot \frac{\ell - 1}{\ell}}}} \omega(\varphi) \geq \sum_{\mathfrak{q} \in T} \sum\limits_{\substack{\varphi: G_Q \twoheadrightarrow \Gamma \\ D(\varphi) \leq X^{|\Gamma| \cdot \frac{\ell - 1}{\ell}}}} \mathbf{1}_{\varphi \circ \iota_\mathfrak{q}^\ast = \varphi_\mathfrak{q}}.
\]
To this end, consider those $\varphi$ of the shape $\widetilde{\varphi} + \chi$ with $\chi$ satisfying $\chi \circ \iota_\mathfrak{q}^\ast = \varphi_\mathfrak{q}$. Then we observe that 
\[
(\widetilde{\varphi} + \chi) \circ \iota_\mathfrak{q}^\ast = \varphi_\mathfrak{q}
\]
by construction of $T$. Indeed, recall that $\mathfrak{q}$ splits completely in $Q(\widetilde{\varphi})$.

Furthermore, there exists a constant $C > 0$, depending only on our choice of $\widetilde{\varphi}$, such that
\[
D(\widetilde{\varphi} + \chi) \leq C \cdot \mathrm{rDisc}(\chi)^{|\Gamma| \cdot \frac{\ell - 1}{\ell}}.
\]
Combining this with equations (\ref{eLower}) and (\ref{eTwisting}) we get the theorem.
\end{proof}

\begin{remark}\label{rmk:disc}
The condition on $\Gamma$ in the theorem is necessary for the limit to be infinite. Indeed, consider for example $\Gamma = \Z/6\Z$ and $\Gamma_0 = \Z/3\Z$ and $Q=\Q$. Then the reason for this phenomenon is essentially that
\[
\sum_{\substack{a^3b^4 \leq X}} 1 \asymp \sum_{\substack{a^3b^4 \leq X}} \sum_{\ell \mid b} 1.
\]
\end{remark}

% \bib, bibdiv, biblist are defined by the amsrefs package.


\begin{bibdiv}
\begin{biblist}

\bib{Benson}{book}{
      author={Benson, D.~J.},
       title={Representations and cohomology. {I}},
     edition={Second},
      series={Cambridge Studies in Advanced Mathematics},
   publisher={Cambridge University Press, Cambridge},
        date={1998},
      volume={30},
        ISBN={0-521-63653-1},
        note={Basic representation theory of finite groups and associative
  algebras},
      review={\MR{1644252}},
}

\bib{Boston-Wood}{article}{
      author={Boston, Nigel},
      author={Wood, Melanie~Matchett},
       title={Non-abelian {C}ohen-{L}enstra heuristics over function fields},
        date={2017},
        ISSN={0010-437X},
     journal={Compos. Math.},
      volume={153},
      number={7},
       pages={1372\ndash 1390},
  url={https://doi-org.ezproxy.library.wisc.edu/10.1112/S0010437X17007102},
      review={\MR{3705261}},
}

\bib{Cohen-Lenstra}{incollection}{
      author={Cohen, H.},
      author={Lenstra, H.~W., Jr.},
       title={Heuristics on class groups of number fields},
        date={1984},
   booktitle={Number theory, {N}oordwijkerhout 1983 ({N}oordwijkerhout, 1983)},
      series={Lecture Notes in Math.},
      volume={1068},
   publisher={Springer, Berlin},
       pages={33\ndash 62},
         url={https://doi-org.ezproxy.library.wisc.edu/10.1007/BFb0099440},
      review={\MR{756082}},
}

\bib{Cohen-Martinet}{article}{
      author={Cohen, H.},
      author={Martinet, J.},
       title={Class groups of number fields: numerical heuristics},
        date={1987},
        ISSN={0025-5718, 1088-6842},
     journal={Mathematics of Computation},
      volume={48},
      number={177},
       pages={123\ndash 137},
         url={http://www.ams.org/mcom/1987-48-177/S0025-5718-1987-0866103-4/},
}

\bib{Evens}{article}{
      author={Evens, Leonard},
       title={The {S}chur multiplier of a semi-direct product},
        date={1972},
        ISSN={0019-2082},
     journal={Illinois J. Math.},
      volume={16},
       pages={166\ndash 181},
         url={http://projecteuclid.org/euclid.ijm/1256052393},
      review={\MR{417310}},
}

\bib{EVW}{article}{
      author={Ellenberg, Jordan~S.},
      author={Venkatesh, Akshay},
      author={Westerland, Craig},
       title={Homological stability for {H}urwitz spaces and the
  {C}ohen-{L}enstra conjecture over function fields},
        date={2016},
        ISSN={0003-486X,1939-8980},
     journal={Ann. of Math. (2)},
      volume={183},
      number={3},
       pages={729\ndash 786},
         url={https://doi.org/10.4007/annals.2016.183.3.1},
      review={\MR{3488737}},
}

\bib{Gerth84}{article}{
      author={Gerth, Frank, III},
       title={The {$4$}-class ranks of quadratic fields},
        date={1984},
        ISSN={0020-9910},
     journal={Invent. Math.},
      volume={77},
      number={3},
       pages={489\ndash 515},
         url={https://doi-org.proxy2.library.illinois.edu/10.1007/BF01388835},
      review={\MR{759260}},
}

\bib{Gerth86}{article}{
      author={Gerth, Frank, III},
       title={Densities for certain {$l$}-ranks in cyclic fields of degree
  {$l^n$}},
        date={1986},
        ISSN={0010-437X},
     journal={Compositio Math.},
      volume={60},
      number={3},
       pages={295\ndash 322},
  url={http://www.numdam.org.proxy2.library.illinois.edu/item?id=CM_1986__60_3_295_0},
      review={\MR{869105}},
}

\bib{Iwasawa}{article}{
      author={Iwasawa, Kenkichi},
       title={On {G}alois groups of local fields},
        date={1955},
        ISSN={0002-9947},
     journal={Trans. Amer. Math. Soc.},
      volume={80},
       pages={448\ndash 469},
         url={https://doi-org.proxy2.library.illinois.edu/10.2307/1992998},
      review={\MR{75239}},
}

\bib{Koymans-Pagano}{article}{
      author={Koymans, Peter},
      author={Pagano, Carlo},
       title={A sharp upper bound for the 2-torsion of class groups of
  multiquadratic fields},
        date={2022},
        ISSN={0025-5793},
     journal={Mathematika},
      volume={68},
      number={1},
       pages={237\ndash 258},
         url={https://doi-org.proxy2.library.illinois.edu/10.1112/mtk.12123},
      review={\MR{4405977}},
}

\bib{Liu-presentation}{article}{
      author={Liu, Yuan},
       title={Presentations of {G}alois groups of maximal extensions with
  restricted ramification},
        date={2020},
        note={Algebra \& Number Theory, arXiv:2005.07329},
}

\bib{Liu2022b}{article}{
      author={Liu, Yuan},
       title={On the {$p$}-rank of class groups of {$p$}-extensions},
        date={2024},
        ISSN={1073-7928,1687-0247},
     journal={Int. Math. Res. Not. IMRN},
      number={6},
       pages={5274\ndash 5325},
         url={https://doi.org/10.1093/imrn/rnad234},
      review={\MR{4721054}},
}

\bib{Liu-Wood}{article}{
      author={Liu, Yuan},
      author={Wood, Melanie~Matchett},
       title={The free group on {$n$} generators modulo {$n+u$} random
  relations as {$n$} goes to infinity},
        date={2020},
        ISSN={0075-4102},
     journal={J. Reine Angew. Math.},
      volume={762},
       pages={123\ndash 166},
         url={https://doi.org/10.1515/crelle-2018-0025},
      review={\MR{4195658}},
}

\bib{LWZB}{article}{
      author={Liu, Yuan},
      author={Wood, Melanie~Matchett},
      author={Zureick-Brown, David},
       title={A predicted distribution for {G}alois groups of maximal
  unramified extensions},
        date={2024},
        ISSN={0020-9910,1432-1297},
     journal={Invent. Math.},
      volume={237},
      number={1},
       pages={49\ndash 116},
         url={https://doi.org/10.1007/s00222-024-01257-1},
      review={\MR{4756988}},
}

\bib{NSW}{book}{
      author={Neukirch, J\"{u}rgen},
      author={Schmidt, Alexander},
      author={Wingberg, Kay},
       title={Cohomology of number fields},
     edition={Second},
      series={Grundlehren der mathematischen Wissenschaften [Fundamental
  Principles of Mathematical Sciences]},
   publisher={Springer-Verlag, Berlin},
        date={2008},
      volume={323},
        ISBN={978-3-540-37888-4},
         url={https://doi.org/10.1007/978-3-540-37889-1},
      review={\MR{2392026}},
}

\bib{Serre}{book}{
      author={Serre, Jean-Pierre},
       title={Linear representations of finite groups},
     edition={French},
      series={Graduate Texts in Mathematics},
   publisher={Springer-Verlag, New York-Heidelberg},
        date={1977},
      volume={Vol. 42},
        ISBN={0-387-90190-6},
      review={\MR{450380}},
}

\bib{Smith2022}{article}{
      author={Smith, Alexander},
       title={{The distribution of $\ell^{\infty}$-Selmer groups in degree
  $\ell$ twist families}},
        date={2022},
        note={preprint, arXiv:2207.05674},
}

\bib{stacks-project}{misc}{
      author={{Stacks Project Authors}, The},
       title={\textit{Stacks Project}},
         how={\url{https://stacks.math.columbia.edu}},
        date={2018},
}

\bib{Sawin-Wood-category}{article}{
      author={Sawin, Will},
      author={Wood, Melanie~Matchett},
       title={{The moment problem for random objects in a category}},
        date={2022},
        note={preprint, arXiv:2210.06279},
}

\bib{Wood2010}{article}{
      author={Wood, Melanie~Matchett},
       title={On the probabilities of local behaviors in abelian field
  extensions},
        date={2010},
        ISSN={0010-437X,1570-5846},
     journal={Compos. Math.},
      volume={146},
      number={1},
       pages={102\ndash 128},
         url={https://doi.org/10.1112/S0010437X0900431X},
      review={\MR{2581243}},
}

\bib{Wood-lifting}{article}{
      author={Wood, Melanie~Matchett},
       title={An algebraic lifting invariant of {E}llenberg, {V}enkatesh, and
  {W}esterland},
        date={2021},
        ISSN={2522-0144},
     journal={Res. Math. Sci.},
      volume={8},
      number={2},
       pages={Paper No. 21, 13},
         url={https://doi.org/10.1007/s40687-021-00259-2},
      review={\MR{4240808}},
}

\bib{Wright}{article}{
      author={Wright, David~J.},
       title={Distribution of discriminants of abelian extensions},
        date={1989},
        ISSN={0024-6115,1460-244X},
     journal={Proc. London Math. Soc. (3)},
      volume={58},
      number={1},
       pages={17\ndash 50},
         url={https://doi.org/10.1112/plms/s3-58.1.17},
      review={\MR{969545}},
}

\bib{Wang-Wood}{article}{
      author={Wang, Weitong},
      author={Wood, Melanie~Matchett},
       title={Moments and interpretations of the {C}ohen-{L}enstra-{M}artinet
  heuristics},
        date={2021},
        ISSN={0010-2571,1420-8946},
     journal={Comment. Math. Helv.},
      volume={96},
      number={2},
       pages={339\ndash 387},
         url={https://doi.org/10.4171/cmh/514},
      review={\MR{4277275}},
}

\end{biblist}
\end{bibdiv}
\end{document}